\documentclass[11pt]{article}
\usepackage{amsmath}
\usepackage{amsfonts}
\usepackage{amssymb}
\usepackage{eufrak}
\usepackage{srcltx}
\usepackage{bbm}
\usepackage{xcolor}
\usepackage{float}
\usepackage{url}
\usepackage{fullpage}
\usepackage{multirow}
\usepackage{graphicx}

\usepackage{accents}
\newcommand{\ubar}[1]{\underaccent{\bar}{#1}}

\def\prob {{\sf Pr}}

\newcommand{\cA}{{\mathfrak A}}
\newcommand{\cS}{{\mathfrak S}}

\newcommand{\cT}{{\mathfrak T}}

\newcommand{\cM}{{\mathfrak M}}

\newcommand{\cV}{{\mathfrak V}}

\newcommand{\cW}{{\mathfrak W}}
\newcommand{\cQ}{{\mathfrak Q}}

\newcommand{\cm}{{\mathfrak m}}

\newtheorem{theorem}{Theorem}[section]
\newtheorem{lemma}{Lemma}[section]

\newtheorem{example}{Example}
\newtheorem{rem}{Remark}[section]

\newtheorem{cor}{Corollary}[section]

\newcommand{\N}{{\cal N}}

\newcommand{\T}{{\cal T}}

\def\Min{\mathop{\rm Min}}

\newcommand{\var}{{\rm Var}}
\newcommand{\cov}{{\rm Cov}}

\newcommand{\V}{{\cal V}}
\newcommand{\Z}{{\cal Z}}
\newcommand{\F}{{\cal F}}

\newcommand{\X}{{\cal X}}

\newcommand{\LL}{{\cal L}}
\newcommand{\Y}{{\cal Y}}

\newcommand{\hF}{\hat{F}}

\newcommand{\R}{{\cal R}}

\def\bcswitch{\left\{\renewcommand{\arraystretch}{1.2}\begin{array}{c@{,~}c}}

\def\ecswitch{\end{array}\right.}

\def\bcswitchs{\left\{\renewcommand{\arraystretch}{1.2}\begin{array}{c}}

\def\ecswitchs{\end{array}\right.}

\newcommand{\fb}{\rule[-2pt]{4pt}{8pt}}

\newcommand{\cdis}{\stackrel{{\cal D}}{\to}}

\newcommand{\bb}{\mbox{\tiny\boldmath$b$}}

\newcommand{\lan}{\langle}
\newcommand{\ran}{\rangle}

\def\w{\omega}

\def\O{\Omega}
\def\vv{\vartheta}

\def\avr{{\sf AVaR}}

\def\hF{\widehat{F}}

\def\bbr{{\Bbb{R}}} 
\def\bbe{{\Bbb{E}}} 

\def\bbp{{\Bbb{P}}}

\def\bby{{\Bbb{Y}}}

\def\bbt{{\Bbb{T}}}

\def\bb{{\Bbb{B}}}

\newcommand{\ind}{{\bf 1}} 

\newtheorem{remark}{Remark}

\begin{document}


\title{\bf Statistical inference and hypotheses testing of risk averse stochastic programs}

\author{
{\bf Vincent  Guigues}\thanks{Research of this author  was partly supported by an FGV grant, CNPq grant 307287/2013-0,
FAPERJ grants E-26/110.313/2014 and E-26/201.599/2014.}\\
School  of Applied Mathematics\\
Funda\c{c}\~ao Getulio Vargas\\
190 Praia de Botafogo, Rio de Janeiro, Brazil,\\
{\tt vguigues@fgv.br}
\and
{\bf Volker Kr\"atschmer}\\Faculty of Mathematics\\
University of Duisburg-Essen\\
D-45127 Essen, Germany\\
{\tt volker.kraetschmer@uni-due.de}
\and
{\bf Alexander Shapiro}\thanks{Research of this author  was partly supported by  DARPA EQUiPS program, grant SNL 014150709.}\\
School of Industrial and Systems Engineering\\
Georgia Institute of Technology\\
Atlanta, GA 30332-0205, USA,\\ {\tt ashapiro@isye.gatech.edu}\\
}

\date{}

\maketitle

\begin{abstract}
We study   statistical properties of
the optimal value and optimal solutions of the Sample Average Approximation of  risk averse stochastic  problems.
 Central Limit Theorem type results are derived  for the optimal value and optimal solutions
when the stochastic program is expressed in terms of a law invariant coherent risk measure.
The obtained results are applied to  hypotheses testing problems aiming at comparing the optimal values
of several  risk averse convex stochastic programs on the basis of samples of the underlying random vectors.
We also consider  non-asymptotic tests based on
confidence intervals on the optimal values of the stochastic programs
obtained using the Stochastic Mirror Descent algorithm.
Numerical simulations show how to use our developments
to choose among different distributions and show the
superiority of the asymptotic tests on a class of risk averse stochastic programs.
\end{abstract}




\setcounter{equation}{0}
\section{Introduction}
\label{sec:introd}

Consider the following risk averse stochastic program
\begin{equation}\label{ropt-1}
\min_{x \in \X} \big\{ g(x):=\R(G_{x})\big\}.
 \end{equation}
Here  $\X$ is a nonempty compact subset of $\bbr^m$, $G_{x}$ is a random variable depending on
$x\in \X$ and $\R$ is a risk measure. We assume that $G_{x}$ is given in the
form  $G_{x}(\w)=G(x,\xi (\w))$, where $G:\X\times\bbr^d\to \bbr$ and  $\xi:\O\to\bbr^d$  is a random vector defined on a probability space $(\O,\F,P)$ whose distribution is supported
on set $\Xi\subset \bbr^d$. We also assume that risk measure $\R$ is {\em law invariant} (we will give precise definitions in  Section \ref{sec:disc}).

Let $\xi_j=\xi_j(\w)$, $j=1,...,N$, be an i.i.d sample of the random vector $\xi$  defined on the same probability space. Then the respective sample estimate of $g(x)$, denoted $\hat{g}_N(x)$, is obtained by replacing the ``true" distribution of the random vector $\xi$ with its empirical estimate. Consequently the true optimization problem (\ref{ropt-1}) is approximated by the problem
\begin{equation}\label{ropt-2}
\Min_{x\in \X}  \hat{g}_N(x),
\end{equation}
referred to as the Sample Average Approximation  (SAA) problem.  Note that $\hat{g}_N(x)=\hat{g}_N(x,\w)$ is a random function, sometimes  we suppress dependence on $\w$   in the notation. In particular if $\R$ is the
expectation operator, i.e., $g(x)=\bbe[G_x]$, then $\hat{g}_N(x)=N^{-1}\sum_{j=1}^N G(x,\xi_j)$.

 We denote by $\vv_*$ and $\hat{\vv}_N$ the optimal values of problems (\ref{ropt-1}) and (\ref{ropt-2}), respectively, and study statistical properties of $\hat{\vv}_N$. 
The random sample can be given by collected data or can be generated by Monte Carlo sampling techniques in the goal of solving the true problem by the SAA method.
Although conceptually different, both situations lead to the same statistical inference.

 The statistical analysis allows us to address the following question of asymptotic tests of hypotheses. Suppose that we are given
$K \geq 2$ optimization problems of the form \eqref{ropt-1}
with $\xi$, $G$, and $\X$ respectively replaced by $\xi^i$, $G_{i}$, and $\X_i$ for problem $i\in \{1,...,K\}$.
On the basis of samples $\xi_1^i,\ldots,\xi_N^i,$ of size $N$, of
$\xi^i, i=1,\ldots,K$, and
denoting by $\vv_*^i$ the optimal value of problem $i$,
we study the statistical tests:
\begin{equation} \label{deftest1}
\begin{array}{lll}
(a) & H_0 :\; \vv_{*}^1= \vv_{*}^2 = \ldots = \vv_{*}^K ,\\
(b) &H_0 :\; \vv_{*}^i \leq  \vv_{*}^j  \mbox{ for }1 \leq j \neq i \leq K ,\\
(c)& H_0 :\; \vv_{*}^1 \leq  \vv_{*}^2  \leq  \ldots \leq  \vv_{*}^K ,
\end{array}
\end{equation}
against the respective unrestricted alternatives.
As a special case, if the feasibility sets of the $K$ optimizations problems are singletons, say $\{x_*^i \}$ for problem $i$, the above
tests aim at comparing the risks $\R(G_{x_*^1}),
\ldots, \R(G_{x_*^K})$.
These tests are useful when we want to choose among $K$ candidate solutions $x_*^{1}, \ldots, x_*^{K}$ of problem
\eqref{ropt-1} for  the one with the smallest risk measure value, using risk measure $\mathcal{R}$ to rank the distributions $G_{x_*^{i}}, i=1,\ldots,K$.

Setting $\theta:=(\vv_{*}^1, \ldots, \vv_{*}^K)$, we also consider the following  extension of tests \eqref{deftest1}(a),(b),(c):
\begin{equation} \label{testcontregalite}
H_0:\theta \in \Theta_0 \mbox{ against }{H_1}:\theta \in \bbr^K,
\end{equation}
with $\Theta_0\subset \bbr^K$ being  a linear space  or a convex cone, as well as tests on the optimal value $\vv_*$ of \eqref{ropt-1} of the form
\begin{equation}\label{deftest2}
\begin{array}{lll}
(a) & H_0:\;\vv_* =\rho_0  \hspace*{-0.3cm}  & \mbox{ against }{H_1}: \;\vv_* \neq \rho_0,\\
(b) & H_0:\; \vv_* \leq \rho_0 \hspace*{-0.3cm}  & \mbox{ against }{H_1}: \;\vv_* > \rho_0 .
\end{array}
\end{equation}
Tests \eqref{deftest1} and \eqref{deftest2} will also be studied in a nonasymptotic setting.

Finally, numerical simulations illustrate our results: we show how to use our developments to
choose, using tests \eqref{deftest1}, among different distributions. We also use these tests to compare the optimal
value of several risk averse stochastic programs. It is shown that the Normal  (Gaussian) distribution  approximates well
the distribution of $\hat{\vv}_N$ already for $N=20$ and problem sizes up to $n=10\,000$,
and that the asymptotic tests yield much smaller type II errors than the considered  nonasymptotic   tests for small  to moderate sample size ($N$ up to $10^5$) and problem size ($n$ up to 500).

We use the following notation throughout the paper.  By $F_Z(z):=P(Z\le z)$ we denote
 the cumulative distribution function (cdf) of a random variable $Z:\O\to\bbr$. By $F^{-1}(\alpha)=\inf\{t:F(t)\ge \alpha\}$ we denote the left-side $\alpha$-quantile of the cdf $F$. By $\cQ_F(\alpha)$ we denote the interval of $\alpha$-quantiles of cdf $F$, i.e.,
 \begin{equation}\label{quan}
  \cQ_F(\alpha)=[a,b],\;where \;a:=F^{-1}(\alpha),\;b:= \sup\{t:F(t)\le \alpha\}.
 \end{equation}
 By $\ind_A(\cdot)$ we denote the indicator function of set $A$. We consider space  $\Z:=L_p(\O,\F,P)$, $p\in [1,\infty)$, of random variables $Z:\O\to\bbr$ having finite $p$-th order moments. The dual of space $\Z$ is the space $\Z^*=L_q(\O,\F,P)$, where $q\in (1,\infty]$ is such that  $1/p+1/q=1$. For $Z\in \Z$ and $\zeta\in \Z^*$ their scalar product is defined as the integral
 $\lan \zeta,Z\ran=\int_\O \zeta(\w) Z(\w)dP(\w)$.
The notation $Z \succeq Z'$ means that $Z(\w)\geq Z'(\w)$ for a.e. $\w\in \O$.
We denote by $C[a,b]$ the space of continuous functions $\psi:[a,b]\to \bbr$ equipped with the norm $\|\psi\|_\infty:=\sup_{t\in [a,b]}|\psi(t)|$.
 It is said that functions $h_k:\bbr^n\to\bbr$ converge  to $h$ uniformly on $\bbr^n$ if $\sup_{x\in \bbr^n}|h_k(x) -h(x) | \to 0$ as $k\to\infty$.

\setcounter{equation}{0}
\section{Preliminary discussion}
\label{sec:disc}

Risk measure  $\R:\Z\to\bbr$ is a   functional assigning to a random variable $Z\in \Z$ real value $\R(Z)$. Note that we consider here real valued risk measures, i.e., we do not allow $\R(Z)$ to have an infinite value.
In the influential paper of Artzner et al \cite{ADEH99} it was suggested that a ``good" risk measure should satisfy
 the following conditions (axioms).
\begin{itemize}
\item[(i)]
 {\sl Monotonicity:} If $Z,Z'\in\Z$ and  $Z \succeq Z'$,
then $\R(Z)\geq \R(Z')$.
\item [(ii)]
{\sl Convexity:}
\[
\R(t Z  + (1-t)Z') \le t \R(Z) +
(1-t)\R(Z')
\]
\
 for all $Z,Z'\in\Z$ and all $t\in [0,1]$.
\item[(iii)]  {\sl Translation Equivariance:} If $a\in \bbr$ and $Z\in\Z$, then $\R(Z+a)=\R(Z)+a$.
\item [(iv)]  {\sl Positive Homogeneity:} If $t\ge 0$ and $Z\in \Z$, then
$\R(t Z)=t\R(Z)$.
\end{itemize}

Risk measures $\R$ satisfying the above axioms (i)-(iv)  were  called {\em coherent} in  \cite{ADEH99}. If a risk measure satisfies axioms (i)-(iii), but not necessarily (iv),  it is called {\em convex} (cf., \cite{fol02}).
We assume  that $\R$ is {\em law invariant}. That is,  $\R(Z)$ depends only on the distribution of $Z$, i.e., if $Z,Z'\in \Z$ have the same cumulative distribution function then $\R(Z)=\R(Z')$.
We also  assume that the probability space $(\O,\F,P)$ is {\em nonatomic}.

 Since  a law invariant   risk measure   $\R$  can be considered as a function of its cdf    $F(\cdot)=F_Z(\cdot)$,   we also  write  $\R(F)$ to denote the corresponding value  $\R(Z)$.
Let $Z_1,...,Z_N$ be an i.i.d sample of $Z$ and $\hF_N=N^{-1}\sum_{j=1}^N \ind_{[Z_j,\infty)}$ be the corresponding empirical estimate of the cdf $F$.  By replacing $F$ with its empirical estimate $\hF_N$,  we obtain the  estimate $\R(\hF_N)$ to which we refer  as the {\em sample} or {\em empirical} estimate of $\R(F)$.
We assume that for every $x\in \X$, the random variable $G_x$ belongs to the space $\Z$, and hence $g(x)=\R(G_x)$ is well defined for every $x\in \X$.   Let $F_x$ be the cdf of random variable $G_x$, $x\in \X$, and $\hF_{x,N}$ be  the empirical cdf associated with the sample  $G(x,\xi_1),...,G(x,\xi_N)$.
Then we can write  $g(x)=\R(F_x)$
and $\hat{g}_N(x) =\R (\hF_{x,N})$.

 We have the following result about the convergence of the optimal value and optimal solutions of the SAA problem (\ref{ropt-2}) to their counterparts of the ``true" problem (\ref{ropt-1}) (cf., \cite[Theorem 3.3]{sha13}).
\begin{theorem}
 \label{th-3.3}
 Let  $\R:\Z\to\bbr$ be a law invariant   convex risk measure. Suppose that  the set $\X$ is nonempty and compact and the following conditions hold: {\rm (i)}
 the function $G_x(\w)$ is random lower semicontinuous, i.e., the epigraphical  multifunction $\w\mapsto \{(x,t)\in \bbr^{n+1}:G_x(\w)\le t\}$   is closed valued and measurable,
 {\rm (ii)}  for every $\bar{x}\in \bbr^n$ there is a neighborhood $\V_{\bar{x}}$ of $\bar{x}$  and a function $h\in \Z$ such that
    $G_x(\cdot)\geq h(\cdot)$ for all $x\in \V_{\bar{x}}$.

 Then the optimal value $\hat{\vv}_N$ of    problem {\rm (\ref{ropt-2})} converges w.p.1 to the optimal value $\vv_*$ of the ``true" problem {\rm (\ref{ropt-1})}, and the  distance from an optimal solution $\hat{x}_N$ of {\rm (\ref{ropt-2})}  to the set of optimal solutions of {\rm (\ref{ropt-1})} converges w.p.1 to zero as $N\to +\infty$.
  \end{theorem}

\begin{remark}
\label{rem-1}
{\rm
Recall that it is   assumed that   the probability space $(\O,\F,P)$ is nonatomic. Then without loss of generality we can assume that $\O$ is the interval $[0,1]$ equipped with its Borel sigma algebra and uniform probability distribution $P$. We refer to this probability space as the {\em standard probability space}.
\begin{itemize}
  \item [$\bullet$]
 By $\LL_p$, $p\in [1,\infty)$,  we denote the space $L_p(\O,\F,P)$ defined on the
standard probability space $(\O,\F,P)$.
\end{itemize}
Recall that the dual $\LL^*_p=\LL_q$.
For a cdf $F$ we can view $F^{-1}$ as a measurable function defined on the standard probability space.  Then $F^{-1}$ is an element of the space $\LL_p$ iff $\int_{-\infty}^{+\infty}|z|^p dF(z)<+\infty$. With some abuse of
notation we write that a cdf $F\in \LL_p$ if $F^{-1}\in \LL_p$. Note also
that  an element
$Z\in \LL_p$ is distributionally equivalent  to  $F^{-1}_Z$, and $F^{-1}_Z\in \LL_p$ iff $Z\in \LL_p$.
}
\end{remark}

\setcounter{equation}{0}
\subsection{Dual representations of law invariant coherent risk measures}
\label{dualrepresentations}

 Every  coherent risk measure $\R:\LL_p\to\bbr$ has the  dual representation
 \begin{equation}\label{dual}
 \R(Z)=\sup_{\zeta\in \cA} \int_0^1 \zeta(t) Z(t)dt,
 \end{equation}
 where $\cA\subset\LL_p^*$ is a convex weakly$^*$ compact set of density functions. Since real valued coherent risk measures are continuous in the norm topology of the Banach space $\LL_p$, this dual representation follows from the Fenchel-Moreau Theorem  (cf., \cite{rs06}).

 For {\em law invariant} coherent risk measures the dual representation (\ref{dual}) can be written in the following form
\begin{equation}\label{asy-1}
\R(F)=\sup_{\sigma\in \Upsilon} \int_0^1 \sigma(\tau) F^{-1}(\tau)d\tau,
 \end{equation}
where $\Upsilon$ is a weakly$^*$ compact subset of  $\LL_p^*$ consisting of so-called spectral functions. A function $\sigma:[0,1)\to [0,+\infty)$ is called {\em spectral} if $\sigma(\cdot)$ is right side continuous,
monotonically nondecreasing and such that $\int_0^1 \sigma(\tau)d\tau=1$.
The representation (\ref{asy-1}) is obtained from (\ref{dual}) by noting that $Z\in \LL_p$ is distributionally equivalent to $F_Z^{-1}$, and
applying a measure preserving transformation
(cf., \cite{mor12}).
In particular if $\Upsilon=\{\sigma\}$ is a singleton, then
\begin{equation}\label{srmes}
\R(F)=  \int_0^1 \sigma(\tau) F^{-1}(\tau)d\tau
\end{equation}
 is called spectral (or distortion) risk measure. The so-called generating set $\Upsilon$ is not defined uniquely. In a sense minimal generating set is formed by the weak$^*$ topological closure of the set of spectral functions which are exposed points of the set $\cA$ (cf., \cite{pich12}, \cite[Section 6.3.4]{SDR}).
 Consider the set of maximizers in the right-hand side of (\ref{asy-1}),
 \begin{equation}\label{upmax}
 \bar{\Upsilon}(F):=\arg\max_{\sigma\in \Upsilon} \int_0^1 \sigma(\tau) F^{-1}(\tau)d\tau.
 \end{equation}
 Since $F^{-1}\in \LL_p$  and the set   $\Upsilon$ is   weakly$^*$ compact, it follows that the set $\bar{\Upsilon}(F)$ is nonempty and weakly$^*$ compact.

It is also  possible to write representation (\ref{asy-1}) in the following equivalent form  \begin{equation}\label{dist}
\R(F)=\sup_{\psi\in \Psi}
\left\{\int_0^{+\infty}\big[1-\psi(F(t))\big]dt-
\int_{-\infty}^0 \psi (F(t))dt\right\},
\end{equation}
where $\Psi:=\cV(\Upsilon)$ with   $\cV$ being  a mapping from the set of spectral functions into the space $C[0,1]$, defined as
\begin{equation}\label{mapp}
(\cV\sigma)(\alpha):=\int_0^\alpha \sigma(t)dt, \;\alpha\in [0,1].
\end{equation}
Indeed, note that for any spectral function $\sigma$, the corresponding
$\psi=\cV\sigma$ is convex, continuous monotonically nondecreasing  on the interval [0,1] function with $\psi(0)=0$ and $\psi(1)=1$.
By change of variables $\tau=F(t)$ and using integration by parts we can write
\begin{equation}\label{inter-1}
 \int_0^{+\infty}\big[1-\psi(F(t))\big]dt=
 \int_{F(0)}^1 (1-\psi(\tau))dF^{-1}(\tau)=
 \int_{F(0)}^1 F^{-1}(\tau)\psi'(\tau)d\tau.
\end{equation}
Similarly
\begin{equation}\label{inter-2}
 \int_{-\infty}^0 \psi (F(t))dt=\int_0^{F(0)}\psi(\tau)d F^{-1}(\tau)=- \int_0^{F(0)} F^{-1}(\tau)\psi'(\tau)d\tau,
 \end{equation}
 and hence (\ref{dist}) follows from (\ref{asy-1}).

The function $\psi=\cV\sigma$ is directionally differentiable. Its directional derivative
$\psi'(t,h)=\lim_{\tau\downarrow 0}[\psi(t+\tau h)-\psi(t)]/\tau$ is
\begin{equation}\label{dirderpsi}
\psi'(t,h)=\left\{
\begin{array}{lll}
  \psi'_+(t)h& {\rm if}& h\ge 0,\\
 \psi'_-(t)h& {\rm if}& h\le 0,
\end{array}\right.
\end{equation}
where $\psi'_-(t)$ and $\psi'_+(t)$ are the respective left and right side derivatives
\begin{equation}\label{discon}
 \psi'_-(t)=\lim_{\tau\uparrow t}\sigma(\tau)\;{\rm and}\;\psi'_+(t)=\lim_{\tau\downarrow t}\sigma(\tau).
\end{equation}
In particular, if $\sigma(\cdot)$ is continuous at $t$, then $\psi(\cdot)$ is differentiable at $t$ and   $\psi'(t)=\sigma(t)$.
 Note that
since a spectral function is  monotonically nondecreasing, the set of its discontinuous points is countable.

\begin{lemma}
\label{lem-map}
The mapping $\cV:\Upsilon\to C[0,1]$ is continuous with respect to the weak$^*$ topology of $\LL_p^*$ and  norm topology of $C[0,1]$.
\end{lemma}

\begin{proof}
The set   $\Upsilon$ is a bounded subset of $\LL_p^*$. Since the space $\LL_p$ is
separable, it follows that the  weak$^*$ topology on $\Upsilon$ is metrizable. Therefore it suffices to show that if a sequence $\sigma_n\in \Upsilon$ converges to  $\sigma\in \Upsilon$ in weak$^*$ topology, then $\psi_n=\cV\sigma_n$ converges to $\psi=\cV\sigma$ in the  norm topology of $C[0,1]$. Note that functions $\psi_n$ and $\psi$ are convex continuous 
monotonically nondecreasing  on the interval [0,1]. Let us also observe that the
weak$^*$ convergence of $\sigma_n$ to $\sigma$ implies pointwise convergence  $\psi_n(t)\to \psi(t)$ for all $t\in [0,1]$. Indeed,
$\psi_n(t)=\lan \sigma_n,\ind_{[0,t]}\ran$ and
$\ind_{[0,t]}$ belongs to the space $\LL_p$.
Since functions $\psi_n$ are convex, it follows from the pointwise convergence  that $\sup_{t\in I}|\psi_n(t)-\psi(t)|$ tends to zero for any interval $I\subset (0,1)$ (e.g., \cite[Theorem 10.8]{roc70}). By monotonicity and continuity  of $\psi$ this
implies that $\psi_n$ converge to $\psi$ uniformly on [0,1]. This completes the proof. \hfill \fb
\end{proof}
\\

By the above discussion there is a one-to-one correspondence between representations (\ref{asy-1}) and (\ref{dist})  defined  by  $\Psi=\cV(\Upsilon)$.
Since the generating set $\Upsilon$ is weakly$^*$ compact, it follows that the set $\Psi=\cV(\Upsilon)$ is a compact subset of $C[0,1]$ (cf., \cite{bel2011}). Consider the set of maximizers in the right-hand side of (\ref{dist}),
\begin{equation}\label{mset-1}
 \bar{\Psi}(F):=\arg\max_{\psi\in \Psi}
\left\{\int_0^{+\infty}\big[1-\psi(F(t))\big]dt-
\int_{-\infty}^0 \psi (F(t))dt\right\}.
\end{equation}
It follows that
 $\bar{\Psi}(F)=\cV(\bar{\Upsilon}(F))$ is a nonempty and compact subset  of $C[0,1]$.

An important  risk measure is the Average Value-at Risk measure
\begin{equation}\label{avar}
\avr_\alpha (F)=\frac{1}{1-\alpha}\int_\alpha^1 F^{-1}(\tau) d\tau,\;\alpha\in [0,1).
\end{equation}
That is, $\avr_\alpha (F)$ is a spectral risk measure with the  spectral function
$\sigma(\cdot)=(1-\alpha)^{-1}
\ind_{[\alpha,1)}(\cdot)$. The corresponding space here is $\LL_1$, i.e., it is defined for $F$ such that  $\int |z|dF(z)<+\infty$.
Equivalently $\avr_\alpha (Z)$ can be written as
\begin{equation}\label{avar-2}
\avr_\alpha (F)=\inf_{\tau\in \bbr}\left\{
\tau+(1-\alpha)^{-1}
\int_{-\infty}^{+\infty} [z-\tau]_+ dF(z) \right\}.
\end{equation}
For $\alpha\in (0,1)$
the minimum in the right-hand side of (\ref{avar-2}) is  attained at any point of the interval $\cQ_F(\alpha)$  of $\alpha$-quantiles of the distribution   $F$, in particular at the left side quantile  $\tau=F^{-1}(\alpha)$. For $\alpha=0$, $\avr_0 (F)=\bbe_F[Z]$ although the minimum in
(\ref{avar-2}) is not attained if the distribution is unbounded from below.
Note that $\avr_\alpha (F)$ is monotonically nondecreasing in $\alpha\in [0,1)$, and tends to $\lim_{t\uparrow  1}F^{-1}(t)$ as $\alpha\to 1$.

Consider the transformation
\begin{equation}\label{tran-1}
 (\bbt\mu)(t):=\int_0^t (1-\alpha)^{-1}d\mu(\alpha),\;t\in [0,1),
\end{equation}
from the set of probability distribution functions (measures) $\mu(\cdot)$ on the interval [0,1) to the set of spectral functions. The inverse of this transformation is  (cf., \cite[Lemma 4.63]{fol02}, \cite[p.307]{SDR})
\begin{equation}\label{tran-2}
 (\bbt^{-1}\sigma)(\alpha) =(1-\alpha)\sigma(\alpha)+(\cV\sigma)(\alpha),
 \;\alpha\in [0,1),
\end{equation}
where mapping $\cV$ is defined in (\ref{mapp}).
Then representation (\ref{asy-1}) can be written in the following equivalent form
\begin{equation}\label{kus}
\R(F)=\sup_{\mu\in \cM}\int_0^1 \avr_\alpha (F)d\mu(\alpha),
\end{equation}
where $\cM:=\bbt^{-1}(\Upsilon)$.
The representation (\ref{kus}) is referred to as the Kusuoka representation of $\R$ (cf., \cite{kus:01}).

The mapping $\bbt$ is one-to-one and continuous\footnote{We consider here   the weak topology of probability  measures on the interval [0,1] and the weak$^*$ topology of $\LL_p^*$ } (cf., \cite[Proposition 3.4]{pich12}). It follows that the inverse mapping $\bbt^{-1}$ is also continuous on the set $\Upsilon$, the set
$\cM=\bbt^{-1}(\Upsilon)$ is compact, and the set
\begin{equation}\label{max-m}
 \bar{\cM}(F):=\arg\max_{\mu\in \cM}\int_0^1 \avr_\alpha (F)d\mu(\alpha)
\end{equation}
is nonempty compact   and $ \bar{\cM}(F)=\bbt^{-1}(\bar{\Upsilon}(F))$.

Note that since we assume that $F\in \LL_p$, $p\in [1,+\infty)$, measures in $\cM$ do not have positive mass at $\alpha=1$, although they may have positive mass  at $\alpha=0$. It will be convenient to write   explicitly measures   $\mu\in \cM$  in the form
\begin{equation}\label{mu-p}
  \mu=w\delta (0)+(1-w) \mu',
\end{equation}
where $w\in [0,1]$ and $\mu'$ is the respective probability measure on $[0,1]$ having zero mass at 0 and 1.

Using variational representation (\ref{avar-2}) of $\avr_\alpha$, it is possible to write the Kusuoka representation (\ref{kus}) in the following minimax form
\begin{equation}\label{minmax}
 \R(F)=\sup_{\mu\in \cM}\left\{w\int_{-\infty}^{+\infty} zdF(z) +(1-w)
 \int_0^1 \inf_{t \in \bbr}\left\{
\int_{-\infty}^{+\infty}  h_\alpha(z,t) dF(z)\right\}d\mu'(\alpha)\right\},
\end{equation}
where
\begin{equation}\label{funch}
 h_\alpha(z,t):=
  t+(1-\alpha)^{-1} [z-t]_+,\;\alpha\in (0,1).
\end{equation}

By interchanging the integral and minimization operators in the right-hand side of (\ref{minmax}) (cf., \cite[Theorem 14.60]{RW}), we can write
\begin{eqnarray}
\label{minm2}
\R(F)&=&\sup_{\mu\in \cM}\inf_{\tau \in \LL_p}
\left\{w\int_{-\infty}^{+\infty} zdF(z)+(1-w)
\int_0^1 \int_{-\infty}^{+\infty}
  h_\alpha(z,\tau(\alpha)) dF(z)d\mu'(\alpha)\right\} \\
  \label{minm3}
  &=& \sup_{\mu\in \cM}\inf_{\tau \in \LL_p}
 \left\{w\int_{-\infty}^{+\infty} zdF(z)+(1-w)
\int_0^1 \int_{-\infty}^{+\infty}
  h_\alpha(z,\tau(\alpha)) d\mu'(\alpha) dF(z)\right\}.
  \end{eqnarray}
Note that minimization in (\ref{minm2}) and (\ref{minm3}) is performed over functions $\tau\in \LL_p$.

By interchanging the `sup' and `inf' operators we can write the dual of problem (\ref{minm3}):
\begin{equation}\label{dual-1}
\inf_{\tau \in \LL_p}\sup_{\mu\in \cM}
 \left\{w\int_{-\infty}^{+\infty}zdF(z)+(1-w)
 \int_{-\infty}^{+\infty} \int_0^1
  h_\alpha(z,\tau(\alpha)) d\mu'(\alpha) dF(z)\right\}.
 \end{equation}
The set $\bar{\cM}(F)$, defined in (\ref{max-m}),  is also the set of optimal solutions of the problem (\ref{minm3}). We denote by  $\bar{\cT}(F)$ the set of optimal solutions of the dual problem (\ref{dual-1}).  The set  $\bar{\cT}(F)$  can be empty. Also the set $\bar{\cT}(F)$ can be   unbounded. For example, if $\R(\cdot):=\bbe[\cdot]$, then $\cM=\{\delta(0)\}$ and the right-hand side of
(\ref{dual-1}) does not depend on $\tau\in \LL_p$. In that case $\bar{\cT}(F)=\LL_p$.

\setcounter{equation}{0}
\section{Informal analysis}
\label{sec:infan}

In this section we discuss asymptotics of the empirical estimates $\R(\hF_N)$, in a somewhat informal way, and consider examples. By using
(\ref{minm3}) we can write
\begin{eqnarray}
\R(\hF_N)&=&\sup_{\mu\in \cM}\inf_{\tau \in \LL_p}\left\{ w\int_{-\infty}^{+\infty} zd \hF_N(z)+ (1-w)
  \int_{-\infty}^{+\infty} \int_0^1 h_\alpha(z,\tau(\alpha)) d\mu'(\alpha)d \hF_N(z)\right\}
  \\&= & \sup_{\mu\in \cM}\inf_{\tau \in \LL_p}
  \left\{w\bar{Z}+(1-w)
  \int_0^1  \hat{h}_{\alpha,N}(\tau(\alpha)) d\mu'(\alpha)\right\},
\end{eqnarray}
where $\bar{Z}:=N^{-1}\sum_{j=1}^N Z_j$ and
\[
\hat{h}_{\alpha,N}(t):= N^{-1}\sum_{j=1}^N h_\alpha(Z_j,t)=
 t+(1-\alpha)^{-1} N^{-1}\sum_{j=1}^N [Z_j-t]_+,\;  \alpha\in (0,1) .
\]

Suppose that the minimax problem (\ref{minm3}) has a nonempty set of saddle points, given by
 $\bar{\cM}(F)\times \bar{\cT}(F)$. Then
the minimax representation suggests the following asymptotics
\begin{equation}\label{inas-1}
\R(\hF_N)=\sup_{\mu\in \bar{\cM}(F)}\inf_{\tau \in \bar{\cT}(F)}
  \left\{w\bar{Z}+(1-w)
  \int_0^1  \hat{h}_{\alpha,N}(\tau(\alpha)) d\mu'(\alpha)\right\}+o_p(N^{-1/2})
\end{equation}
and
\begin{equation}\label{inas-2}
 N^{1/2}\big[\R(\hF_N)-\R(F)\big]\cdis \sup_{\mu\in \bar{\cM}(F)}\inf_{\tau \in \bar{\cT}(F)}\bby(\mu,\tau),
\end{equation}
where $\bby(\mu,\tau)$, $(\mu,\tau)\in \bar{\cM}(F)\times \bar{\cT}(F)$,  is the corresponding  Gaussian process (we will discuss this later).
 In particular, if the set  of saddle points
 is a singleton, $\bar{\cM}(F)=\{\bar{w}\delta(0)+(1-\bar{w})\bar{\mu}'\}$, $\bar{\cT}(F) =\{\bar{\tau}\}$,  then
   $N^{1/2}\big[\R(\hF_N)-\R(F)\big]$ converges in distribution to normal $\N(0,\nu^2)$ with variance
 \begin{equation}\label{inas-3}
\nu^2=\var_F\left\{\bar{w}Z+(1-\bar{w})\int_0^1
(1-\alpha)^{-1}
\left[Z-\bar{\tau}(\alpha)\right]_+d\bar{\mu}'
(\alpha)\right\}.
 \end{equation}

The above derivations are not rigorous. In Theorem \ref{th-minmax} of the next section we discuss a particular case where   formulas (\ref{inas-1}) - (\ref{inas-3}) can be rigourously  proved  by an application of a finite dimensional  minimax asymptotic distribution theorem (cf., \cite{sha08}). Formula (\ref{inas-2}) suggests that the asymptotic distribution of $\R(\hF_N)$ could be non-normal for two somewhat different reasons. Namely,  it could happen that the set $\bar{\cM}(F)$ is not a singleton. Recall that $\bar{\cM}(F)=\bbt^{-1}(\bar{\Upsilon}(F))$. Therefore $\bar{\cM}(F)$ is a singleton if $\bar{\Upsilon}(F)$ is a singleton, in particular if $\R$ is a spectral risk measure. As it was pointed out, the generating set $\Upsilon$, and hence the sets $\cM$ and $\Psi$, are not defined uniquely. Therefore uniqueness of the respective maximizers in (\ref{upmax}), (\ref{mset-1}) and (\ref{max-m}) should be verified for the  minimal representation.
It could also happen that the set $\bar{\cT }(F)$ is not a singleton. Let us discuss some illustrative examples.

\begin{example}[$\avr$ risk measure]
{\rm
Consider  $\R:=\avr_\alpha$,  $\alpha\in (0,1)$. This is a spectral risk measure. Its Kusuoka representation is given by the singleton $\cM=\{\delta(\alpha)\}$, and $\bar{\cT }(F)=\cQ_F(\alpha)$.
For this risk measure formula (\ref{inas-1}) becomes
\begin{equation}\label{av-v2}
\avr_\alpha (\hF_N)=\inf_{t\in \cQ_F(\alpha)} \left \{t +\frac{1}{(1-\alpha)N}
\sum_{j=1}^N [Z_j-t]_+\right \}
  +o_p(N^{-1/2}\}.
\end{equation}
Suppose that $\bbe_F[Z^2]<+\infty$. Then by the variational representation (\ref{avar-2}) of $\avr_\alpha$, the assertion   (\ref{av-v2})     follows from a general result about asymptotics of sample average approximations of stochastic programs
(cf., \cite[Theorem 3.2]{sha91}, \cite[Section 6.6.1]{SDR}).
If, moreover, $\cQ_F(\alpha)=\{F^{-1}(\alpha)\}$ is a singleton, then  $N^{1/2}\left[\avr_\alpha(\hat{F}_N)-
 \avr_\alpha(F)\right]$ converges in distribution to normal $\N(0,\nu^2)$ with
\begin{equation}\label{av-v3}
\nu^2=(1-\alpha)^2\var_F
\left([Z-F^{-1}(\alpha)]_+\right).
\end{equation}
It follows from (\ref{av-v2}) that uniqueness of the quantile $F^{-1}(\alpha)$ is also a necessary condition for asymptotic normality of $\avr_\alpha(\hat{F}_N)$. In Theorem \ref{th-minmax}  (of  Section \ref{sec:asym})  we give a more general result for which $\avr_{\alpha}$ risk measure is a particular case.

A different formula for the asymptotic variance of $\avr_\alpha(\hat{F}_N)$ was given in  \cite{pfl2010} and \cite{bel2011}. We are going to show now equivalence of their formula to (\ref{av-v3}).
Consider the $F$-Brownian bridge, denoted $\bb_F$. That is, $\bb_F(z)$  is a Gaussian process  with mean zero and covariances
\[
\bbe[\bb_F(z)\bb_F(z')]=F(z \wedge z')-F(z)F(z').
\]
Note that for any $\tau\in \bbr$,
\begin{eqnarray}
\label{su-2}
\var\left\{ \int_{\tau}^{+\infty} \bb_F(z)dz\right\}&=&
 \int_{\tau}^{+\infty} \int_{\tau}^{+\infty}\big [F(x \wedge y)-F(x)F(y)\big]dxdy\\
 \label{su-3}
 &=&
 \int_{\tau}^{+\infty} \int_{\tau}^{+\infty}\big [\bar{F}(x \vee y)-\bar{F}(x)\bar{F}(y)\big]dxdy\\
  \label{su-4}
 &=&
 \int_{\alpha}^{1}\int_{\alpha}^{1}
 (s\wedge t- st)dF^{-1}(s)dF^{-1}(t),
\end{eqnarray}
where $\bar{F}(\cdot):=1-F(\cdot)$.

\begin{lemma}
\label{lm-eq}
Suppose that  $F$ has finite second order moment  and let $\tau\in \bbr$. Then
\begin{equation}\label{equv}
\var_F \left( [Z -\tau]_+\right)=
\int_{\tau}^{+\infty} \int_{\tau}^{+\infty}\big [F(x \wedge y)-F(x)F(y)\big]dxdy.
\end{equation}
\end{lemma}

\begin{proof}
We have
\[
\var_F \left( [Z -\tau]_+\right)=
\int_{\tau}^{+\infty}(z-\tau)^2dF(z)-
\left(\int_{\tau}^{+\infty} (z-\tau)dF(z)\right)^2.
\]
Since $\bbe_F[Z^2]<+\infty$ it follows that $\lim_{z\to+\infty} z \bar{F}(z)=0$. Hence
using integration by parts
 we can write
\[
\int_{\tau}^{+\infty}
(z-\tau)dF(z)=-\int_{\tau}^{+\infty}
(z-\tau)d\bar{F}(z)=-\int_{\tau}^{+\infty}
\bar{F}(z)dz,
\]
and hence
\begin{equation}\label{derv-1}
\left(\int_{\tau}^{+\infty} (z-\tau)dF(z)\right)^2=\int_{\tau}^{+\infty}
\int_{\tau}^{+\infty}
\bar{F}(x)\bar{F}(y) dxdy.
\end{equation}
Now
\[
 \int_{\tau}^{+\infty}  \bar{F}(x \vee y)dx=\int_\tau^{y}\bar{F}(y)dx+
\int_y^{+\infty} \bar{F}(x)dx=(y-\tau)\bar{F}(y)+
\int_y^{+\infty} \bar{F}(x)dx.
\]
Hence
\begin{eqnarray*}
\int_{\tau}^{+\infty} \int_{\tau}^{+\infty}  \bar{F}(x \vee y)dxdy&=&
\int_{\tau}^{+\infty}(y-\tau)\bar{F}(y)dy+
\int_{\tau}^{+\infty}\int_y^{+\infty} \bar{F}(x)dxdy\\
&=& \int_{\tau}^{+\infty}(x-\tau)\bar{F}(x)dx+
\int_{\tau}^{+\infty}\int_{\tau}^{x}
\bar{F}(x)dydx\\
&=& 2\int_{\tau}^{+\infty}(x-\tau)\bar{F}(x)dx.
\end{eqnarray*}
Since $\bbe_F[Z^2]<+\infty$ it follows that $\lim_{z\to+\infty} z^2 \bar{F}(z)=0$, and hence using integration by parts we can write
\begin{equation}\label{derv-2}
\int_{\tau}^{+\infty}(z-\tau)^2dF(z)=
-\int_{\tau}^{+\infty}(z-\tau)^2d\bar{F}(z)
= 2\int_{\tau}^{+\infty}(z-\tau)\bar{F}(z)dz=
\int_{\tau}^{+\infty} \int_{\tau}^{+\infty}  \bar{F}(x \vee y)dxdy.
\end{equation}
Noting equivalence of (\ref{su-2}) and (\ref{su-3}), we obtain (\ref{equv})  by (\ref{derv-1}) and (\ref{derv-2}).
\hfill \fb
\end{proof}
\\

By using (\ref{equv}) the right-hand side of (\ref{av-v3}) can be written in terms of the cdf $F$.
}
\end{example}

\begin{example}[Absolute semideviation  risk measure]\label{ex-2}
{\rm
Consider  risk measure
\begin{equation}\label{abs-1}
  \R_c(F):=\bbe_F[Z]+c\,
  \bbe_F[Z-\bbe_F(Z)]_+, \;c\in (0,1].
\end{equation}
We assume that cdf $F$ has finite first order moment, i.e., $\R_c(\cdot)$ is defined on $\LL_1$. This risk measure has the following representation (cf., \cite{mor12})
\begin{eqnarray}
\label{abs-2}
\R_c(F)&=&\sup_{\gamma\in [0,1]}
  \big\{(1-c\gamma)\bbe_F(Z)+c\gamma
  \avr_{1-\gamma}(F)\big\}\\
\label{abs-3}
&=& \sup_{\gamma\in [0,1]}\inf_{t\in \bbr}
\bbe_F\big\{(1-c\gamma)Z+c\gamma t +c[Z-t]_+\big\}\\
\label{abs-3a}
&=&\inf_{t\in \bbr} \sup_{\gamma\in [0,1]}
\bbe_F\big\{(1-c\gamma)Z+c\gamma t +c[Z-t]_+\big\}.
\end{eqnarray}
Representation (\ref{abs-2}) is the (minimal) Kusuoka representation (\ref{kus}) of $\R_c$ with the corresponding set
$
\cM=\cup_{\gamma\in [0,1]}\{(1-c\gamma)\delta(0)+c\gamma \delta(1-\gamma)\}.
$
 Since
 \[
 \sup_{\gamma\in [0,1]}
\bbe_F\big\{(1-c\gamma)Z+c\gamma t +c[Z-t]_+\big\}=\bbe[Z]+
c\max\left\{\bbe[Z-t]_+,\bbe[t-Z]_+\right\},
 \]
 it follows that problem (\ref{abs-3a}) has unique optimal solution $t^*=\cm$, where $\cm:=\bbe_F[Z]$.
Now the set of minimizers of $\gamma t+ \bbe[Z-t]_+$,  over $t\in \bbr$, is defined by the equation $F(t)=1-\gamma$. It follows that
the set of saddle points of the minimax representation  (\ref{abs-3}) is
$[\underline{\gamma},\overline{\gamma}]\times \{\cm\}$, where
\[
\underline{\gamma}:=1-\prob(Z\le\cm),\;\;
 \overline{\gamma}:=1-\prob(Z<\cm)
  \]
  (cf., \cite[Section 6.6.2]{SDR}). In other words here
 \[
\bar{\cM}(F)=\cup_{\gamma\in [\underline{\gamma},\overline{\gamma}]}
\{(1-c\gamma)\delta(0)+c\gamma \delta(1-\gamma)\}
\]
  and  $\bar{\cT}(F)=\{\bar{\tau}(\alpha)\}$ is the singleton with $\bar{\tau}(\cdot)\equiv \bbe_F[Z]$.

The minimax representation (\ref{abs-3}) leads to the following asymptotics. Suppose that $\bbe_F[Z^2]<+\infty$. Then by a finite dimensional minimax asymptotics theorem (cf., \cite{sha08})
\begin{equation}\label{abs-4}
  \R_c(\hF_N)=\sup_{\gamma\in [\underline{\gamma},\overline{\gamma}]}
  \left\{ c\gamma \cm+(1-c\gamma) \bar{Z}+
  cN^{-1}
  \sum_{j=1}^N \big [Z_j- \cm\big]_+\right\}+o_p(N^{-1/2}),
\end{equation}
where $\bar{Z}:=N^{-1} \sum_{j=1}^N  Z_j$. In particular if the cdf $F(\cdot)$ is continuous at $\cm=\bbe_F[Z]$, then $N^{1/2}\big[\R_c(\hF_N)-\R_c(F)\big]$ converges in distribution to normal $\N(0,\nu^2)$ with variance
\begin{equation}\label{varabs}
 \nu^2=\var_F\big\{(1-c \gamma^*)Z
  +c[Z-\cm]_+\big\},
\end{equation}
where $\gamma^*:=1-F(\cm)=\bar{F}(\cm)$.
}
\end{example}

\subsection{Von Mises statistical functionals}

Let $G\in \LL_p$ be an arbitrary cdf and consider  convex combination $(1-t)F+tG=F+tH$, where $H:=G-F$. Suppose that the risk measure $\R$ is directionally differentiable at $F$ in direction $H$, i.e., the following limit exists
\begin{equation}\label{mises-1}
 \R'_F(H):=\lim_{t\downarrow 0}\frac{\R(F+tH)-\R(F)}{t}.
\end{equation}
If moreover the directional derivative $\R'_F(\cdot)$ is linear, then it is said that $\R$ is G\^{a}teaux differentiable at $F$.

Consider the approximation
\begin{equation}\label{mises-2}
\R(\hF_N)-\R(F)\approx \R'_F(\hF_N-F),
\end{equation}
and hence
\begin{equation}\label{mises-3}
N^{1/2}[\R(\hF_N)-\R(F)]\approx \R'_F\left(N^{1/2}(\hF_N-F)\right ).
\end{equation}
Since $N^{1/2}(\hF_N-F)$ converges in distribution to the $F$-Brownian bridge $\bb_F$, the approximation (\ref{mises-3}) suggests that
$N^{1/2}[\R(\hF_N)-\R(F)]$ converges in distribution to $\R'_F\left(\bb_F\right)$.

If moreover $\R$ is G\^{a}teaux differentiable at $F$, then
\begin{equation}\label{mises-4}
\R'_F(\hF_N-F)=\frac{1}{N}\sum_{j=1}^N IF(Z_j),
\end{equation}
where
\begin{equation}\label{mises-5}
IF(z):=\R'_F\left(\delta(z)-F\right)
\end{equation}
is the so-called {\em influence function}.
 Consequently
 $N^{1/2}[\R(\hF_N)-\R(F)]$ converges in distribution to normal $\N(0,\nu^2)$, with
$
\nu^2=\bbe_F[IF(Z)^2].
$
Of course, the above are heuristic arguments which require a rigourous justification.

Now consider representation (\ref{dist}) of the risk measure $\R$. Recall that each   function $\psi=\cV\sigma$ is directionally differentiable with directional derivative (\ref{dirderpsi}).
We have that
\[
\lim_{t\downarrow 0}\int_{-\infty}^0\frac{\psi (F(z)+t H(z))  -  \psi (F(z))}{t} dz =
\int_{-\infty}^0 \psi'(F(z),H(z))dz,
 \]
provided the limit and integral operators can be interchanged. Similar arguments can be applied  to the first integral term in the right-hand side of (\ref{dist}). It follows that if the set $\Psi=\{\psi\}$ is a singleton, i.e., $\R$ is a spectral risk measure, then
\begin{equation}\label{mises-6}
\R'_F(H)=-\int_{-\infty}^{+\infty} \psi'(F(z),H(z))dz,
\end{equation}
provided that the limit and integral operators can be interchanged. This suggests the asymptotics
 \begin{equation}\label{mises-7}
 N^{1/2}\left[\R(\hF_N)-
 \R(F)\right] \cdis - \int_{-\infty}^{+\infty} \psi'(F(z),\bb_F(z))dz.
\end{equation}

Now suppose that  every spectral function $\sigma\in \bar{\Upsilon}(F)$ is continuous at every point where $F^{-1}$ is discontinuous. Then
\begin{equation}\label{mises-8}
\R'_F(H)=\sup_{\sigma\in \bar{\Upsilon}(F)}\left\{ -\int_{-\infty}^{+\infty} \sigma(F(z))H(z)dz\right\}.
\end{equation}
In that case the suggested asymptotics are
\begin{equation}\label{mises-9}
 N^{1/2}\left[\R(\hF_N)-
 \R(F)\right] \cdis \sup_{\sigma\in \bar{\Upsilon}(F)} \left\{- \int_{-\infty}^{+\infty}\sigma(F(z))\bb_F(z)dz\right\}.
\end{equation}

For example, consider the mean-semideviation  risk measure
\begin{equation}\label{semi-1}
 \R_c(F):=\bbe_F[Z]+
 c\left(\bbe_F[Z-\bbe_F[Z]]_+^2\right)^{1/2},\;c\in (0,1].
\end{equation}
If $F(\cdot)$ is continuous at $\cm:=\bbe_F[Z]$, then $\R_c(\cdot)$ is G\^{a}teaux differentiable at $F$ and the corresponding influence function is
\begin{equation}\label{influ}
 IF(z)=z+ c(2\theta)^{-1}\left(
 [z-\cm]_+^2-\theta^2+2\kappa (1-F(\cm))(z-\cm)\right),
\end{equation}
where $\theta:=\left(\bbe_F[Z-\bbe_F[Z]]_+^2
\right)^{1/2}$ and $\kappa:=\bbe_F[Z-\cm]_+$. This indicates that continuity of $F(\cdot)$ at $\cm$ is a necessary condition for  $\R_c(\cdot)$ to be G\^{a}teaux differentiable at $F$, and hence for $\R(\hF_N)$ to be asymptotically normal.

\setcounter{equation}{0}
\section{Discrete Kusuoka case}
\label{sec:asym}

\subsection{Asymptotics of risk measures}

In this section we discuss asymptotics of empirical estimates $\R(\hF_N)$, and more generally of the optimal values $\hat{\vv}_N$ of the SAA problem (\ref{ropt-2}),
 for risk measures of the following form. For this class of risk measures some of the  required results are readily available.

Consider Kusuoka representation (\ref{kus}) and suppose that the set $\cM$ consists of measures supported on finite set $\{\alpha_0,\alpha_1,...,\alpha_k\}$, where
 $0=\alpha_0< \alpha_1<\alpha_2<\cdots\alpha_k<1$. That is
\begin{equation}\label{kex-1}
 \R(Z)=\sup_{w\in \cW}\left\{w_0\bbe[Z]+\sum_{i=1}^k w_i\avr_{\alpha_i}(Z)\right\},
\end{equation}
where $\cW$ is a nonempty   subset of $\Delta_{k+1}:=\{w\in \bbr_+^{k+1}:w_0+...+w_{k}=1\}$.
Recall that $\bbe[Z]=\avr_0[Z]$.
Note that $\R(Z)$ is finite valued for every $Z\in \LL_1$. Therefore we assume that $\R$ is defined on $\LL_1$, i.e., $\R:\LL_1\to \bbr$.
Note that $\R$ is not changed if $\cW$ is replaced by the topological closure of its convex hull. Therefore we assume that $\cW$ is {\em convex} and {\em closed}.

By making transformation (\ref{tran-1}) the above risk measure $\R$ can be written in the form
(\ref{asy-1}) with the corresponding set of spectral functions
\begin{equation}\label{spectr}
\Upsilon=\left\{\sigma: \sigma=w_0+\sum_{i=1}^k w_i (1-\alpha_i)^{-1}  \ind_{[\alpha_i,1]},\;w\in \cW\right\}.
\end{equation}
 Using representation (\ref{avar-2})
 we can write $\R(F)$ in the form
 \begin{equation}
 \label{kex-2}
 \R(F)=\sup_{w\in \cW}\inf_{\tau\in\bbr^k} \bbe_F[\phi(Z,w,\tau)]\\
= \inf_{\tau\in\bbr^k}\sup_{w\in \cW}\bbe_F[\phi(Z,w,\tau)],
 \end{equation}
where
 \begin{equation}
 \label{fun-v}
\phi(z,w,\tau):=
 w_0 z +\sum_{i=1}^k w_i\left(\tau_i+(1-\alpha_i)^{-1}
  [z-\tau_i]_+\right).
 \end{equation}
This representation is a particular case of the general minimax formula  (\ref{minm2})-(\ref{minm3}).
 The  $`\inf$' and $`\sup$' in (\ref{kex-2})  can be interchanged since the objective function is linear in $w$ and convex in $\tau$ and the set $\cW$  is compact.

 We make the following assumption.
 \begin{itemize}
   \item [(A)]
For every $i\in \{1,...,k\}$ there exists $w\in \cW$ such that $w_i\ne 0$.
 \end{itemize}
 This is a natural condition. Otherwise there is
$i\in \{1,...,k\}$ such that $w_i=0$ for all $w\in \cW$. In that case we can reduce the considered set $\{\alpha_0,\alpha_1,...,\alpha_k\}$ by removing the corresponding point $\alpha_i$.
Since the set $\cW\subset \Delta_{k+1}$ is convex, condition (A) means that the relative interior of $\cW$ consists of points with all their coordinates being positive.

  Consider the set
 \begin{equation}
 \label{extr-1}
 \overline{\cW}:=\arg\max_{w\in \cW}\sum_{i=0}^k w_i\avr_{\alpha_i}(Z),
 \end{equation}
 of maximizers in (\ref{kex-1}). Since the set $\cW$ is nonempty and compact, the set $\overline{\cW}$ is nonempty. This is also the set of maximizers in  (\ref{kex-2}).
 Note also that, under condition (A),
 \begin{equation}
 \label{extr-2}
 \arg\min_{\tau\in \bbr^k}
  \sup_{w\in \cW} \bbe_F[\phi(Z,w,\tau)]=\cQ_F(\alpha_1)\times\cdots\times \cQ_F(\alpha_k)=: \overline{\cT}.
  \end{equation}
  Indeed for every $\bar{w}\in \cW$   such that $\bar{w}_i\ne 0$, $i=1,...,k$, we have that
  $\arg\min_{\tau\in \bbr^k} \bbe_F[\phi(Z,\bar{w},\tau)]=\overline{\cT}$. The
 maximum in (\ref{extr-2}) will not be changed if we replace the set $\cW$ by its relative  interior.  Since the relative  interior  of the set $\cW$ consists of   points with all nonzero coordinates,  the equality (\ref{extr-2}) follows. It follows that the set of saddle points of the minimax problem (\ref{kex-2}) is $\overline{\cW}\times \overline{\cT}$.

 \begin{theorem}
 \label{th-minmax}
 Suppose that $\R$ is of the form {\rm (\ref{kex-1})}, condition {\rm (A)} holds  and $\bbe_F[Z^2]<+\infty$. Then
 \begin{equation}
 \label{asex-1}
  \R(\hat{F}_N)=\sup_{w\in \overline{\cW}}  \inf_{\tau\in \overline{\cT}}\left\{w_0\bar{Z}+\sum_{i=1}^k w_i\left(\tau_i+\frac{1}{N (1-\alpha_i)}
 \sum_{j=1}^N[Z_j-\tau_i]_+\right)\right\}
 +o_p(N^{-1/2})
 \end{equation}
and
 \begin{equation}
 \label{asex-1a}
 N^{1/2}\big[\R(\hF_N)-\R(F)\big]\cdis \sup_{w\in \overline{\cW}}  \inf_{\tau\in \overline{\cT}}\bby(w,\tau),
 \end{equation}
 where $\bby(w,\tau)$  is a Gaussian process with mean zero and covariances
 \begin{equation}\label{covg}
 \bbe_F[\bby(w,\tau)\bby(w',\tau') ]=
 \cov_F\left(w_0Z+\sum_{i=1}^k \frac{w_i}{1-\alpha_i}
\left[Z-\tau_i\right]_+,w'_0Z+\sum_{i=1}^k \frac{w'_i}{1-\alpha_i}
\left[Z-\tau_i'\right]_+\right).
 \end{equation}
  Moreover, if the sets $\overline{\cW}=\{\bar{w}\}$ and  $ \overline{\cT}=\{\bar{\tau}\}$ are singletons, then $N^{1/2}\big[\R(\hF_N)-\R(F)\big]$ converges in distribution to normal $\N(0,\nu^2)$ with variance
 \begin{equation}\label{asex-2}
\nu^2=\var_F\left\{\bar{w}_0Z+\sum_{i=1}^k \frac{\bar{w}_i}{1-\alpha_i}
\left[Z- \bar{\tau}_i\right]_+\right\}.
 \end{equation}
 \end{theorem}

 \begin{proof}
 Consider function $\phi(Z,w,\tau)$
 defined in (\ref{fun-v}). Together with (\ref{kex-2})
 we have that
\begin{equation}\label{minm-2}
 \R(\hF_N)=\sup_{w\in \cW}  \inf_{\tau\in\bbr^k}N^{-1}\sum_{j=1}^N
 \phi(Z_j,w,\tau)=  \inf_{\tau\in\bbr^k}\sup_{w\in \cW}N^{-1}\sum_{j=1}^N
 \phi(Z_j,w,\tau).
 \end{equation}
The set $\overline{\cT}$ is nonempty and compact. We have  that the distance from a minimizer $\hat{\tau}_N$ in (\ref{minm-2}) to the set $\overline{\cT}$ tends to zero w.p.1 as $N\to +\infty$. Therefore as far as the  asymptotics is concerned,  the minimization  in $\tau$ in (\ref{minm-2})  can be reduced to a compact set $\cS\subset \bbr^k$   containing the set $\overline{\cT}$ in its interior. We can view $\hat{\phi}_N(w,\tau):=N^{-1}\sum_{j=1}^N
\phi(Z_j,w,\tau)$ as a random element of $C(\cW,\cS)$.
 Note that
 \[
 |\phi(Z,w,\tau)-\phi(Z,w',\tau')|\le C(Z)(\|w-w'\|+\|\tau- \tau'\|),
 \]
 where $C(\cdot)$ is a piecewise linear function. Hence it follows from the condition $\bbe[Z^2]<+\infty$ that $\bbe[C(Z)^2]<+\infty$. Consequently $N^{1/2}\left[\hat{\phi}_N(w,\tau)-
 \bbe_F[\phi(w,\tau,Z)]\right]$ converges in distribution (weakly) to a  random element of $C(\cW,\cS)$ with the respective  covariance structure of the
  Gaussian process  $\bby(w,\tau)$ (e.g., \cite[Example 19.7, p.271]{vaart}).

The minimax problem  (\ref{kex-2}) is convex in $\tau$ and concave (linear) in $w$, and  $\overline{\cW}\times \overline{\cT}$ is its set of saddle points.
Now proof can be completed by applying   a general result about asymptotics of minimax SAA problems (cf., \cite{sha08}, \cite[Section 5.1.4]{SDR}). \hfill\fb
 \end{proof}
 \\

Compared with the corresponding  results in \cite{pfl2010} and \cite{bel2011}, no assumptions about tail behavior of the distribution $F$ and uniqueness of the respective quantiles  were made in Theorem \ref{th-minmax}  apart from the assumption of existence of the second order moments. Also note that
\begin{equation}\label{covg-2}
 \cov_F\left([Z-t]_+,[Z-s]_+\right)=
 \int_{t}^{+\infty} \int_{s}^{+\infty}\big [F(x \wedge y)-F(x)F(y)\big]dxdy
\end{equation}
(see Lemma \ref{lm-eq}).

 \begin{cor}
\label{cor-avas}
Suppose that $\R$ is of the form {\rm (\ref{kex-1})}, condition {\rm (A)} holds  and $\bbe_F[Z^2]<+\infty$. Then
\begin{equation}\label{asd-2}
 N^{1/2}\left[\R(\hF_N)-
 \R(F)\right] \cdis \sup_{w\in \overline{\cW}}  \inf_{\tau\in\overline{\cT}}
 \int_{-\infty}^{+\infty}\kappa_{w,\tau} (z)
  \bb_F(z)dz,
\end{equation}
where
\begin{equation}\label{sigm}
 \kappa_{w,\tau}(z):=
 w_0+\sum_{i=1}^k(1-\alpha_i)^{-1}w_i
 \ind_{[\tau_i,\infty)}(z).
\end{equation}
\end{cor}

\begin{proof}
By Lemma \ref{lm-eq} we have that $\int_{-\infty}^{+\infty}\ind_{[\tau,\infty)} (z)\bb_F(z)dz$ has the same variance as $[Z-\tau]_+$. In a similar way it can be shown that the covariance structure of the process
$\int_{-\infty}^{+\infty}\kappa_{w,\tau} (z)
  \bb_F(z)dz$ is the same as the process $\bby(w,\tau)$ in Theorem \ref{th-minmax}. Hence (\ref{asd-2}) follows from (\ref{asex-1a}). \hfill\fb
\end{proof}
\\

Recall that $\bb_F(z)=\bb (F(z))$, where $\bb$ is the standard Brownian bridge corresponding to the uniform distribution  on the interval [0,1]. Hence
\[
\inf_{\tau_i\in \cQ_F(\alpha_i)}
\int_{-\infty}^{+\infty}\ind_{[\tau_i,\infty)}(z)
  \bb_F(z)dz= \left\{
 \begin{array}{ccc}
  \int_{b_i}^{\infty}
 \bb (F(z))dz& \mbox{if} & \bb(\alpha_i)>0,\vspace{1mm} \\
 \int_{a_i}^{\infty}\bb (F(z))dz& \mbox{if} & \bb(\alpha_i)\le 0,
 \end{array}\right.
\]
where $[a_i,b_i]=\cQ_F(\alpha_i)$. Therefore the asymptotics  (\ref{asd-2}) can be written as
\begin{equation}\label{asd-3}
 N^{1/2}\left[\R(\hF_N)-
 \R(F)\right] \cdis \sup_{w\in \overline{\cW}}  \int_{0}^{1}\sigma_{w} (z)
  \bb(F(z))dz,
\end{equation}
where
\begin{equation}\label{sigm-2}
 \sigma_{w}(t):=
 w_0+\sum_{i=1}^k(1-\alpha_i)^{-1}w_i
 \ind_{[\tau_i,\infty)}(z)
\end{equation}
with $\tau_i=b_i$ if $\bb(\alpha_i)>0$, and $\tau_i=a_i$ if $\bb(\alpha_i)\le 0$. Note that unless all quantile sets $\cQ_F(\alpha_i)$, $i=1,...,k$, are singletons,  $\sigma_{w}$  depends on $\bb$.

In particular, if all quantile sets $\cQ_F(\alpha_i)$, $i=1,...,k$, are singletons, then by making change of variables $z=F(t)$ we can write
\begin{equation}\label{asd-4}
 N^{1/2}\left[\R(\hF_N)-
 \R(F)\right] \cdis \sup_{\sigma\in \overline{\Upsilon}}  \int_{0}^{1}\sigma(t)\bb(t)dF^{-1}(t),
\end{equation}
where $\overline{\Upsilon}$ is the set of maximizers in the corresponding representation (\ref{asy-1}).

\subsection{Asymptotics of the optimization problem}

 Consider optimization   problem (\ref{ropt-1}) and its sample counterpart (\ref{ropt-2}).  Suppose that $\R$ is of the form (\ref{kex-1}),
 the set $\X$ is nonempty convex compact,  $G(x,\xi)$ is convex in $x$ for all $\xi\in \Xi$, and $\bbe|G_x|<+\infty$ for all $x\in \X$. It follows that functions $g(x)$ and $\hat{g}_N(x)$ are convex and finite valued,  and hence the respective optimization problems (\ref{ropt-1}) and (\ref{ropt-2}) are convex. Since $\R$ is of the form (\ref{kex-1}),  the optimal value
 $\vv_*$ of problem (\ref{ropt-1}) can be written as
 \begin{eqnarray}
\label{optas-1}
\vv_* &=& \inf_{x\in \X}\sup_{w\in \cW}\left\{w_0\bbe[G_x]+\sum_{i=1}^k w_i\avr_{\alpha_i}(G_x)\right\}\\
\label{optas-2}
  &=&\sup_{w\in \cW} \inf_{x\in \X}\left\{w_0\bbe[G_x]+\sum_{i=1}^k w_i\avr_{\alpha_i}(G_x)\right\}.
 \end{eqnarray}
Let $\overline{\X}$ and $\overline{\cW}$ be the sets of optimal solutions of problems (\ref{optas-1}) and   (\ref{optas-2}), respectively.

We also can write
\begin{eqnarray}
\label{kex-4}
\vv_*& =& \inf_{(x,\tau)\in \X\times \bbr^k}\sup_{w\in \cW}\bbe[\phi(G_x,w,\tau)]\\
\label{kex-4a}
&=& \sup_{w\in \cW}\inf_{(x,\tau)\in \X\times \bbr^k}    \bbe[\phi(G_x,w,\tau)],
\end{eqnarray}
where function $\phi(\cdot,\cdot,\cdot)$ is defined in (\ref{fun-v}).
Denote $\Y:=\X\times \bbr^k$ and let $\overline{\Y}\subset \Y$ be the set of optimal solutions of problem (\ref{kex-4}). Assuming  that condition (A) holds, the set $\overline{\Y}$ is nonempty and compact.
The minimax problem  (\ref{kex-4})--(\ref{kex-4a}) is convex in $(x,\tau)\in \Y$ and concave (linear) in $w\in \bbr^k$. The set of saddle points of this minimax problem is $\overline{\cW}\times \overline{\Y}$.
The SAA problem for \eqref{kex-4} writes
\begin{equation}\label{saarisk}
\hat \vv_N =
\inf_{(x,\tau)\in \X\times \bbr^k}\sup_{w\in \cW} \frac{1}{N} \sum_{j=1}^N \phi\left (G(x,\xi_j),w, \tau\right ).
\end{equation}

The following theorem can be proved in a way similar to the proof of   Theorem \ref{th-minmax}.

\begin{theorem}
 \label{th-minopt}
 Suppose that: {\rm (i)} $\R$ is of the form {\rm (\ref{kex-1})}, {\rm (ii)} the set $\X$ is convex and $G(x,\xi)$ is convex in $x$, {\rm (iii)}
 condition {\rm (A)} holds, {\rm (iv)} $\bbe[G_x^2]$ is finite for some $x\in \X$,
 {\rm (v)} there is a measurable function $C(\xi)$ such that $\bbe[C(\xi)^2]$ is finite and
 \begin{equation}\label{lips}
 |G(x,\xi)-G(x',\xi)|\le C(\xi)\|x-x'\|,\;\forall
x,x'\in \X,\;\forall \xi\in \Xi.
 \end{equation}
 Then
 \begin{equation}\label{kex}
 \hat{\vv}_N=\inf_{(x,\tau)\in \overline{\Y}}\sup_{w\in \overline{\cW}}\left\{
 \frac{w_0}{N}\sum_{j=1}^NG(x,\xi_j)
 +
 \sum_{i=1}^k w_i\left(\tau_i+ \frac{1}{N(1-\alpha_i)}
 \sum_{j=1}^N[G(x,\xi_j)-\tau_i]_+\right)
 \right\}+o_p(N^{-1/2}).
\end{equation}
  Moreover, if the sets $\overline{\cW}=\{w_*\}$ and   $\overline{\Y}=\{(x_*,\tau_*)\}$ are singletons, then $N^{1/2}\big(\hat{\vv}_N-\vv_*\big)$ converges in distribution to normal $\N(0,\nu_*^2)$ with variance
 \begin{equation}\label{asex-2v}
\nu_*^2:=\var\left[\phi  (G_{x_*} ,w_*, \tau_*  )\right]=
\var\left\{w_{*0} G_{x_*}+\sum_{i=1}^k \frac{w_{*i}}{1-\alpha_i}
\big[G_{x_*}-  \tau_{*i}\big]_+\right\}.
 \end{equation}
 \end{theorem}

Let us discuss now estimation of the variance $\nu_*^2$ given in (\ref{asex-2v}). Let $(\hat{x}_N,\hat{\tau}_N,\hat{w}_N)$ be a saddle point of the SAA problem (\ref{saarisk}). Suppose that the sets $\overline{\cW}=\{w_*\}$ and   $\overline{\Y}=\{(x_*,\tau_*)\}$ are singletons. Since the sets $\Y$ and $\cW$ are convex and   the function $\phi\left (G(x,\xi),w, \tau\right )$ is convex in $(x,\tau)$ and concave (linear) in $w$, it follows that $(\hat{x}_N,\hat{\tau}_N)$ converges w.p.1 to
$(x_*,\tau_*)$ and $\hat{w}_N$ converges w.p.1 to $w_*$ as $N\to\infty$ (e.g., \cite[Theorem 5.4]{SDR}). It follows that the variance $\nu_*^2$  can be consistently estimated by its sample counterpart, i.e., the estimator
\begin{equation}\label{varsaa}
\hat{\nu}^2_N=\frac{1}{N-1}\sum_{j=1}^N
\left[\phi  (G(\hat{x}_N,\xi_j) ,\hat{w}_N, \hat{\tau}_N)-\frac{1}{N}\sum_{j=1}^N \phi  (G(\hat{x}_N,\xi_j) ,\hat{w}_N, \hat{\tau}_N)\right]^2
\end{equation}
converges w.p.1 to $\nu_*^2$.
 Then employing Slutsky's theorem we obtain  that under the assumptions of Theorem \ref{th-minopt}, it follows that
\begin{equation}\label{conv1estimatedvariance}
\frac{N^{1/2} \big( \hat{\vv}_N - \vv_* \big) }  {\hat \nu_N} \cdis \mathcal{N}(0,1).
\end{equation}

\section{Hypotheses testing} \label{testssection}

On the basis of samples $\xi^{N,i}=(\xi_{1}^i, \ldots, \xi_{N}^i)$ of $\xi^i$ for $i=1,\ldots,K$,
we propose nonasymptotic rejection regions for tests \eqref{deftest1} and \eqref{deftest2} (in Section \ref{natest})
and asymptotic rejection regions for tests \eqref{deftest1}, \eqref{testcontregalite}, and \eqref{deftest2}
(in Section \ref{astest}).
For the nonasymptotic tests, we show that the probability of type II error can be controlled under some assumptions.
We will denote by $0<\beta<1$ the maximal type I error.

\subsection{Nonasymptotic tests} \label{natest}

\subsubsection{Risk-neutral case}

Let us consider $K \geq 2$ optimization problems of the form \eqref{ropt-1} with 
$\mathcal{R}:=\mathbb{E}$ the expectation.
In this situation, several papers have derived nonasymptotic confidence intervals on the optimal value of
\eqref{ropt-1}: \cite{pflug99} using Talagrand
inequality (\cite{talagrand1994}, \cite{talagrand1996}),
\cite{shap2000}, \cite{ioudnemgui15} using large-deviation type results,
and \cite{nemjudlannem09}, \cite{nemlansh09}, \cite{guigues2016} using
Robust Stochastic Approximation (RSA) \cite{polyak90}, \cite{polyakjud92}, Stochastic
Mirror Descent (SMD) \cite{nemjudlannem09} and variants of SMD. In all cases, the confidence interval depends on a sample
$\xi^N=(\xi_1,\ldots,\xi_N)$ of $\xi$ and of parameters. For instance,
the confidence interval $[{\tt{Low}}(\Theta_2, \Theta_3, N), {\tt{Up}}(\Theta_1, N)]$ with confidence level
$1-\beta$
from \cite{guigues2016}
obtained using RSA
depends on parameters $\Theta_1=2\sqrt{\ln(2/\beta)}$, $\Theta_3=2\sqrt{\ln(4/\beta)}$,
$\Theta_2$ satisfying $e^{1-\Theta_2^2} + e^{-\Theta_2^2/4}=\frac{\beta}{4}$, and
$L, M_1, M_2, D(\X )$ with $D( \X )$ the maximal Euclidean distance in $\X$ to $x_1$ (the initial point of the RSA algorithm),
$L$ a uniform upper bound on $\X$ on the $\|\cdot\|_2$-norm of some selection (say, selection $g'(x) \in \partial g(x)$ at $x$) of subgradients of $g$,
and $M_1, M_2<+\infty$ such that for all $x\in \X$ it holds
\begin{equation}\label{asss}
\begin{array}{lrcl}
(a)&\mathbb{E}\Big[ (G(x, \xi) - g(x)   )^2 \Big] &\leq& M_1^2, \vspace*{0.1cm} \\
(b)&\mathbb{E} \Big[ \| G_x'(x, \xi) - \mathbb{E}[ G_x'(x, \xi) ]  \|_2^2  \Big] &\leq & M_2^2,
\end{array}
\end{equation}
for some selection $G_x'(x, \xi)$ belonging to the subdifferential $\partial_x G(x, \xi)$.

With this notation, on the basis of a sample
$\xi^N=(\xi_1, \ldots,\xi_N)$
of size $N$ of $\xi$ and of the trajectory $x_1, \dots,x_N$ of the RSA
algorithm, setting
\begin{equation}\label{defab}
a(\Theta, N) =\frac{\Theta M_1}{\sqrt{N}} \mbox{ and }b(\Theta, \X, N)=\frac{K_1( \X ) + \Theta(K_2( \X )-M_1)}{\sqrt{N}},
\end{equation}
where the constants $K_1( \X )$ and $K_2( \X )$ are given by
$$
K_1( \X )=\frac{D( \X ) (M_2^2+2L^2)}{\sqrt{2 (M_2^2 + L^2 )}}
\mbox{ and }K_2( \X )=\frac{D( \X ) M_2^2}{\sqrt{2 (M_2^2 + L^2)}}+2 D( \X ) M_2 + M_1,
$$
the lower bound
${\tt{Low}}(\Theta_2, \Theta_3, N)$ is
\begin{equation} \label{lowsmd1}
{\tt{Low}}(\Theta_2, \Theta_3, N) = \frac{1}{N} \sum_{t=1}^N G(x_t, \xi_t)  - b(\Theta_2, \X, N)-a(\Theta_3, N),
\end{equation}
and the upper bound ${\tt{Up}}(\Theta_1, N)$ is
\begin{equation}\label{upsmd1}
{\tt{Up}}(\Theta_1, N) = \frac{1}{N} \sum_{t=1}^N G(x_t, \xi_t) + a(\Theta_1, N).
\end{equation}
More precisely, we have $\mathbb{P}(\vv_* < {\tt{Low}}(\Theta_2, \Theta_3, N)) \leq \beta/2$
and $\mathbb{P}(\vv_* > {\tt{Up}}(\Theta_1, N) \leq \beta/2$.\\

\par {\textbf{Test \eqref{deftest1}-(a).}} Using these bounds ${\tt{Low}}$ and ${\tt{Up}}$ or one of the aforementioned cited procedures, we can determine for optimization problem $i \in \{1,\ldots,K\}$
(stochastic) lower and upper bounds on $\vv_*^i$ that we will denote by ${\tt{Low}}_i$ and  ${\tt{Up}}_i$ respectively for short,
such that $\mathcal{P}(\vv_*^i < {\tt{Low}}_i ) \leq \beta/2K$
and $\mathcal{P}(\vv_*^i > {\tt{Up}}_i ) \leq \beta/2K$.

We define for test \eqref{deftest1}-(a) the rejection region $\mathcal{W}_{\eqref{deftest1}-(a)}$ to be the set of samples
such that the realizations of the confidence intervals $\Big[{\tt{Low}}_i, {\tt{Up}}_i  \Big], i=1, \ldots,K$, on the optimal values have no intersection, i.e.,
$$
\begin{array}{lll}
\mathcal{W}_{\eqref{deftest1}-(a)} & = & \left\{(\xi^{N,1}, \ldots, \xi^{N,K}) \;:\; \displaystyle \bigcap_{i=1}^{K} \;\Big[ {\tt{Low}}_i, {\tt{Up}}_i  \Big] = \emptyset \right\} \\& = &
\left\{(\xi^{N,1}, \ldots, \xi^{N,K}) \;:\; \;\displaystyle \max_{i=1,\ldots,K}\; {\tt{Low}}_i  > \displaystyle \min_{i=1,\ldots,K}\; {\tt{Up}}_i  \right\}.
\end{array}
$$
If $H_0$ holds, denoting $\vv_{*} = \vv_{*}^1= \vv_{*}^2 = \ldots = \vv_{*}^K$, we have
$$
\begin{array}{l}
\mathbb{P}\Big(\displaystyle \max_{i=1,\ldots,K}\; {\tt{Low}}_i  >\displaystyle \min_{i=1,\ldots,K}\; {\tt{Up}}_i     \Big)=
\mathbb{P}\left(\displaystyle  \max_{i=1,\ldots,K}\; \Big[ {\tt{Low}}_i - \vv_{*}    \Big]+\displaystyle \max_{i=1,\ldots,K}\; \Big[\vv_{*}-{\tt{Up}}_i   \Big] >0 \right)\\
\leq \sum_{i=1}^K \left[ \mathbb{P}\Big( {\tt{Low}}_i - \vv_{*}^i  >0 \Big) + \mathbb{P}\Big( \vv_{*}^i -{\tt{Up}}_i >0\Big) \right]
  \leq  \beta
\end{array}
$$
and $\mathcal{W}_{\eqref{deftest1}-(a)}$ is a rejection region for \eqref{deftest1}-(a) yielding a type I error of at most $\beta$.
Moreover, as stated in the following lemma, if $H_0$ does not hold and if two optimal values are sufficiently distant then
the probability to accept
$H_0$ will be small:
\begin{lemma}\label{tIItest1} Consider test \eqref{deftest1}-(a) with rejection region $\mathcal{W}_{\eqref{deftest1}-(a)}$.
If for some $i, j \in \{1,\ldots,K\}$ with $i \neq j$, we have almost surely
$
\vv_{*}^i > \vv_{*}^j  + {\tt{Up}}_i - {\tt{Low}}_i + {\tt{Up}}_j -  {\tt{Low}}_j
$
then the probability to accept $H_0$ is not larger than  $\frac{\beta}{K}$.
\end{lemma}
\begin{proof} We first check that
\begin{equation}\label{eq1typeIItest}
\left\{
\begin{array}{ll}
\vv_{*}^i > \vv_{*}^j + {\tt{Up}}_j - {\tt{Low}}_j + {\tt{Up}}_i - {\tt{Low}}_i  & (a)\\
{\tt{Low}}_j \leq \vv_{*}^j  & (b)\\
\vv_{*}^i \leq {\tt{Up}}_i & (c)
\end{array}
\right\}
\Rightarrow
{\tt{Up}}_j  < {\tt{Low}}_i.
\end{equation}
Indeed, if \eqref{eq1typeIItest}-(a), (b), and (c) hold then
$$
{\tt{Up}}_j = {\tt{Low}}_j + {\tt{Up}}_j - {\tt{Low}}_j
\stackrel{\eqref{eq1typeIItest}-(b)}{\leq} \vv_{*}^j + {\tt{Up}}_j  - {\tt{Low}}_j
\stackrel{\eqref{eq1typeIItest}-(a)}{<} \vv_{*}^i + {\tt{Low}}_i - {\tt{Up}}_i \stackrel{\eqref{eq1typeIItest}-(c)}{\leq} {\tt{Low}}_i.
$$
Assume now that $\vv_{*}^i > \vv_{*}^j + {\tt{Up}}_j - {\tt{Low}}_j + {\tt{Up}}_i - {\tt{Low}}_i$.
Since
${\tt{Up}}_j  < {\tt{Low}}_i$ implies that $H_0$ is rejected, we get
$$
\begin{array}{lcl}
\mathbb{P}\Big(\mbox{reject }H_0\Big)& \geq  &  \mathbb{P}\Big( {\tt{Up}}_j  < {\tt{Low}}_i \Big) \stackrel{\eqref{eq1typeIItest}}{\geq}    \mathbb{P}\Big( \Big\{ {\tt{Low}}_j \leq \vv_{*}^j \Big\} \bigcap \Big\{\vv_{*}^i \leq {\tt{Up}}_i \Big\} \Big) \\
& \geq  &  \mathbb{P}\Big( {\tt{Low}}_j \leq \vv_{*}^j \Big) + \mathbb{P}\Big( \vv_{*}^i \leq {\tt{Up}}_i  \Big) -1 \geq 1-\frac{\beta}{K}
\end{array}
$$
which achieves the proof of the lemma.\hfill \fb\\
\end{proof}

\par {\textbf{Test \eqref{deftest1}-(b).}} We now consider the test
$$
H_0^i :\; \vv_{*}^i \leq  \vv_{*}^j  \mbox{ for }1 \leq j \neq i \leq K  \mbox{ against unrestricted }{H_1^i}.
$$
Let $[{\tt{Low}}_i, {\tt{Up}}_i]$ be a confidence interval with confidence level at least $1-\beta/2(K-1)$ for problem $i$:
\begin{equation}\label{testbconfi}
\bbp(\vv_*^i < {\tt{Low}}_i ) \leq \beta/2(K-1)\mbox{ and }\bbp(\vv_*^i > {\tt{Up}}_i ) \leq \beta/2(K-1).
\end{equation}
We define for test \eqref{deftest1}-(b) the rejection region
$$
\begin{array}{lll}
\mathcal{W}_{\eqref{deftest1}-(b)} & = & \left\{(\xi^{N,1}, \ldots, \xi^{N,K}) \;:\; \displaystyle \exists   1 \leq j \neq i  \leq K \mbox{ such that }{\tt{Low}}_i>{\tt{Up}}_j\right\}.
\end{array}
$$
If $H_0$ holds, we have
$$
\begin{array}{l}
\bbp\Big(\displaystyle \exists  \;1 \leq j \neq i  \leq K\;:\;  {\tt{Low}}_i>{\tt{Up}}_j \Big) \leq \sum_{1 \leq j \neq i \leq K}\; \mathbb{P}\Big( {\tt{Low}}_i>{\tt{Up}}_j \Big) \\
\leq  \sum_{1 \leq j \neq i \leq K}\; \mathbb{P}\Big({\tt{Low}}_i- \vv_{*}^i + \vv_{*}^j - {\tt{Up}}_j>0 \Big)\\
\leq \sum_{1 \leq j \neq i \leq K}\; \left( \mathbb{P}\Big( {\tt{Low}}_i- \vv_{*}^i >0  \Big) + \mathbb{P}\Big( \vv_{*}^j - {\tt{Up}}_j > 0\Big)   \right)
\leq  \beta
\end{array}
$$
and $\mathcal{W}_{\eqref{deftest1}-(b)}$ is a rejection region for \eqref{deftest1}-(b) yielding a type I error of at most $\beta$.
We also have an analog of Lemma \ref{tIItest1}:
\begin{lemma}\label{tIItest2} Consider test \eqref{deftest1}-(b) with rejection region $\mathcal{W}_{\eqref{deftest1}-(b)}$.
If for some $j \in \{1,\ldots,K\}$ with $i \neq j$, we have almost surely
$\vv_{*}^i > \vv_{*}^j + {\tt{Up}}_i - {\tt{Low}}_i + {\tt{Up}}_j - {\tt{Low}}_j$
then the probability to accept $H_0$ is not larger than $\frac{\beta}{K-1}$.
\end{lemma}
\begin{proof} The proof is analogue to the proof of Lemma \ref{tIItest1}.\hfill \fb\\
\end{proof}
\par {\textbf{Test \eqref{deftest1}-(c).}}
Consider test \eqref{deftest1}-(c):
$$
\begin{array}{ll}
H_0 :\; \vv_{*}^1 \leq \vv_{*}^2  \leq  \ldots \leq  \vv_{*}^K \hspace*{-0.3cm}  & \mbox{ against unrestricted }{H_1}.
\end{array}
$$
Let $[{\tt{Low}}_i, {\tt{Up}}_i]$  be a confidence interval on $\vv_*^i$ satisfying \eqref{testbconfi}.
We define the rejection region
$$
\begin{array}{lll}
\mathcal{W}_{\eqref{deftest1}-(c)} & = & \left\{(\xi^{N,1}, \ldots, \xi^{N,K}) \;:\; \displaystyle \exists i \in \{1,\ldots,K-1\} \mbox{ such that }  {\tt{Low}}_i>{\tt{Up}}_{i+1}\right\}.
\end{array}
$$
If $H_0$ holds, we have
$$
\begin{array}{l}
\mathbb{P}\Big(\displaystyle \exists i \in \{1,\ldots,K-1\} \;:\; {\tt{Low}}_i>{\tt{Up}}_{i+1} \Big) \leq \sum_{i=1}^{K-1} \; \mathbb{P}\Big( {\tt{Low}}_i -\vv_{*}^i + \vv_{*}^{i+1} - {\tt{Up}}_{i+1} > 0\Big)\\
\leq \displaystyle \sum_{i=1}^{K-1} \;  \mathbb{P}\Big( {\tt{Low}}_i -\vv_{*}^i >0 \Big)  + \mathbb{P}\Big( \vv_{*}^{i+1} - {\tt{Up}}_{i+1} > 0 \Big) \leq \beta
\end{array}
$$
and $\mathcal{W}_{\eqref{deftest1}-(c)}$ is a rejection region for \eqref{deftest1}-(c) yielding a type I error of at most $\beta$.
As for test \eqref{deftest1}-(a), we can bound from above the probability of type II error under some assumptions:
\begin{lemma}\label{tIItest3} Consider test \eqref{deftest1}-(c) with rejection region $\mathcal{W}_{\eqref{deftest1}-(c)}$.
If for some $i \in \{1,\ldots,K-1\}$ we have almost surely
$\vv_{*}^i > \vv_{*}^{i+1} + {\tt{Up}}_i - {\tt{Low}}_i + {\tt{Up}}_{i+1} - {\tt{Low}}_{i+1}$
then the probability to accept $H_0$ is not larger than $\frac{\beta}{K-1}$.
\end{lemma}
\begin{proof} The proof is analogue to the proof of Lemma \ref{tIItest1}.\hfill \fb\\
\end{proof}

\begin{rem} Though ${\tt{Low}}$ and ${\tt{Up}}$ are stochastic,
for bounds \eqref{lowsmd1} and \eqref{upsmd1}, the difference
${\tt{Up}}$-${\tt{Low}}=a(\Theta_1, N) + b(\Theta_2, \X, N)+a(\Theta_3, N)$ is deterministic
and inequality
$\vv_{*}^i > \vv_{*}^j  + {\tt{Up}}_i - {\tt{Low}}_i + {\tt{Up}}_j -  {\tt{Low}}_j $
in Lemmas \ref{tIItest1} and \ref{tIItest2} is deterministic too.
With this choice of lower and upper bounds, inequality
$\vv_{*}^i > \vv_{*}^{i+1} + {\tt{Up}}_i - {\tt{Low}}_i + {\tt{Up}}_{i+1} - {\tt{Low}}_{i+1}$
in Lemma \ref{tIItest3} is also deterministic.
\end{rem}

\par {\textbf{Tests \eqref{deftest2}.}} For tests
\begin{equation}\label{testVrho0}
\begin{array}{lll}
H_0:\;\vv_{*}=\rho_0  \hspace*{-0.3cm}  & \mbox{ against }{H_1}: \;\vv_{*} \neq \rho_0 & (a)\\
H_0:\; \vv_{*}\leq \rho_0 \hspace*{-0.3cm}  & \mbox{ against }{H_1}: \;\vv_{*} > \rho_0 & (b)\\
H_0:\; \vv_{*}\geq \rho_0 \hspace*{-0.3cm}  & \mbox{ against }{H_1}: \;\vv_{*}< \rho_0, & (c)
\end{array}
\end{equation}
we define rejection regions which are respectively of the form
$$
\begin{array}{lll}
\mathcal{W}_{\eqref{testVrho0}-(a)} & = & \left\{(\xi_{1}, \ldots, \xi_{N}) \;:\; \Big\{ \rho_0 > {\tt{Up}} \Big\}  \bigcup \Big\{ \rho_0 < {\tt{Low}} \Big\}  \right\},\\
\mathcal{W}_{\eqref{testVrho0}-(b)} & = & \left\{(\xi_{1}, \ldots, \xi_{N}) \;:\;  \rho_0 < {\tt{Low}} \right\},\\
\mathcal{W}_{\eqref{testVrho0}-(c)} & =&  \left\{(\xi_{1}, \ldots, \xi_{N}) \;:\; \rho_0 > {\tt{Up}}  \right\}.
\end{array}
$$
To ensure a type I error of at most $\beta$, the confidence interval
$[{\tt{Low}}, {\tt{Up}}]$ on $\vv_*$ satisfies
(i) $\mathbb{P}(\vv_* > {\tt{Up}} ) \leq \beta/2$ and $\mathbb{P}(\vv_* < {\tt{Low}} ) \leq \beta/2$ for test \eqref{testVrho0}-(a), (ii)
$\mathbb{P}(\vv_* < {\tt{Low}} ) \leq \beta$ for test \eqref{testVrho0}-(b), and (iii) $\mathbb{P}(\vv_* > {\tt{Up}} ) \leq \beta$ for
test \eqref{testVrho0}-(c).

\subsubsection{Risk averse case}

Consider $K \geq 2$ optimization problems of the form \eqref{ropt-1}.
For such problems, nonasymptotic confidence intervals $[{\tt{Low}}, {\tt{Up}}]$ on the optimal value $\vv_*$ were derived in
\cite{guigues2016} and  \cite{nemlansh09} using RSA and SMD,
taking for $\mathcal{R}$ an extended polyhedral risk measure (introduced in \cite{guiguesromisch10}) in
\cite{guigues2016} and $\mathcal{R}=\avr_\alpha$ and $G(x,\xi)=\xi^T x$
in \cite{nemlansh09}. With such confidence intervals at hand, we can use the developments of the previous section for testing hypotheses \eqref{deftest1} and \eqref{deftest2}.
However, the analysis in \cite{guigues2016} assumes  boundedness of the feasible set of the optimization problem
defining the risk measure; an assumption that can be enforced  for risk measure
$\mathcal{R}$ given by \eqref{secop-1}. We provide in this situation formulas for the constants $L, M_1,$ and $M_2$ defined in the previous section,
necessary to compute the bounds from \cite{guigues2016}.
These constants are slighlty refined versions of the constants given in
Section 4.2 of \cite{nemlansh09} for the special case $\mathcal{R}=\avr_\alpha$ and $G(x,\xi)=\xi^T x$.

We   assume here  that the set  $\Xi$ is compact, $G(\cdot,\cdot)$ is continuous, for every $x \in \X$ the distribution of $G_x$ is continuous, and that
that the set $\cW=\{w\}$ is a singleton  i.e.,  
\begin{equation}\label{secop-1}
 \R(Z)= w_0\bbe[Z]+\sum_{i=1}^k w_i\avr_{\alpha_i}(Z)
\end{equation}
for some $w\in \Delta_{k+1}$.
Consequently  problem \eqref{ropt-1} can be written
as
\begin{equation}\label{secop-1a} 
\vv_* = \inf_{(x,\tau)\in \X\times \mathbb{R}^k }\big\{   \bbe[\phi(G_x,\tau)]=\bbe[H(x, \tau, \xi)]\big\},
\end{equation}
where $\phi(G_x,\tau)$ is defined in (\ref{fun-v}), with vector $w$ omitted, and
\[
H(x, \tau, \xi):=w_0 G(x, \xi) + \sum_{i=1}^k w_i\left (\tau_i +  \frac{1}{1-\alpha_i}[G(x,\xi) - \tau_i ]_{+}\right).
\]
For a given $x\in \X$ the minimum in (\ref{secop-1a}) is attained at ${\tau}_i=F_x^{-1}(\alpha_i)$, $i=1,...,k$, where $F_x$ is the cdf of $G_x$. 
Therefore, using the lower and upper bounds from \cite{nemlansh09} for the quantile of a continuous distribution with finite mean and variance,
we can restrict   $\tau$ to   compact set  $\T=[{\ubar{\tau}}, {\bar \tau}]\subset \bbr^k$ where
\begin{equation}\label{formulastaubar}
\begin{array}{lll}
{\ubar{\tau}}_i & = & \min_{x \in \X} \bbe[G_{x}]-\sqrt{\frac{1-\alpha_i}{\alpha_i}} \sqrt{ \max_{x \in \X} \mbox{Var}(G_{x}) },  \\
\bar \tau_i & = & \max_{x \in \X} \bbe[G_{x}]+\sqrt{\frac{\alpha_i}{1-\alpha_i}} \sqrt{ \max_{x \in \X} \mbox{Var}(G_{x}) },
\end{array}
\end{equation}
for $i=1,\ldots,k$. This implies  that we can take for $D(\X\times \T)$
the quantity $\sqrt{D(\X)^2 + \|{\bar \tau} - {\ubar \tau}\|_2^2 }$.

\par {\textbf{Computation of $M_1$.}}
Setting 
\[
M_0:=\max_{(x, \xi) \in \X \times \Xi} G(x,\xi)\;{\rm and}\;  m_0 :=\min_{(x, \xi) \in \X \times \Xi} G(x,\xi),
\]
we have
for $(x,\tau)\in \X\times \T$ that $|G_{x} - \bbe[G_{x}]| \leq M_0 - m_0$
and $|[G_{x}- \tau_i]_{+}  - \bbe[G_{x}- \tau_i]_{+}| \leq M_0 - {\ubar{\tau}}_i$
which implies that almost surely
$$
|\phi(G_x,\tau) - \bbe[\phi(G_x,\tau)]| \leq M_1 :=w_0 (M_0 - m_0 ) + \sum_{i=1}^k \frac{w_i}{1-\alpha_i}(M_0 - {\ubar \tau}_i ).
$$

\par {\textbf{Computation of $M_2$ and $L$.}} We have $H_{x, \tau}'(x, \tau, \xi)=[H_{x}'(x, \tau, \xi); H_{\tau}'(x, \tau, \xi)]$ with
$$
\begin{array}{lll}
H_{x}'(x, \tau, \xi)&=& w_0 G_{x}'(x, \xi) + \sum_{i=1}^k \frac{w_i}{1-\alpha_i} G_{x}'(x, \xi) \ind_{G(x, \xi) \geq \tau_i},\\
H_{\tau}'(x, \tau, \xi)&=& (w_i(1-\frac{1}{1-\alpha_i} \ind_{G(x, \xi) \geq \tau_i}))_{i=1,\ldots,k}.
\end{array}
$$
We assume that for every $x \in \X$, the stochastic subgradients $G_{x}'(x, \xi)$ are almost surely bounded and we denote
by ${\underline{m}}$ and ${\overline{M}}$ vectors such that almost surely ${\underline{m}} \leq G_{x}'(x, \xi) \leq {\overline{M}}$.
Then for $(x,\tau)\in \X\times \T$, setting
$$
\begin{array}{l}
a_i = w_0 {\overline M}_i + \sum_{i=1}^k \frac{w_i}{1-\alpha_i}\max(0,{\overline M}_i) \mbox{ and }b_i = w_0 {\underline m}_i + \sum_{i=1}^k \frac{w_i}{1-\alpha_i}\min(0,{\underline m}_i)
\end{array}
$$
we have
$$
\begin{array}{l}
\|\bbe[H_{x, \tau}'(x, \tau, \xi)]\|_2^2  \leq  L^2:=\sum_{i=1}^m \max(a_i^2, b_i^2)  + \sum_{i=1}^k  w_i^2 \max  \left(1,\frac{\alpha_i^2}{(1 - \alpha_i)^2}\right) ,\\
\bbe \|H_{x, \tau}'(x, \tau, \xi) - \bbe[H_{x, \tau}'(x, \tau, \xi)] \|_2^2  \leq  M_2^2 := \sum_{i=1}^m (a_i  - b_i )^2  + \sum_{i=1}^k  \left( \frac{ w_i }{1 - \alpha_i } \right)^2 .
\end{array}
$$
In some cases, the above formulas for ${\bar \tau}, {\ubar \tau}, L, M_1,$ and $M_2$ can be simplified:
\begin{example}
{\rm
Let $k=1$ in \eqref{secop-1} and $G(x, \xi)=\xi^T x$ where $\xi$ is a random vector with mean $\mu$
and covariance matrix $\Sigma$. In this case $\min_{x \in \X} \bbe[G_{x}]$ and $\max_{x \in \X} \bbe[G_{x}]$ are convex
optimization problems with linear objective functions
and denoting by $U_1$ the quantity  $\max_{x \in \X} \|x\|_1$ or an upper bound  on this quantity,
we can replace $\max_{x \in \X} \mbox{Var}(G_{x}) $ by $U_1^2 \max_i \Sigma(i,i)$ in the expressions of ${\ubar{\tau}}_i$ and ${\bar \tau}_i$.
Computing $M_0$ and $m_0$ also amounts to solve convex optimization problems with linear objective.
Assume also that almost surely $\|\xi\|_{\infty} \leq U_2$ for some $0<U_2<+\infty$. We  have
$|G_x - \bbe[G_x]| \leq 2 U_1 U_2 $ and $|[G_x -\tau ]_{+} - \bbe[G_x -\tau ]_{+}| \leq U_1 U_2 - {\ubar \tau}$
which shows that we can take $M_1=2w_0 U_1 U_2 + \frac{w_1}{1-\alpha_1}(  U_1 U_2 - {\ubar \tau}  )$.
We have
$\bbe[  H_{\tau}'(  x, \tau, \xi   )  ] = w_1 ( 1 - \frac{\mathbb{P}( \xi^T x \geq \tau )}{1-\alpha_1})$
so that $|\bbe[  H_{\tau}'(  x, \tau, \xi   )  ]| \leq w_1 \max(1, \frac{\alpha_1}{1-\alpha_1})$
and $\|\bbe[ H_{x}'(  x, \tau, \xi   )   ]\|_2^2 \leq n (w_0 + \frac{w_1}{1-\alpha_1})^2 U_2^2$, i.e., we can take
$L^2= w_1^2 \max(1, \frac{\alpha_1^2}{(1-\alpha_1)^2}) + n (w_0 + \frac{w_1}{1-\alpha_1})^2 U_2^2$.
Next, for all $\xi_0 \in \Xi$ we have
$$
\begin{array}{lll}
|H_{\tau}'(x, \tau, \xi_0 ) - \bbe[  H_{\tau}'(x, \tau, \xi ) ]|& = &\frac{w_1 (1-\mathbb{P}(\xi^T x \geq \tau))}{1-\alpha_1} \mbox{ if }\xi_0^T x \geq \tau,\\
& = & \frac{w_1  \mathbb{P}(\xi^T x \geq \tau)  }{1-\alpha_1} \mbox{ otherwise},
\end{array}
$$
implying that $|H_{\tau}'(x, \tau, \xi_0 ) - \bbe[  H_{\tau}'(x, \tau, \xi ) ]| \leq \frac{w_1}{1-\alpha_1}$.
Since $\|H_{x}'(x, \tau, \xi_0 ) - \bbe[  H_{x}'(x, \tau, \xi ) ] \|_{\infty} \leq 2(w_0 + \frac{w_1}{1-\alpha_1}) U_2$, we can take
$M_2^2 = \frac{w_1^2}{(1-\alpha_1)^2} + 4 n(w_0 + \frac{w_1}{1-\alpha_1})^2 U_2^2$.
In the special case when $\X =\{ x_* \}$ is a singleton, denoting $\eta = \xi^T x_*$, we have $\vv_* = \mathcal{R}( \eta )$ and the above computations show that we can take
\begin{equation}\label{formulaLM1M2simplecase}
\begin{array}{l}
L=w_1 \max(1, \frac{\alpha_1}{1-\alpha_1}), M_1= w_0 (b_0 -a_0 ) + \frac{w_1}{1-\alpha_1}(b_0 - {\ubar \tau}   ), \mbox{ and }M_2=\frac{w_1}{1-\alpha_1},
\end{array}
\end{equation}
where $\ubar \tau = \mathbb{E}[\eta] - \sqrt{\frac{1-\alpha_1}{\alpha_1}}\sqrt{\mbox{Var}(\eta)}$ with $a_0, b_0$ satisfying $ a_0 \leq \eta \leq b_0$ almost surely.
}
\end{example}

Finally, note that the nonasymptotic tests of this and the previous section do not require the independence of $\xi^{N, 1},\ldots,\xi^{N, K}$
and are valid for any sample size $N$.
However, they use conservative confidence bounds and rejection regions meaning that they can lead to large probabilities of type II errors.
The asymptotic tests to be presented in the next section are valid as the sample size tends to infinity but work well
in practice for small sample sizes ($N=20$) for problems of small to moderate size ($n$ up to $500$); see the numerical simulations of
Section \ref{numsim}.

\subsection{Asymptotic tests} \label{astest}

\par {\textbf{Test \eqref{deftest2}.}} Consider optimization problem \eqref{ropt-1} and the
SAA approximation
$\hat{\vv}_N$ of its optimal value
$\vv_*$ obtained using a sample
$(\xi_1, \ldots, \xi_N)$ of $\xi$. Let also $\hat \nu_N^2$ be the empirical estimator (\ref{varsaa})
of the variance \eqref{asex-2v}.
Under the assumptions of   Theorem \ref{th-minopt}, we have the asymptotics (\ref{conv1estimatedvariance}). 
Therefore  for $N$ large, we can approximate the distribution of $\frac{N^{1/2}\left(\hat{\vv}_N  - \vv_* \right)}{  {\hat \nu}_{N}}$
by the standard normal $\mathcal{N}(0,1)$.

It follows that for tests \eqref{deftest2}-(a) and \eqref{deftest2}-(b), we obtain respectively  the asymptotic
rejection regions
$$
\begin{array}{l}
\mathcal{W}_{\eqref{deftest2}-(a)}^{As} = \left\{(\xi_1,\ldots,\xi_N) \;:\;|\hat{\vv}_N-\rho_0|>\frac{ {\hat \nu}_{N} }{\sqrt{N}} \Phi^{-1}(1-\frac{\beta}{2}) \right\} \mbox{ and }\\
\mathcal{W}_{\eqref{deftest2}-(b)}^{As} = \left\{(\xi_1,\ldots,\xi_N) \;:\;\hat{\vv}_N>\rho_0+ \frac{ {\hat \nu}_{N}    }{\sqrt{N}} \Phi^{-1}(1-\beta) \right\},
\end{array}
$$
where $\Phi( \cdot )$ is the cumulative distribution function of the standard normal distribution.\\

\par {\textbf{Tests \eqref{deftest1} and \eqref{testcontregalite}.}}
Let us now consider $K>1$, optimization problems of the form \eqref{ropt-1}
with $\xi$, $g$, and $\X$ respectively replaced by $\xi^i$, $g_{i}$, and $\X_i$ for problem $i$.
For $i=1,\ldots,K$,
let $(\xi^i_1,\ldots,\xi^i_{N})$ be a sample from the distribution
of $\xi^i$, let
$\vv_*^i$ be the optimal value of problem $i$ and $z_*^i=(x_*^i,\tau_*^i)$ the optimal solution.
Let also $\hat{\vv}_{N}^i$ be the SAA estimator  of the optimal
value for problem $i=1,...,K$, and ${\hat{\nu}}_N^i$ be
the empirical estimator 
of the variance $\mbox{Var}[H_i(z_*^i , \xi_i)]$ based  on the  sample for problem
$i$. We assume that the samples are i.i.d. and that
$\xi^{N, 1}, \ldots, \xi^{N, K}$ are independent.
Under the assumptions of Theorem \ref{th-minopt}    for $N$ large  we can approximate the distribution of $\frac{N^{1/2}\left(\hat{\vv}_N^i  - \vv_*^i \right)}{  {\hat \nu}_{N}^i }$
by the standard normal $\mathcal{N}(0,1)$.

Let us first consider the statistical tests \eqref{deftest1}-(a) and \eqref{deftest1}-(b)  with $K=2$:
$$
\begin{array}{ll}
H_0:\; \vv_*^1 = \vv_*^2 \hspace*{-0.3cm}  & \mbox{ against }{H_1}: \;\vv_*^1 \neq \vv_*^2 \\
H_0:\; \vv_*^1 \leq  \vv_*^2 \hspace*{-0.3cm}  & \mbox{ against }{H_1}: \;\vv_*^1 > \vv_*^2.
\end{array}
$$
For $N$ large, we approximate  the distribution of
$
\frac{ ({\hat{\vv}}_{N}^1 - {\hat{\vv}}_{N}^2 ) - (\vv_{*}^1 - \vv_{*}^2)}{ \sqrt{\frac{ ({\hat{\nu}}_{N}^1  )^2 }{N}+ \frac{ ({\hat{\nu}}_{N}^2 )^2 }{N}  }}
$
by the standard normal $\mathcal{N}(0,1)$  and
we obtain the rejection regions
\begin{equation}\label{rejectionregionasK2}
\begin{array}{l}
\left\{(\xi_{N}^1, \xi_{N}^2) \;:\;|{\hat{\vv}}_{N}^1 -{\hat{\vv}}_{N}^2 |> \sqrt{\frac{ ({\hat{\nu}}_{N}^1 )^2 }{N}+ \frac{ ({\hat{\nu}}_{N}^2   )^2  }{N}  } \Phi^{-1}(1-\frac{\beta}{2}) \right\}\mbox{ for test \eqref{deftest1}-(a) with }K=2, \\
\left\{(\xi_{N}^1, \xi_{N}^2) \;:\;{\hat{\vv}}_{N}^1  > {\hat{\vv}}_{N}^2  + \sqrt{\frac{ ({\hat{\nu}}_{N}^1   )^2 }{N}+ \frac{ ({\hat{\nu}}_{N}^2   )^2  }{N}  } \Phi^{-1}(1-\beta) \right\}\mbox{ for test \eqref{deftest1}-(b) with }K=2.
\end{array}
\end{equation}

We finally consider test \eqref{testcontregalite}:
$$
H_0:\theta \in \Theta_0 \mbox{ against }{H_1}:\theta \notin \Theta_0
$$
for $\theta=(\vv_{*}^1,\ldots,\vv_{*}^K)^T$ with $\Theta_0$ a linear space or a closed convex cone.

Let $\Theta_0$ be the subspace
 \begin{equation} \label{defcontegaltestbis}
\Theta_0 = \{\theta \in \mathbb{R}^K : A \theta = 0\}
\end{equation}
where
$A$ is a $k_0 \small{\times}K$ matrix of full rank $k_0$.
Note that  test \eqref{deftest1}-(a) can be written under this form with
$A$ a $(K-1)\small{\times}K$ matrix of rank $K-1$.
We have for $\theta$
the estimator  ${\hat \theta}_{N}=\Big({\hat{\vv}}_{N}^1,\ldots,{\hat{\vv}}_{N}^K\Big)^T$.
Fixing $N$ large, since $\xi^{N, 1}, \ldots, \xi^{N, K}$ are independent, using
the fact that $\frac{N^{1/2}\left(\hat{\vv}_N^i  - \vv_*^i \right)}{  {\hat \nu}_{N}^i } \xrightarrow{\mathcal{D}}  \mathcal{N}(0,1)$, the distribution of ${\hat \theta}_{N}$
can be approximated by the Gaussian
$\mathcal{N}(\theta, \Sigma)$ distribution with $\Sigma$ the diagonal
matrix $\Sigma=(1/N)\mbox{diag}\Big(\mbox{Var}(H_1( z_*^1 , \xi_1 )), \ldots, \mbox{Var}(H_K( z_*^K , \xi_K ))\Big)$.
The log-likelihood ratio statistic for test \eqref{testcontregalite}
is
$\Lambda=\frac{\displaystyle \sup_{\theta \in \Theta_0, \Sigma \succ 0} \mathcal{L}(\theta, \sigma)}{\displaystyle \sup_{\theta, \Sigma \succ 0} \mathcal{L}(\theta, \sigma)}$
where $\mathcal{L}(\theta, \Sigma)$ is the likelihood function for a Gaussian multivariate model.
For a sample $(\tilde \theta_1, \ldots, \tilde \theta_M)$ of $\hat \theta_N$, introducing the estimators
$$
{\hat \theta}=\frac{1}{M} \sum_{i=1}^{M}\,\tilde \theta_i   \mbox{ and }
{\hat \Sigma}=\frac{1}{M-1} \sum_{i=1}^{M}\,\Big(\tilde \theta_i - {\hat \theta}\Big) \Big(\tilde \theta_i - {\hat \theta}\Big)^T
$$
of respectively $\theta$ and $\Sigma$, we have
\begin{equation}\label{hotelling}
-2 \ln \Lambda=K \ln \Big( 1+\frac{T^2}{M-1} \Big) \;\mbox{ where }\;T^2=K \displaystyle \min_{\theta \in \Theta_0}\,\Big({\hat \theta}-\theta \Big)^T {\hat \Sigma}^{-1} \Big({\hat \theta}-\theta \Big)
\end{equation}
and when $\Theta_0$ is of the form \eqref{defcontegaltestbis}, under $H_0$, we have $T^2 \sim \frac{k_0 (M-1)}{M-k_0 } F_{k_0 , M-k_0 }$
where $F_{p , q}$ is the Fisher-Snedecor distribution with parameters $p$ and $q$.
For asymptotic test
\eqref{testcontregalite} at confidence level $\beta$ with $\Theta_0$ given by \eqref{defcontegaltestbis},
we then reject $H_0$ if
$T^2  \geq \frac{k_0 (M -1)}{M-k_0 } F_{k_0 , M-k_0}^{-1}(1-\beta)$
where $F_{p, q}^{-1}(\beta)$ is the $\beta$-quantile of the Fisher-Snedecor distribution with parameters $p$ and $q$.

Now take for $\Theta_0$ the convex cone $\Theta_0 = \{\theta \in \mathbb{R}^K : A \theta \leq 0\}$
where $A$ is a $k_0 \small{\times}K$ matrix of full rank $k_0$ (tests
\eqref{deftest1}-(b), (c) are special cases) and assume that $M \geq K+1$.
Since the corresponding null hypothesis is {\em{$\theta$ belongs to a one-sided cone}},
on the basis of the sample $(\tilde \theta_1, \ldots, \tilde \theta_M)$ of $\hat \theta_N$, we can use \cite{perlman}
and we reject $H_0$ for large values of the statistic
$$
\mathcal{U}(\Theta_0)= \| {\hat \theta} \|_S^2 - \| \Pi_{S}({\hat \theta}|\Theta_0 )\|_S^2 = \|   {\hat \theta}  -  \Pi_{S}({\hat \theta}|\Theta_0 ) \|_S^2
$$
where $S=\frac{M-1}{M}{\hat \Sigma}$, $\|x\|_S =\sqrt{x^T S^{-1} x}$, and $\Pi_{S}(x|A)$ is any point in $A$
minimizing $\|y-x\|_S$ among all $y \in A$.
For a type I error of at most $0<\beta<1$, knowing that \cite{perlman}
\begin{equation}\label{testconeerrI}
\sup_{\theta \in \Theta_0, \Sigma \succ 0}  \mathbb{P}\Big(\mathcal{U}(\Theta_0) \geq u | \theta, \Sigma \Big) \leq 
\mbox{\tt{Err}}(u):=\frac{1}{2}\Big[ \mathbb{P}\Big( G_{K-1, M-K-1} \geq u \Big) + \mathbb{P}\Big(  G_{K, M-K} \geq u \Big)\Big],
\end{equation}
where $G_{p, q}=(p/q)F_{p, q}$,
we reject $H_0$ if $\mathcal{U}(\Theta_0) \geq u_\beta$  where $u_\beta$ satisfies
$\beta=\mbox{\tt{Err}}(u_\beta)$ with $\mbox{\tt{Err}}(\cdot)$ given by \eqref{testconeerrI}.

\section{Numerical experiments} \label{numsim}

\subsection{Comparing the risk of two distributions} \label{examplecompare}

We consider test \eqref{deftest1} with $K=2$ and $\X$ a singleton.
We use the rejection regions given in Section \ref{natest} (resp. given by \eqref{rejectionregionasK2}) in the nonasymptotic (resp. asymptotic) case.
In this situation, the test aims at comparing the risk of two distributions.
We use the notation $\mathcal{N}(m, \sigma^2; a_0, b_0)$ for the normal
distribution with mean $m$ and variance $\sigma^2$ conditional on
this random variable being in $[a_0, b_0]$ (truncated normal distribution
with support $[a_0, b_0]$).
More precisely, we compare the risks $\mathcal{R}( \xi_1 )$ and $\mathcal{R}( \xi_2 )$  of two  truncated normal (loss) distributions $\xi_1$ and $\xi_2$
with support $[a_0, b_0]=[0, 30]$ in three cases: (I)
$\xi_1 \sim \mathcal{N}(10, 1; 0, 30)$, $\xi_2 \sim \mathcal{N}(20, 1; 0, 30)$,
(II) $\xi_1 \sim \mathcal{N}(5, 1; 0, 30)$, $\xi_2 \sim \mathcal{N}(10, 25; 0, 30)$,
and (III) $\xi_1 \sim \mathcal{N}(10, 49; 0, 30)$, $\xi_2 \sim \mathcal{N}(14, 0.25; 0, 30)$.
For these three cases, the densities of $\xi_1$ and $\xi_2$ are represented in
Figure \ref{figuredensity} (top left for (I), top right for (II), and bottom for (III)).

We take for $\mathcal{R}$ the risk measure
$\mathcal{R}(\xi) = w_0 \mathbb{E}[\xi] + w_1 \avr_{\alpha}( \xi )$ for $0<\alpha<1$ where $w_0, w_1 \geq 0$
with $w_0 + w_1=1$.
We assume that only the support $[a_0,b_0]$ of $\xi_1$ and $\xi_2$ and two samples $\xi_1^N$ and $\xi_2^N$
of size $N$ of respectively $\xi_1$ and $\xi_2$ are known.
Since
the distribution of $\xi$ has support $[a_0,b_0]$, we can write
\begin{equation} \label{confintr}
\mathcal{R}(\xi) = \min_{\tau \in [a_0 , b_0 ]} w_0 \mathbb{E}[\xi] + w_1 \left( \tau + \frac{1}{1-\alpha} \mathbb{E} [\xi - \tau ]_{+} \right)
\end{equation}
which is of form \eqref{ropt-1} with a risk-neutral objective function,
$G(\tau,\xi)=w_0 \xi + w_1 \tau + \frac{w_1}{1-\alpha}[\xi-\tau]_{+}$,
and $\X$ the compact set $\X=[a_0,b_0]=[0,30]$.
It follows that the RSA algorithm can be used to estimate $\mathcal{R}(\xi_1)$ and $\mathcal{R}(\xi_2)$ and to compute the confidence bounds
\eqref{lowsmd1} and \eqref{upsmd1} with $L, M_1,$ and $M_2$ given by \eqref{formulaLM1M2simplecase}.
In these formulas, we replace $\ubar \tau$ by its lower bound $0$
since we do not assume the mean and standard deviation of $\xi_1$ and $\xi_2$ known.
We obtain  $L= w_1 \max(1,\frac{\alpha}{1-\alpha})$,
$M_2= \frac{w_1}{1-\alpha}$, and $M_1 = 30 (w_0 + \frac{w_1}{1-\alpha})$.

\begin{figure}[H]
\centering
\begin{tabular}{ll}
\includegraphics[scale=0.35]{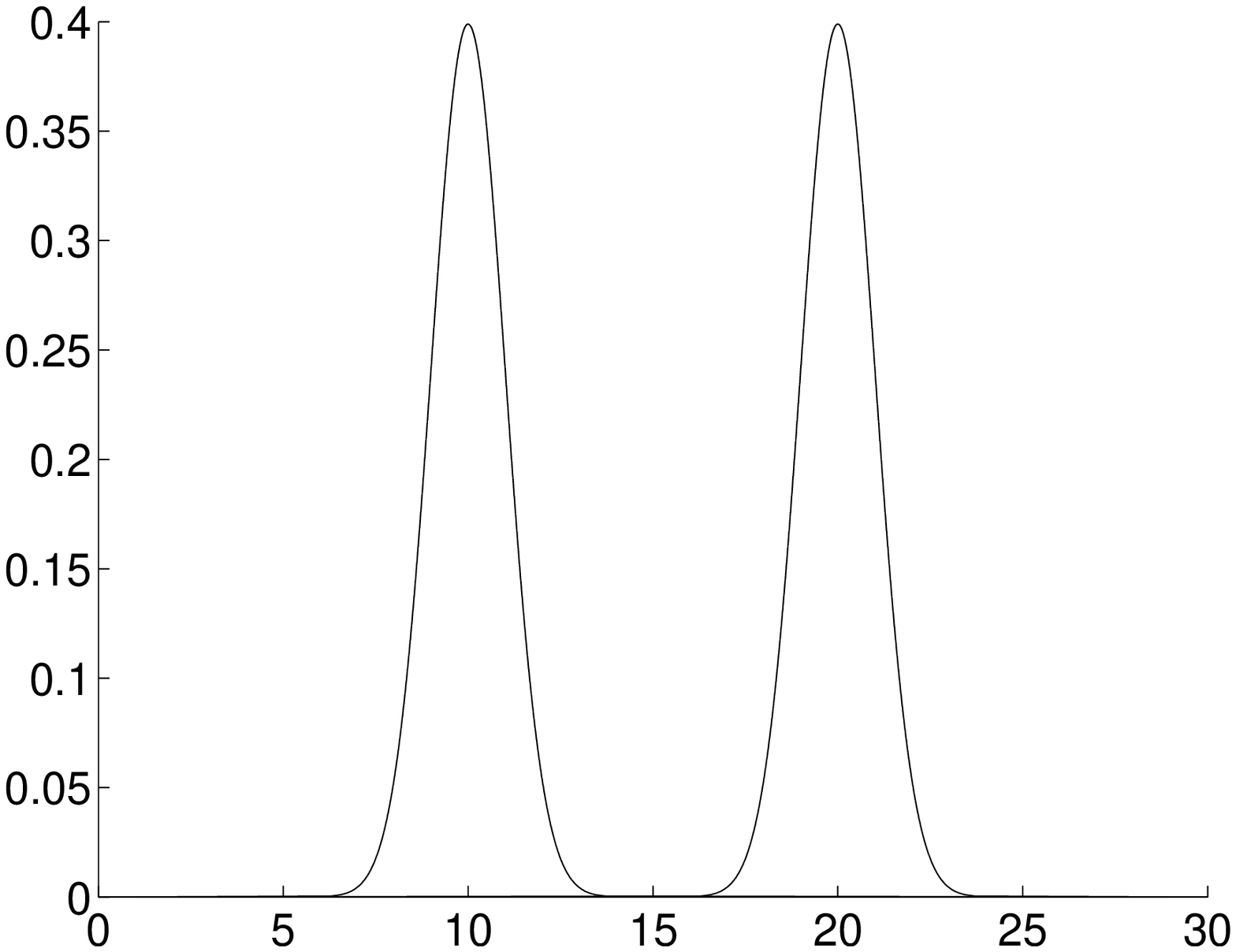}
&
\includegraphics[scale=0.35]{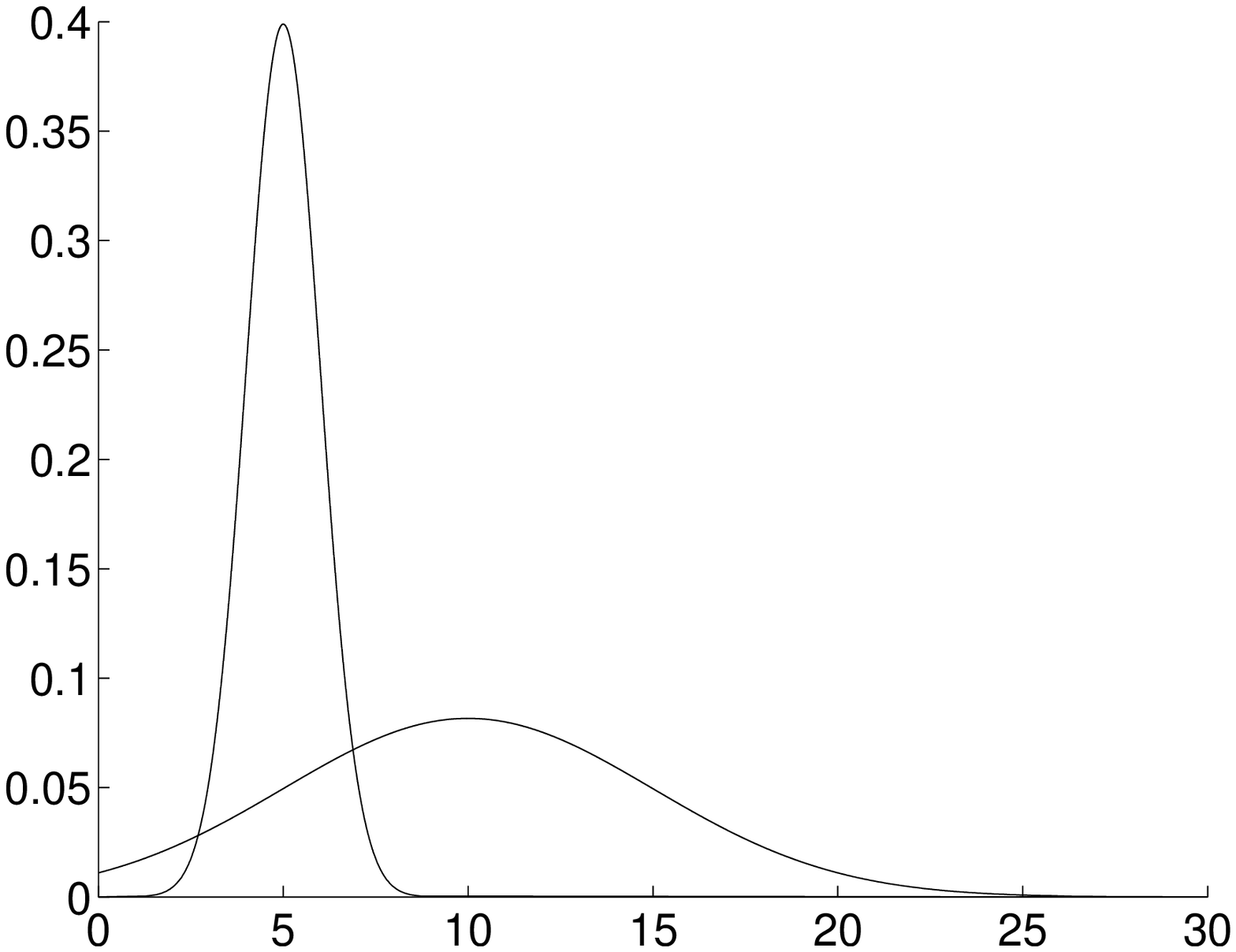}
\end{tabular}
\begin{tabular}{l}
\includegraphics[scale=0.35]{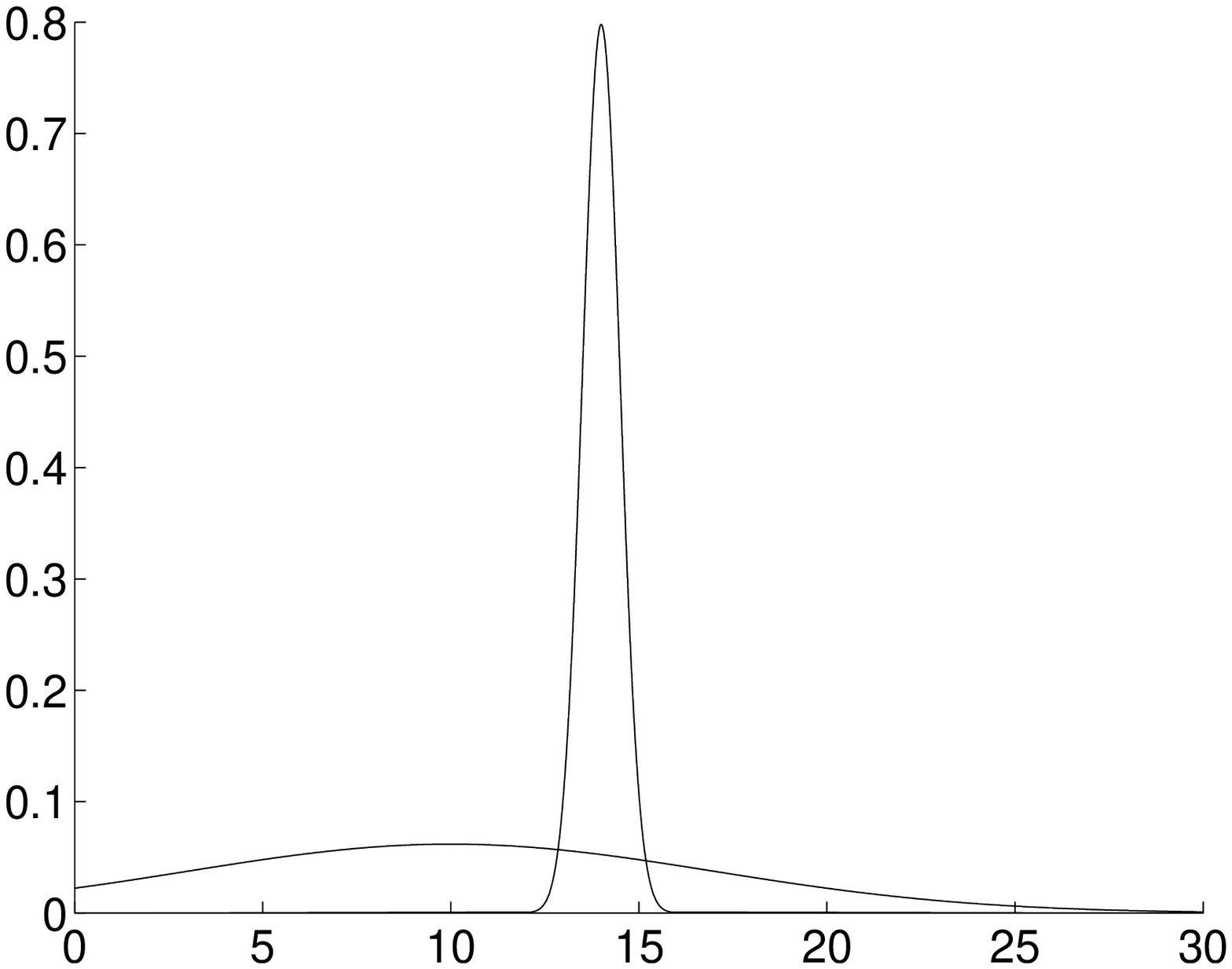}
\end{tabular}
\caption{Densities of truncated normal loss distributions $\xi_1$ and $\xi_2$.
Top left plot: $\xi_1 \sim \mathcal{N}(10, 1; 0, 30)$ and $\xi_2 \sim \mathcal{N}(20, 1; 0, 30)$.
Top right plot: $\xi_1 \sim \mathcal{N}(5, 1; 0, 30)$ and $\xi_2 \sim \mathcal{N}(10, 25; 0, 30)$.
Bottom plot: $\xi_1 \sim \mathcal{N}(10, 49; 0, 30)$ and $\xi_2 \sim \mathcal{N}(14, 0.25; 0, 30)$.}
\label{figuredensity}
\end{figure}

We first illustrate Theorem \ref{th-minmax} computing the empirical estimation $\mathcal{R}( {\hat F}_{N, 1}  )$ of $\mathcal{R}(\xi_1)$
on $200$ samples of size $N$ of $\xi_1 \sim \mathcal{N}(10, 1; 0, 30)$
for $w_0=0.1$, $w_1=0.9$, and various values of $\alpha$ and of the sample size $N$.
For this experiment, the QQ-plots of the empirical distribution of $\mathcal{R}( {\hat F}_{N, 1}  )$ versus the normal distribution with parameters the empirical mean and standard deviation
of this empirical distribution are reported in Figure \ref{figureqq}. We see that even for small values of $1-\alpha$
and $N$ as small as $20$, the distribution of $\mathcal{R}( {\hat F}_{N, 1}  )$
is well approximated by a Gaussian distribution: for $N=20$ the  Jarque-Bera test accepts the hypothesis of normality at the significance level 0.05 for $1-\alpha=0.01$
and $1-\alpha=0.5$.

\begin{figure}
\begin{tabular}{ll}
\includegraphics[scale=0.5]{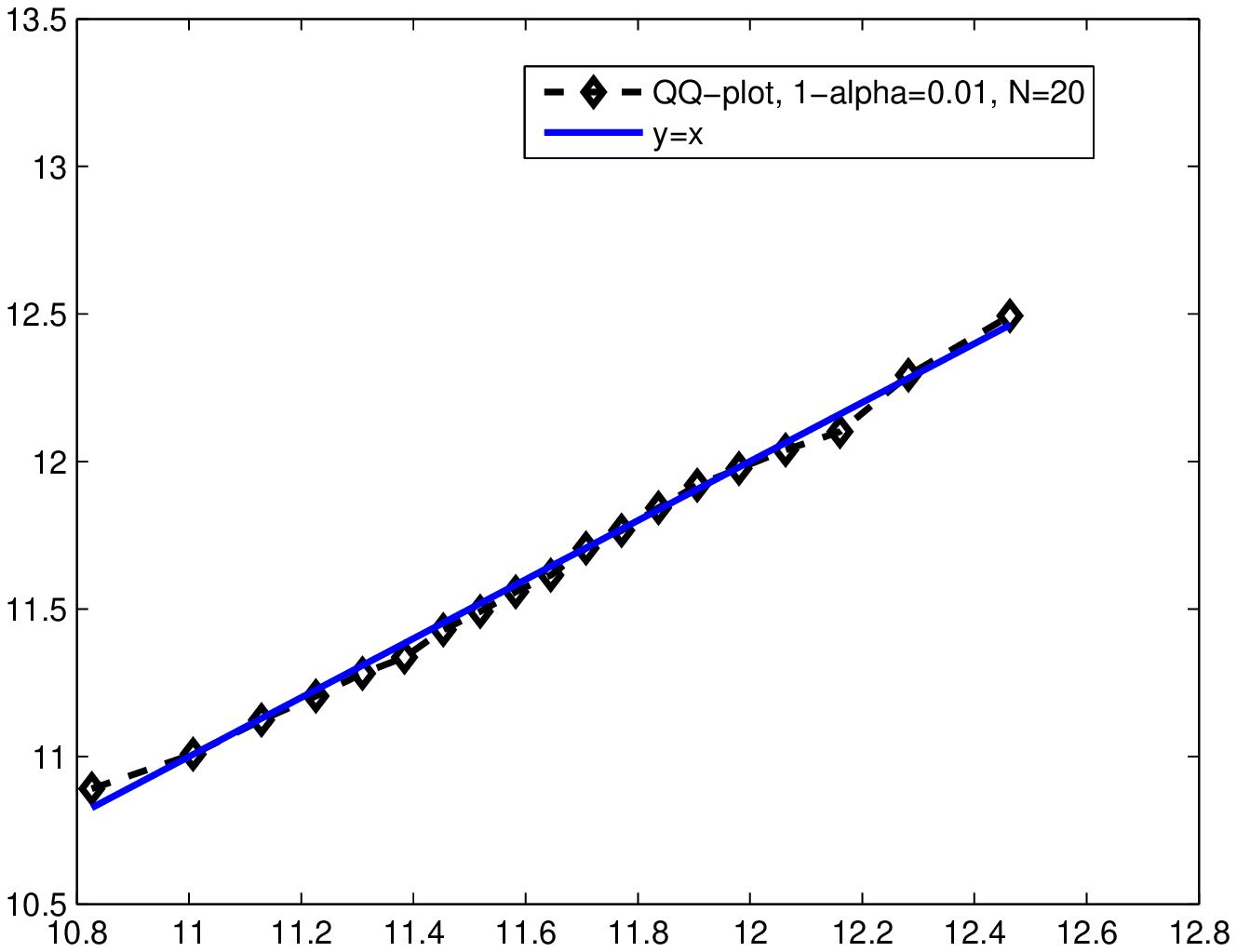}
&
\includegraphics[scale=0.5]{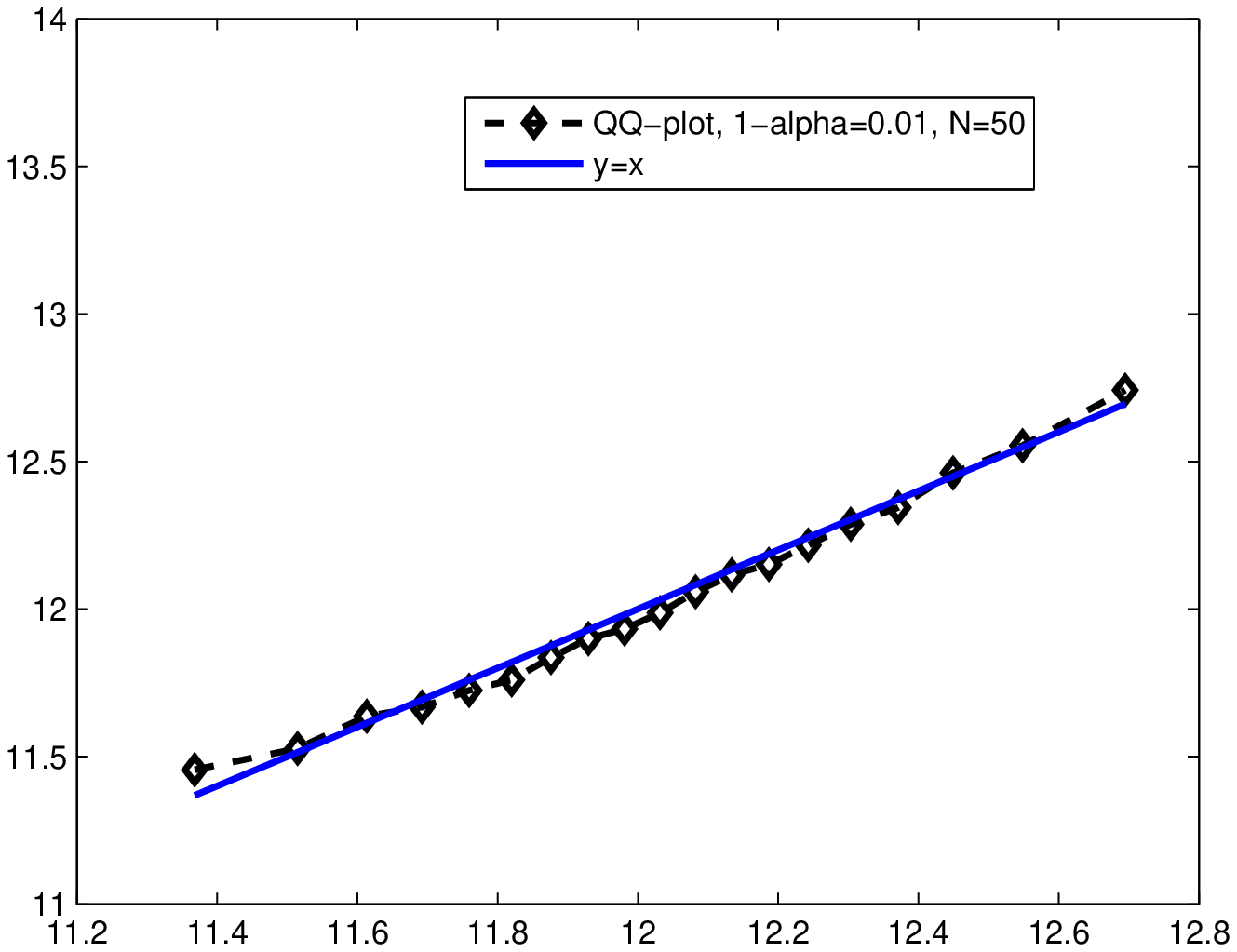}\\
\includegraphics[scale=0.5]{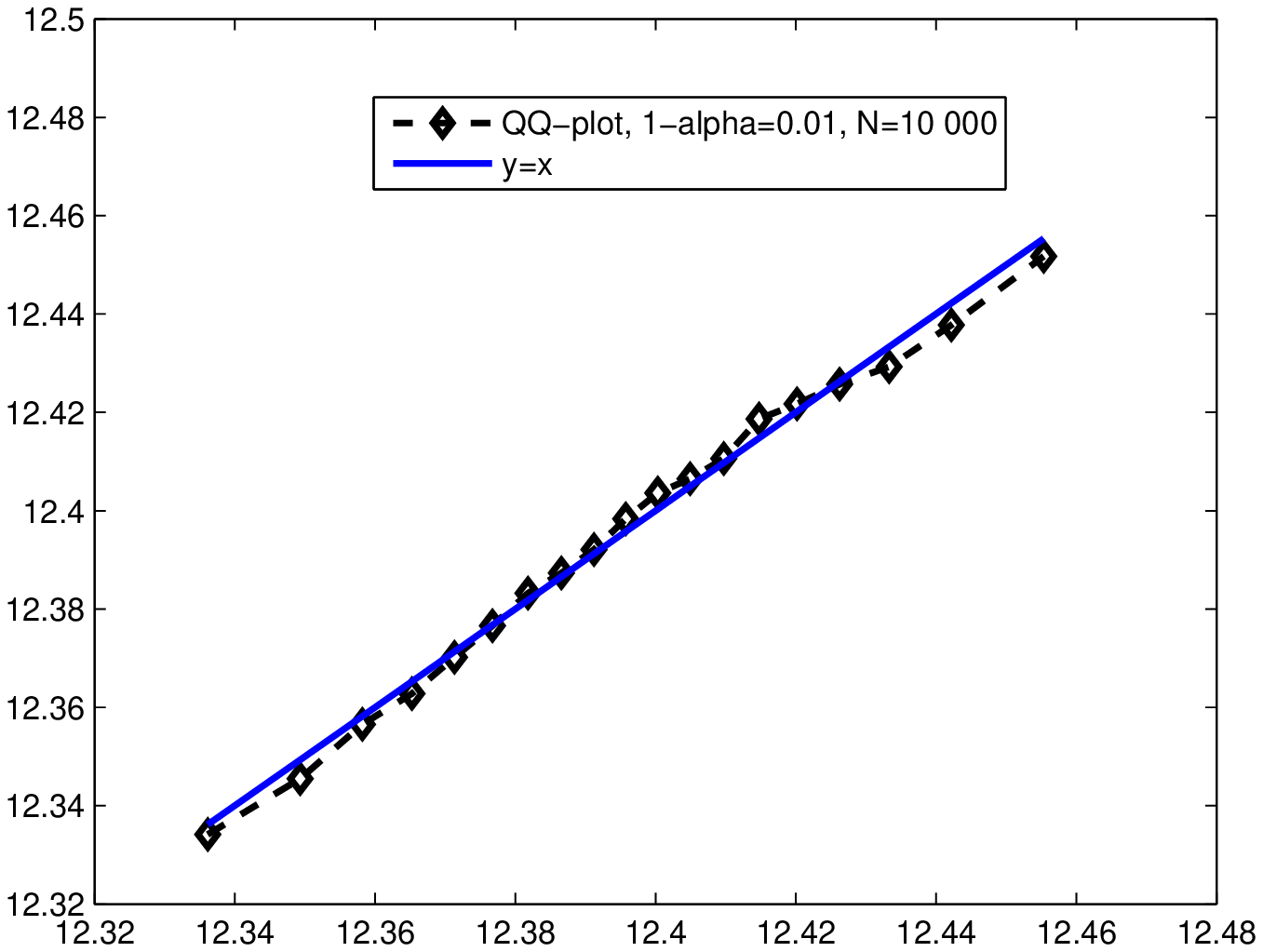}
&\includegraphics[scale=0.5]{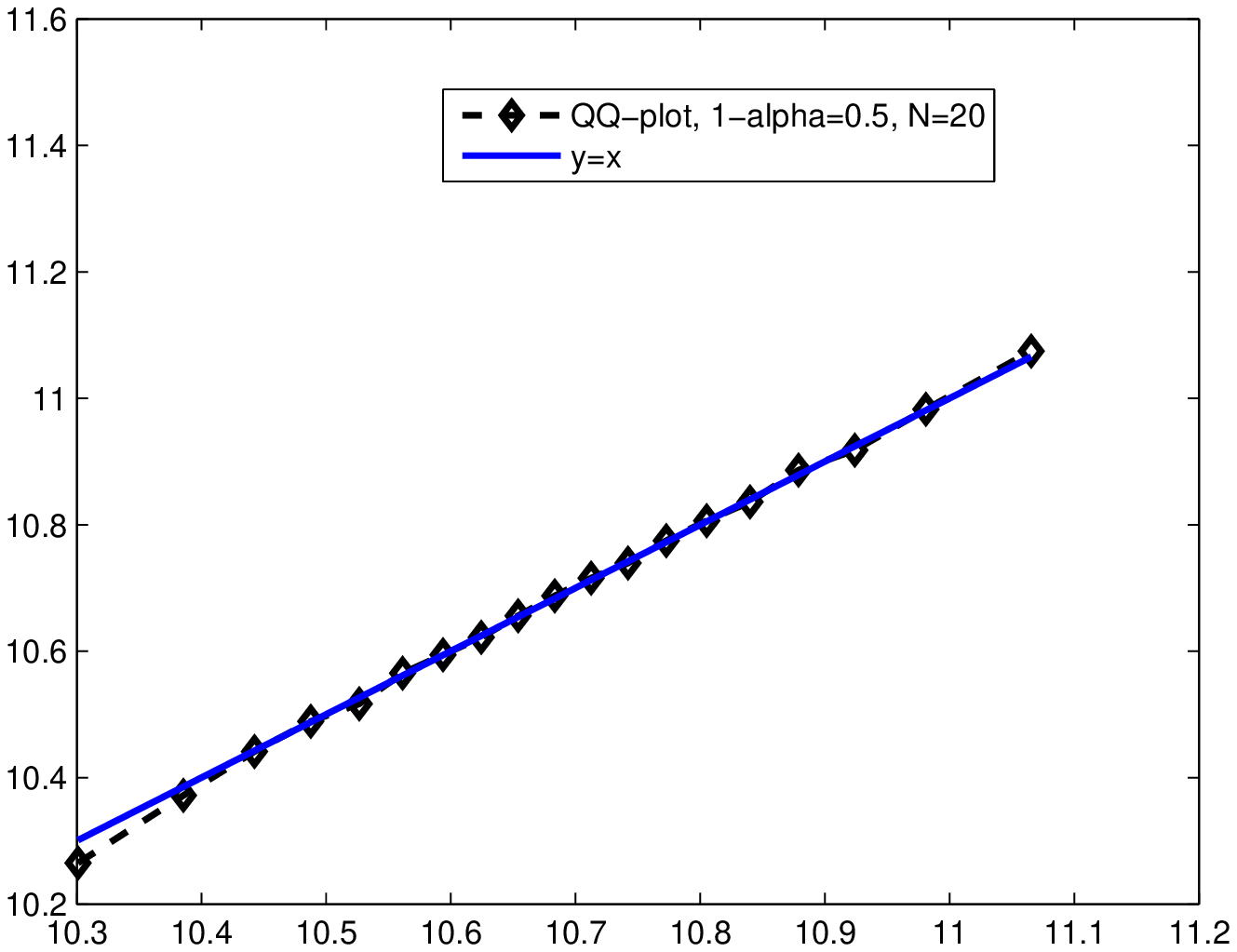}\\
\includegraphics[scale=0.5]{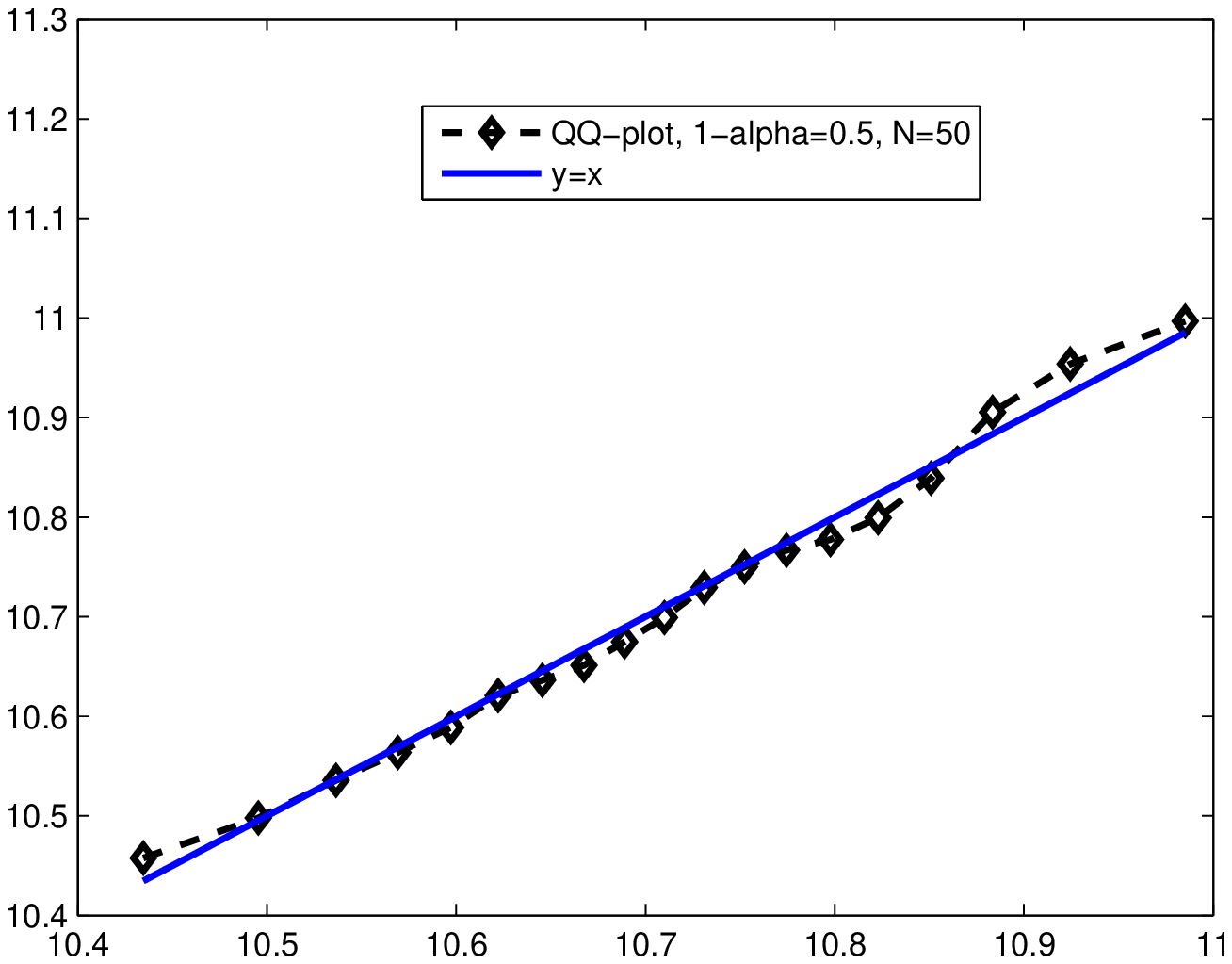}&
\includegraphics[scale=0.5]{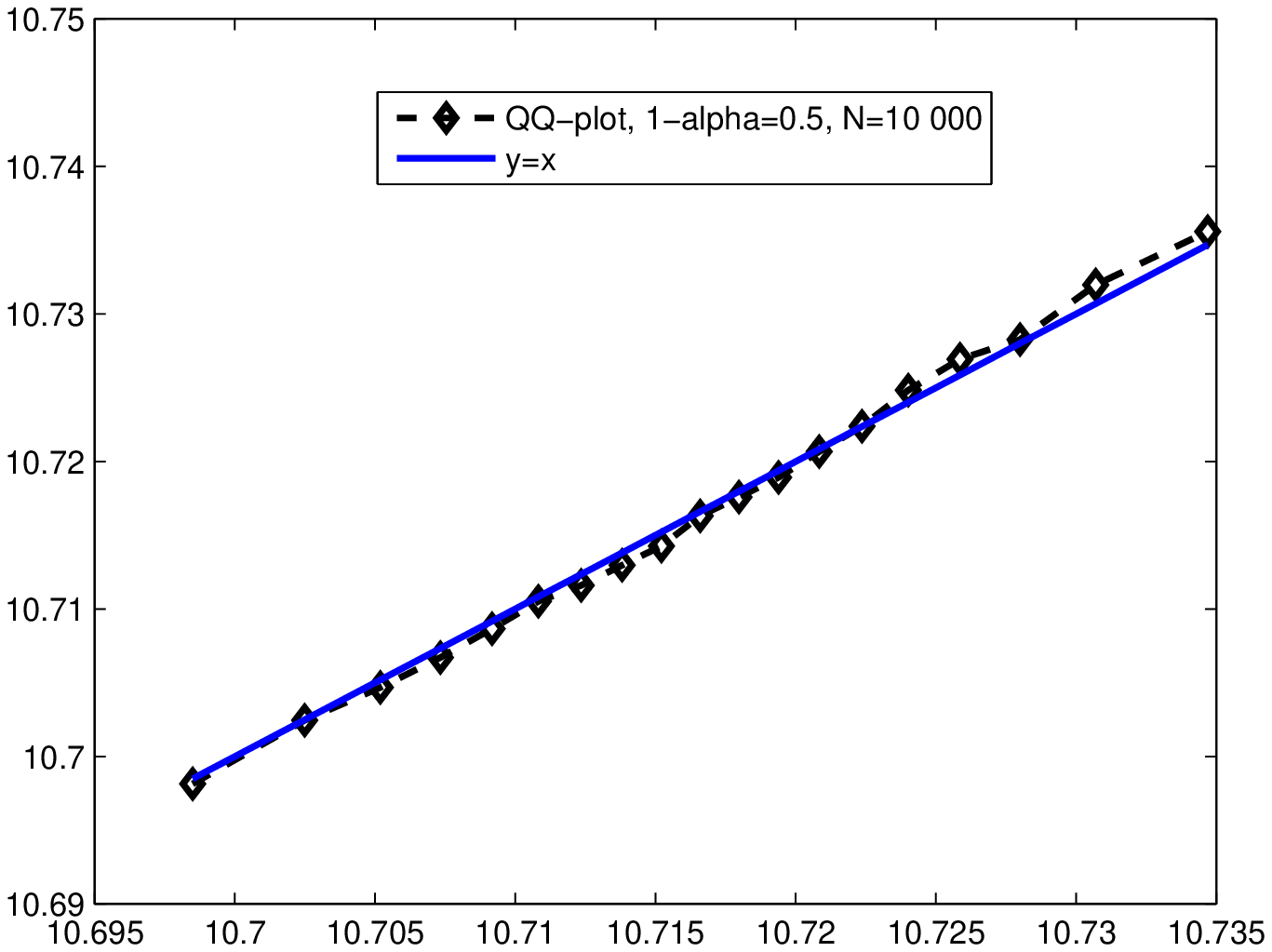}
\end{tabular}
\caption{QQ-plots of the empirical distribution of $\mathcal{R}( {\hat F}_{N, 1}  )$ versus the normal distribution with parameters the empirical mean and standard deviation
of this empirical distribution for $\xi_1 \sim \mathcal{N}(10, 1; 0, 30)$ for $(w_0, w_1)=(0.1, 0.9)$ and various values of
$(\alpha, N)$.}
\label{figureqq}
\end{figure}

We fix again the distribution $\xi_1 \sim \mathcal{N}(10, 1; 0, 30)$ and approximately compute $\mathcal{R}(\xi_1)$
for various values of $(w_0, w_1, \alpha, N)$ using the RSA and SAA methods on samples
$\xi_1^N$ of size $N$ of $\xi_1$.
For a sample of size $N$ of $\xi_1$, let ${\hat{\mathcal{R}}}_{N, \,{\tt{RSA}}}(\xi_1)$ and ${\hat {\mathcal{R}}}_{N, \,{\tt{SAA}}}(\xi_1)= \mathcal{R}( {\hat F}_{N, 1}  )$
be these estimations using respectively RSA and SAA.
For fixed $(w_0, w_1, \alpha, N)$, we generate 200 samples of size $N$ of $\xi_1$ and for each sample we
compute  ${\hat{\mathcal{R}}}_{N,\,{\tt{RSA}}}(\xi_1)$ and ${\hat{\mathcal{R}}}_{N, \,{\tt{SAA}}}(\xi_1)$  and report in Table \ref{tablegn1} the average of these values
for $N \in \{20, 50, 100, 10^3, 10^4, 10^5, 10^6 \}$.
Considering that $\mathcal{R}(\xi_1)$ is
the value obtained using SAA for $N=10^6$, we observe that RSA correctly approximates
$\mathcal{R}(\xi_1)$ as  $N$ grows and that the estimation of $\mathbb{E}[ {\hat{\mathcal{R}}}_{N, {\tt{SAA}}}(\xi_1) ]$   (resp. $\mathbb{E}[ {\hat{\mathcal{R}}}_{N, {\tt{RSA}}}(\xi_1) ]$)
increases (resp. decreases) with the sample size $N$, as expected.
We also naturally observe that the more weight is given to the $\avr$
and the smaller $1-\alpha$
the more difficult it is to estimate the risk measure, i.e., the more distant the expectation of the
approximation is to the optimal value and the larger the sample size needs to be to obtain an expected approximation
with given accuracy.

\begin{table}
\centering
\begin{tabular}{|c|c|c|c|c|c|c|c|c|c|}
\cline{4-10}
 \multicolumn{1}{c}{} & \multicolumn{1}{c}{}  &  \multicolumn{1}{c}{}  &\multicolumn{7}{|c|}{Sample size $N$}\\
\hline
$(w_0,w_1)$&$1-\alpha$& {\tt{Method}}    & $20$ & $50$ &  $10^2$ & $10^3$ &  $10^4$ & $10^5$  & $10^6$  \\
\hline
(0.1, 0.9) & $10^{-2}$ & {\tt{SAA}} & 11.71& 12.00& 12.21& 12.37& 12.40&12.40  &12.40\\
\hline
(0.1, 0.9) & $10^{-2}$ & {\tt{RSA}} &14.35 & 14.26& 14.16&13.46 &12.75 & 12.51 &12.43\\
\hline
(0.1, 0.9) & 0.1 & {\tt{SAA}} & 11.51&11.50 & 11.54&11.58 &11.58 &11.58  & 11.58 \\
\hline
(0.1, 0.9) & 0.1 & {\tt{RSA}} & 20.50& 16.78&15.10 &12.61 &11.90 & 11.68 &  11.61\\
\hline
(0.1, 0.9) & 0.5 & {\tt{SAA}} & 10.71& 10.69& 10.72& 10.72&10.72 & 10.72 & 10.72  \\
\hline
(0.1, 0.9) & 0.5 & {\tt{RSA}} &11.42 & 11.12& 11.02&10.81 &10.75 &10.73  &10.72   \\
\hline
(0.9, 0.1) & $10^{-2}$ & {\tt{SAA}} & 10.19&10.23 &10.25 & 10.26& 10.27&10.27  & 10.27\\
\hline
(0.9, 0.1) & $10^{-2}$ & {\tt{RSA}} &10.49 &10.48 &10.47 &10.38 &10.31 &10.28  &10.27 \\
\hline
(0.9, 0.1) & 0.1 & {\tt{SAA}} & 10.17& 10.16&10.19 &10.18 &10.18 &10.18  &10.18 \\
\hline
(0.9, 0.1) & 0.1 & {\tt{RSA}} &10.34 &10.28 &10.27 &10.20 &10.18 & 10.18 &10.18 \\
\hline
(0.9, 0.1) & 0.5 & {\tt{SAA}} &10.09 &10.07 &10.08 & 10.08& 10.08& 10.08 & 10.08 \\
\hline
(0.9, 0.1) & 0.5 & {\tt{RSA}}  & 10.17& 10.11& 10.12&10.09 &10.08 & 10.08 & 10.08 \\
\hline
\end{tabular}
\caption{Estimation of the risk measure value $\mathcal{R}(\xi_1)$ for $\xi_1 \sim \mathcal{N}(10, 1; 0, 30)$
using SAA and RSA for various values of $(w_0, w_1, \alpha)$ and various sample sizes $N$.}
\label{tablegn1}
\end{table}

We now study for case (I) the test
\begin{equation}\label{test1}
H_0: \mathcal{R}(\xi_1) = \mathcal{R}(\xi_2) \;\;\mbox{against }H_1: \mathcal{R}(\xi_1) \neq \mathcal{R}(\xi_2).
\end{equation}
We first fix $(w_0, w_1)=(0.1,0.9)$ and report in Tables \ref{tabletest0} and \ref{tabletest01} for various values of the pair $(\alpha, N)$
the average nonasymptotic and asymptotic confidence bounds for
$\mathcal{R}(\xi_1)$ and $\mathcal{R}(\xi_2)$ when $\xi_1 \sim \mathcal{N}(10, 1; 0, 30)$ and $\xi_2 \sim \mathcal{N}(20, 1; 0, 30)$.\footnote{The nonasymptotic confidence interval is given by \eqref{lowsmd1}-\eqref{upsmd1}.
Recalling that $\mathcal{R}(\xi)$ is the optimal value of optimization problem \eqref{confintr} which is of the form
\eqref{ropt-1}, we compute for $\mathcal{R}(\xi)$ the asymptotic confidence interval
$[\hat \vv_N - \Phi^{-1}(1-\beta/2) \frac{\hat \nu_N   }{\sqrt{N}}, \hat \vv_N + \Phi^{-1}(1-\beta/2) \frac{\hat \nu_N  }{\sqrt{N}}]$,
where  
$\hat \vv_N$ is the optimal value of the SAA of \eqref{confintr}. Note that in this case the optimal value $\hat \tau_N$ of
the SAA problem is the $\alpha$-quantile of the distribution of $\xi$ (no optimization step is necessary to solve the SAA problem).}
We observe that even for small values of the sample size and of the confidence level $1-\alpha$, the asymptotic confidence interval is of small width
and its bounds close to the risk measure value. For RSA, a large sample is needed to obtain a confidence interval of small width,
especially when $1-\alpha$ is small.

For all the remaining tests of this section, we choose $\beta=0.1$ for the maximal type I error and $1-\alpha=0.1$.
Since in case (I) we have $\mathcal{R}(\xi_1) \neq \mathcal{R}(\xi_2)$ (see Figure \ref{figuredensity}),
from this experiment we expect to obtain a large probability of type II error using the nonasymptotic tests of Section \ref{natest}  based on the confidence intervals
computed using RSA, unless the sample size is very large.
More precisely, we compute the probability of type II error for \eqref{test1} considering asymptotic and nonasymptotic rejection regions using
various sample sizes $N \in \{20$, $50$, $100$, $1\,000$, $5\,000$, $10\,000$, $20\,000$, $50\,000$, $100\,000$, $130\,000$, $150\,000\}$,
taking $1-\alpha=0.1$ and $(w_0$, $w_1)$ $\in$ $\{(0,1), (0.1,0.9)$, $(0.2,0.8)$, $(0.3,0.7)$, $(0.4,0.6)$, $(0.5,0.5)$, $(0.6,0.4)$, $(0.7,0.3)$, $(0.8,0.2)$, $(0.9,0.1)\}$.
For fixed $N$, the probability of type II error is estimated using 100 samples of size $N$ of $\xi_1$ and $\xi_2$.
Using the asymptotic rejection region, we reject $H_0$ for all realizations and all parameter combinations, meaning that  the  probability of type II error is null
(since $H_1$ holds for all
parameter combinations). For the nonasymptotic test, the probability of type II errors are reported in Table
\ref{tabletest1}. For sample sizes less than $5\,000$, the probability of  type II error is always 1 (the nonasymptotic test always takes the wrong decision)
and the larger $w_1$ the larger the sample size $N$ needs to be to obtain a probability of  type II error of zero. In particular, if
$w_1=1$ (we estimate the $\avr_{\alpha}$ of the distribution) as much as $150\,000$ observations are needed to obtain a null probability of type II error.
However, if the sample size is sufficiently large, both tests always take the correct decision $\mathcal{R}(\xi_1) \neq \mathcal{R}( \xi_2 )$.
\begin{table}
\centering
\begin{tabular}{|c|c|c|c|c||c|c|c|c|}
\hline
   $N$    & {\tt{Low-As1}} & {\tt{Up-As1}} &  {\tt{Low-RSA1}} & {\tt{Up-RSA1}} & {\tt{Low-As2}} & {\tt{Up-As2}} &  {\tt{Low-RSA2}} & {\tt{Up-RSA2}}  \\
\hline
  $50$ & 11.20&  11.94   &-347.30& 146.23 & 21.30&    21.97 &-335.67&157.97\\
\hline
  $10^3$&11.49 &11.67 &  -68.65   &41.71  &21.49 &    21.67 &-58.37&51.99\\
\hline
  $10^4$ & 11.55  &  11.61   &  -13.79  & 21.11 & 21.55  &  21.61   &-3.71&31.19\\
\hline
  $10^5$ &11.57 &  11.59   &3.56& 14.59 &21.57 &  21.59   &13.58&24.62\\
\hline
  $1.5 \small{\times}10^5$ &11.57 & 11.59  & 5.03 &14.04  &21.57 &   21.59  &15.05&24.06\\
\hline
\end{tabular}
\caption{Average values of the asymptotic and nonasymptotic confidence bounds for  $\mathcal{R}(\xi_1)$ and $\mathcal{R}(\xi_2)$
when $\xi_1 \sim \mathcal{N}(10, 1; 0, 30)$ and $\xi_2 \sim \mathcal{N}(20, 1; 0, 30)$, $1-\alpha=0.1$.
For $\mathcal{R}( \xi_i )$, the average asymptotic confidence interval is [{\tt{Low-Asi}}, {\tt{Up-Asi}}] and
the average nonasymptotic confidence interval is  [{\tt{Low-RSAi}}, {\tt{Up-RSAi}}].}
\label{tabletest0}
\end{table}

\begin{table}
\centering
\begin{tabular}{|c|c|c|c|c||c|c|c|c|}
\hline
  $N$    & {\tt{Low-As1}} & {\tt{Up-As1}} &  {\tt{Low-RSA1}} & {\tt{Up-RSA1}} & {\tt{Low-As2}} & {\tt{Up-As2}} &  {\tt{Low-RSA2}} & {\tt{Up-RSA2}}  \\
\hline
 $50$ &10.45 &  10.98   &-230.29&  40.16& 20.47& 21.00    &-220.25&50.21\\
\hline
 $10^3$ &10.65 & 10.77    &-43.18&17.30  &20.65 &   20.77  &-33.18&27.29\\
\hline
 $10^4$ &10.70 &  10.74   &-6.33& 12.80 &20.70 &   20.74  &3.68&22.80\\
\hline
 $10^5$ & 10.71&  10.72   &5.33& 11.38 &20.71 &   20.72  & 15.33  &21.38\\
\hline
 $1.5 \small{\times}10^5$ & 10.71&   10.72  &6.32& 11.26  & 20.71&   20.72  &16.32&21.25\\
\hline
\end{tabular}
\caption{Average values of the asymptotic and nonasymptotic confidence bounds for  $\mathcal{R}(\xi_1)$ and $\mathcal{R}(\xi_2)$
when $\xi_1 \sim \mathcal{N}(10, 1; 0, 30)$ and $\xi_2 \sim \mathcal{N}(20, 1; 0, 30)$, $1-\alpha=0.5$.
For $\mathcal{R}( \xi_i )$, the average asymptotic confidence interval is [{\tt{Low-Asi}}, {\tt{Up-Asi}}] and
the average nonasymptotic confidence interval is  [{\tt{Low-RSAi}}, {\tt{Up-RSAi}}].}
\label{tabletest01}
\end{table}

\begin{table}
\centering
\begin{tabular}{|c|c|c|c|c|c|c|c|}
\cline{2-8}
 \multicolumn{1}{c}{}  &\multicolumn{7}{|c|}{Sample size $N$}\\
 \hline
$(w_0,w_1)$& $5\,000$ & $10\,000$ &  $20\,000$ & $50\,000$ &  $100\,000$ & $130\,000$ & $150\,000$ \\
\hline
(0.0, 1.0)      & 1  &  1 &   1 &  1 &  1 & 1  &0\\
\hline
(0.1, 0.9) & 1  &  1 &   1 &  1 &  1 & 0  &   0 \\
\hline
(0.2, 0.8) & 1  &  1 &   1 &  1 &  0 & 0   &   0 \\
\hline
(0.3, 0.7) & 1  &  1 &   1 &  1 &  0 &  0  &   0 \\
\hline
(0.4, 0.6) & 1  &  1 &   1 &  1 &  0 &  0  &   0 \\
\hline
(0.5, 0.5) & 1  & 1  &   1 &  0 &  0 &  0  &   0 \\
\hline
(0.6, 0.4) & 1  &  1 &   1 &  0 &  0 &  0  &   0 \\
\hline
(0.7, 0.3) & 1  &  1 &   0 &  0 &  0 &   0 &   0 \\
\hline
(0.8, 0.2) & 1  &  0 &   0 &  0 &  0 &  0  &   0  \\
\hline
(0.9, 0.1) & 0  & 0  &   0 & 0  &  0 &  0  &  0      \\
\hline
\end{tabular}
\caption{Empirical probabilities of type II error for  tests  \eqref{test1} and \eqref{test2} using a nonasymptotic rejection region when  $\xi_1 \sim \mathcal{N}(10, 1; 0, 30)$, $\xi_2 \sim \mathcal{N}(20, 1; 0, 30)$, and $1-\alpha=0.1$.}
\label{tabletest1}
\end{table}

Given (possibly small) samples of size $N$ of $\xi_1$ and $\xi_2$, to know which of the two risks $\mathcal{R}(\xi_1)$
and $\mathcal{R}(\xi_2)$  is the smallest, we now consider the test
\begin{equation}\label{test2}
H_0: \mathcal{R}(\xi_1) \geq  \mathcal{R}(\xi_2) \;\;\mbox{against }H_1: \mathcal{R}(\xi_1) < \mathcal{R}(\xi_2).
\end{equation}
Computing $\mathcal{R}(\xi_1)$ and $\mathcal{R}(\xi_2)$ with a very large sample (of size $10^6$)
of $\xi_1$ and $\xi_2$
either with SAA or RSA or looking at Figure \ref{figuredensity},
we know that $\mathcal{R}(\xi_1)<\mathcal{R}(\xi_2)$. We again analyze the probability of  type II error using the asymptotic and nonasymptotic rejection regions
when the decision is taken on the basis of a much smaller sample. For the nonasymptotic test, the empirical probabilities of  type II error for various sample sizes (estimated, for fixed $N$, using 100 samples
of size $N$ of $\xi_1$ and $\xi_2$) are exactly those obtained for test \eqref{test1} and are given in Table \ref{tabletest1}.
The asymptotic test again always takes the correct decision $\mathcal{R}(\xi_1)<\mathcal{R}(\xi_2)$ while
a large sample size is needed to always take the correct decision using the nonasymptotic test, as large as $150\,000$ for
$w_1=1$.

We now consider tests \eqref{test1} and \eqref{test2} for case (II).
In this case, there is a larger overlap between the distributions of $\xi_1$ and  $\xi_2$.
However, from Figure \ref{figuredensity} and computing  $\mathcal{R}(\xi_1)$ and $\mathcal{R}(\xi_2)$ with a very large sample (say of size $10^6$)
of $\xi_1$ and $\xi_2$ either using SAA or RSA,
we check that we have again $\mathcal{R}(\xi_2)>\mathcal{R}(\xi_1)$ for all
values of $(w_0, w_1)$. The empirical probabilities of type II error are null for the asymptotic test
for all sample sizes $N$ tested
while for the nonasymptotic test, the probabilities of  type II error are given in Table \ref{tabletest2} for both
tests \eqref{test1} and \eqref{test2}. As a result, here again, the asymptotic test always takes the correct decision $\mathcal{R}(\xi_1)<\mathcal{R}(\xi_2)$ while
a large sample size is needed to always take the correct decision using the nonasymptotic test (as large as $110\,000$ for
$w_1=1$). For sample sizes less than $10\,000$, the empirical probability of  type II error with the nonasymptotic test is 1.
We see that for fixed $(w_0, w_1)$, in most cases, we need a larger sample size than in case (I) to
have a null probability of  type II error, due the overlap of the two distributions.

\begin{table}
\centering
\begin{tabular}{|c|c|c|c|c|c|}
\cline{2-6}
 \multicolumn{1}{c}{}  &\multicolumn{5}{|c|}{Sample size $N$}\\
 \hline
$(w_0,w_1)$& $10\,000$ &  $20\,000$ & $50\,000$ &  $100\,000$ & $110\,000$  \\
\hline
(0.0, 1.0)   &  1 &   1    &  1 &  1   &0\\
\hline
(0.1, 0.9)   &  1 &   1    &  1 &  0   &   0 \\
\hline
(0.2, 0.8)   &  1 &   1    &  1 &  0    &   0 \\
\hline
(0.3, 0.7)   &  1 &   1    &  1 &  0   &   0 \\
\hline
(0.4, 0.6)   &  1 &   1    &  1 &  0   &   0 \\
\hline
(0.5, 0.5)   & 1  &   1    &  1 &  0   &   0 \\
\hline
(0.6, 0.4)   &  1 &   1    &  0 &  0   &   0 \\
\hline
(0.7, 0.3)   &  1 &   1    &  0 &  0   &   0 \\
\hline
(0.8, 0.2)   &  1 &   1 &  0 &  0   &   0  \\
\hline
(0.9, 0.1)   & 0.06 & 0    &  0  &  0  &  0      \\
\hline
\end{tabular}
\caption{Empirical probabilities of  type II error for  tests  \eqref{test1} and \eqref{test2} using a nonasymptotic rejection region when  $\xi_1 \sim \mathcal{N}(5, 1; 0, 30)$, $\xi_2 \sim \mathcal{N}(10, 25; 0, 30)$, and $1-\alpha=0.1$.}
\label{tabletest2}
\end{table}

We finally consider case (III) where the choice between $\xi_1$ and $\xi_2$
is more delicate and depends on the pair $(w_0, w_1)$.
In this case, we have (see Figure \ref{figuredensity}) $\mathbb{E}[\xi_2]>\mathbb{E}[\xi_1]$
and $\mbox{\avr}_{\alpha}( \xi_2 )<\mbox{\avr}_{\alpha}( \xi_1 )$ for $1-\alpha = 0.1$. It follows that for pairs $(w_0, w_1)$
summing to one, when
$$
0 \leq w_0 < w_{{\tt{Crit}}}=\frac{\mbox{\avr}_{\alpha}(\xi_1)-\mbox{\avr}_{\alpha}(\xi_2)}{\mathbb{E}[\xi_2]-\mathbb{E}[\xi_1]+\mbox{\avr}_{\alpha}(\xi_1)-\mbox{\avr}_{\alpha}(\xi_2)}
$$ then
$\mathbb{R}(\xi_2)<\mathbb{R}(\xi_1)$ and for $w_0 > w_{{\tt{Crit}}}$ then  $\mathbb{R}(\xi_2)>\mathbb{R}(\xi_1)$.
The empirical estimation of $w_{{\tt{Crit}}}$ (estimated using a sample of size $10^6$) is 0.71.
For $w_0$ close to $w_{{\tt{Crit}}}$,  $\mathcal{R}(\xi_1)$ and $\mathcal{R}(\xi_2)$ are close and
the probability of type II error for test \eqref{test1} can be large even for the asymptotic test if the sample size
is not sufficiently large. More precisely, for the asymptotic test, when $(w_0, w_1)=(0.7, 0.3)$,
the empirical probabilities of type II error are given in Table \ref{tabletest3} for $N \in \{20, 50, 100, 200, 500, 1\,000, 2\,000, 5000\}$,
and are $0.28, 0.11, 0.01$, and $0$ for respectively
$N=10\,000, 20\,000, 40\,000$, and $45\,000$.
For the remaining values of $w_0$ the empirical probabilities of type II error are given in Table \ref{tabletest3} for the asymptotic test.
For the nonasymptotic test, the empirical probabilities of  type II error for test \eqref{test1} are given in Table \ref{tabletest4}.
It is seen that much larger sample sizes are needed in this case to obtain a small probability of type II error.
However, for the sample size  $N=5\small{\times}10^6$, the nonasymptotic test still always takes the wrong decision
for the difficult case $w_0 = 0.7$.
\begin{table}
\centering
\begin{tabular}{|c|c|c|c|c|c|c|c|c|}
\cline{2-9}
 \multicolumn{1}{c}{}  &\multicolumn{8}{|c|}{Sample size $N$}\\
 \hline
$(w_0,w_1)$& $20$ & $50$ & $100$ & $200$ & $500$ & $1000$ & $2000$ & $5000$\\
\hline
(0.0, 1.0)   & 0.13  &  0.01  &  0 &  0   &0&  0 &0&0\\
\hline
(0.1, 0.9)   & 0.24  &   0    &  0 &   0  & 0   &0&0&0\\
\hline
(0.2, 0.8)   &  0.32 &   0.03    & 0  &  0   & 0  &0&0&0 \\
\hline
(0.3, 0.7)   & 0.50  &   0.07   &  0 &   0  &   0 &0&0&0\\
\hline
(0.4, 0.6)   &  0.61 &    0.11  &  0 &   0  &   0 &0&0&0\\
\hline
(0.5, 0.5)   & 0.71 &   0.46    & 0.11   & 0.01    & 0   &0&0&0\\
\hline
(0.6, 0.4)   & 0.86  &   0.69   &  0.50  &    0.28 &   0.01 &0&0&0\\
\hline
(0.7, 0.3)   &  0.83 &   0.85   &  0.90  &  0.91   & 0.87   &0.89&0.69&0.53\\
\hline
(0.8, 0.2)   &  0.71 &   0.71   &  0.65  &  0.29   &  0.07  &0&0 &0\\
\hline
(0.9, 0.1)   & 0.57 &   0.34    &  0.09  &    0   & 0 & 0 &0&0\\
\hline
\end{tabular}
\caption{Empirical probabilities of  type II error for  test  \eqref{test1} using an asymptotic rejection region when  $\xi_1 \sim \mathcal{N}(10, 49; 0, 30)$, $\xi_2 \sim \mathcal{N}(14, 0.25; 0, 30)$, and $1-\alpha=0.1$.}
\label{tabletest3}
\end{table}

\begin{table}
\centering
\begin{tabular}{|c|c|c|c|c|c|c|}
\cline{2-7}
 \multicolumn{1}{c}{}  &\multicolumn{6}{|c|}{Sample size $N$}\\
 \hline
$(w_0,w_1)$& $100\,000$ & $300\,000$ & $500\,000$ & $700\,000$ & $10^6$ & $5\small{\times}10^6$ \\
\hline
(0.0, 1.0)   & 1  & 0.83   & 0  &  0   &0&0   \\
\hline
(0.1, 0.9)   & 1  &     1  & 0  &   0  &   0 &0\\
\hline
(0.2, 0.8)   & 1  &    1   & 0  &    0 &   0& 0\\
\hline
(0.3, 0.7)   & 1  &   1    & 0  &   0  &   0 &0\\
\hline
(0.4, 0.6)   & 1  &   1    &  1 &   0  &   0 &0\\
\hline
(0.5, 0.5)   & 1 &    1   &   1 &  1   &   0 &0\\
\hline
(0.6, 0.4)   &  1 &   1   &   1 &   1  &   1 &0\\
\hline
(0.7, 0.3)   &  1 &   1   &   1 &   1  &   1 &1\\
\hline
(0.8, 0.2)   &  1 &   1   &   1 &   1  &   1 & 0\\
\hline
(0.9, 0.1)   & 0 &    0   &   0 &   0  &    0  & 0 \\
\hline
\end{tabular}
\caption{Empirical probabilities of  type II error for  test  \eqref{test1} using a nonasymptotic rejection region when  $\xi_1 \sim \mathcal{N}(10, 49; 0, 30)$, $\xi_2 \sim \mathcal{N}(14, 0.25; 0, 30)$, and $1-\alpha=0.1$.}
\label{tabletest4}
\end{table}
For $w_0<w_{{\tt{Crit}}}$ with $w_0 \in \{0.0, 0.1, 0.2, 0.3, 0.4, 0.5, 0.6, 0.7\}$,
we are interested in the probability of  type II error of the test
\begin{equation}\label{test3}
H_0: \mathcal{R}(\xi_2) \geq \mathcal{R}(\xi_2) \;\;\mbox{against }H_1: \mathcal{R}(\xi_2) < \mathcal{R}(\xi_1)
\end{equation}
since $H_1$ holds in this case.
Using the asymptotic rejection region, except for the difficult case $w_0=0.7$ where the probability of  type II error
is still positive for $N=30\,000$, the empirical  probability of type II error is null for small to moderate (at most $1\,000$) sample sizes; see Table \ref{tabletest5}.
Using the nonasymptotic rejection region, much larger sample sizes are necessary to obtain a small probability of type II error, see Table \ref{tabletest6}.

\begin{table}[H]
\centering
\begin{tabular}{|c||c|c|c|c|c|c|c|c|}
\cline{2-9}
 \multicolumn{1}{c}{}  &\multicolumn{8}{|c|}{Sample size $N$}\\
 \hline
$(w_0,w_1)$& $20$ & $100$ & $200$ & $1\,000$ & $5\,000$ & $10\,000$ & $30\,000$ & $50\,000$\\
\hline
(0.0, 1.0)   &  0.11     &  0     &  0  &   0  &0 & 0&0&0\\
\hline
(0.1, 0.9)   &  0.26  &     0  &  0  &   0  & 0  & 0&0&0\\
\hline
(0.2, 0.8)   &   0.28    &   0    & 0   & 0    &  0  &0&0&0\\
\hline
(0.3, 0.7)   &    0.35   &    0   &  0 &   0  &  0  &0&0&0\\
\hline
(0.4, 0.6)   &   0.51    &    0   &  0  &   0  &  0  &0&0&0\\
\hline
(0.5, 0.5)   &    0.66   &    0.2   & 0.01   & 0    & 0  &0&0&0\\
\hline
(0.6, 0.4)   &    0.83   &     0.53  &  0.22  & 0    &  0  &0&0&0\\
\hline
(0.7, 0.3)   &  0.87     &     0.88  &  0.90  &   0.81  & 0.61   &0.39&0.05&0\\
\hline
\end{tabular}
\caption{Empirical probabilities of  type II error for  test  \eqref{test3} using an asymptotic rejection region when  $\xi_1 \sim \mathcal{N}(10, 49; 0, 30)$, $\xi_2 \sim \mathcal{N}(14, 0.25; 0, 30)$, and $1-\alpha=0.1$.}
\label{tabletest5}
\end{table}

\begin{table}[H]
\centering
\begin{tabular}{|c|c|c|c|c|c|c|c|}
\cline{2-8}
 \multicolumn{1}{c}{}  &\multicolumn{7}{|c|}{Sample size $N$}\\
 \hline
$(w_0,w_1)$& $300\,000$ & $400\,000$ & $500\,000$ & $700\,000$ & $900\,000$ & $2\small{\times}10^6$ & $5\small{\times}10^6$\\
\hline
(0.0, 1.0)   &  0     &   0    &  0  &  0   & 0&0 &0\\
\hline
(0.1, 0.9)   &  0.85  &   0    &   0 &   0  &  0 &0 &0\\
\hline
(0.2, 0.8)   &  1     &   0    &  0  &    0 &   0 &0&0\\
\hline
(0.3, 0.7)   &  1     &   1    &  0  &   0  &   0 &0&0\\
\hline
(0.4, 0.6)   &  1     &   1    &  1  &   0  &   0 &0&0\\
\hline
(0.5, 0.5)   &  1     &   1    &  1  &  1   &   0 &0&0\\
\hline
(0.6, 0.4)   &  1     &   1    &  1  &   1  &   1 &0.75&0\\
\hline
(0.7, 0.3)   &  1     &   1    &  1  &   1  &   1 &1&1\\
\hline
\end{tabular}
\caption{Empirical probabilities of  type II error for  test  \eqref{test3} using a nonasymptotic rejection region when  $\xi_1 \sim \mathcal{N}(10, 49; 0, 30)$, $\xi_2 \sim \mathcal{N}(14, 0.25; 0, 30)$, and $1-\alpha=0.1$.}
\label{tabletest6}
\end{table}
For $w_0>w_{{\tt{Crit}}}$ with $w_0 \in \{0.8, 0.9\}$,
we are interested in the probability of  type II error of test \eqref{test2}
since $H_1$ holds in this case. The probability of  type II error for this test using the nonasymptotic rejection region
is $1$ (resp. $0$) for $(N, w_0, w_1)=(10^6, 0.8, 0.2)$ (resp. $(N, w_0, w_1)=(10^6, 0.9, 0.1)$),
and null  for $(N, w_0, w_1)=(5\small{\times}10^6, 0.8, 0.2), (5\small{\times}10^6, 0.9, 0.1)$, meaning
that we always take the correct decision $\mathcal{R}(\xi_1)<\mathcal{R}(\xi_2)$ for $N=5\small{\times}10^6$ and
$(w_0, w_1)=(0.8, 0.2), (0.9, 0.1)$.
Using the asymptotic rejection region, the probabilities of  type II errors are null already for $N=1\,000$.
For $N=100$, we get probabilities of  type II error of $0.09$ and $0.42$ for respectively
$(w_0, w_1)=(0.8, 0.2)$ and $(w_0, w_1)=(0.9, 0.1)$.

\subsection{Tests on the optimal value of two risk averse stochastic programs} \label{example2}

We illustrate the results of Sections \ref{sec:asym} and \ref{testssection} on the risk averse problem
\begin{equation}\label{definstance3}
\left\{
\begin{array}{l}
\min w_0 \mathbb{E}[ \sum_{i=1}^n \xi_i x_i ]  + {w_1}\left( x_0 + \mathbb{E}\left[ \frac{1}{1-\alpha}[ \sum_{i=1}^n \xi_i x_i  -  x_0]_+  \right] \right) + \lambda_0 \|[x_0;x_1;...;x_n]\|_2^2 + c_0\\
-1 \leq x_0 \leq 1, \sum_{i=1}^n x_i = 1, \;x_i \geq 0, i=1,\ldots,n,
\end{array}
\right.
\end{equation}
where
$\xi$ is a random vector with i.i.d. Bernoulli entries:
$\mathbb{P}(\xi_i=1)=\Psi_i, \;\mathbb{P}(\xi_i=-1)=1-\Psi_i$, with $\Psi_i$ randomly drawn over $[0,1]$.\footnote{Of course $c_0$ can be ignored to solve the problem. However, it will be used to define
several instances and test the equality about their optimal values.}
This problem amounts to minimizing a linear combination of the expectation and the $\avr_{\alpha}$ of $\sum_{i=1}^n \xi_i x_i$ plus a penalty obtained taking $\lambda_0>0$. Therefore, it has a unique
optimal solution.
SAA formulation of this problem as well as the quadratic problems of each iteration of RSA were
solved numerically using Mosek Optimization Toolbox \cite{mosek}.
We will again use the rejection regions given in Section \ref{natest} (resp. given by \eqref{rejectionregionasK2}) in the nonasymptotic (resp. asymptotic) case.

To illustrate Theorem \ref{th-minopt}, for several instances of this problem, we report in Figures \ref{figureqq2} and  \ref{figureqq3}
the QQ-plots of the empirical distribution of the SAA optimal value for problem \eqref{definstance3}
versus the normal distribution with parameters the empirical mean and standard deviation
of this empirical distribution for various sample sizes $N$. We observe again that this distribution
is well approximated by a Gaussian distribution even when the sample size is small ($N=20$):
for all problem sizes ($n=100$, $n=500$, $n=10^3$, and $n=10^4$)
and the smallest sample size tested ($N=20$), the
Jarque-Bera test accepts the null hypothesis (the data comes from a normal distribution with unknown mean and variance)
at the 5\% significance level.

\begin{figure}
\begin{tabular}{ll}
\includegraphics[scale=0.5]{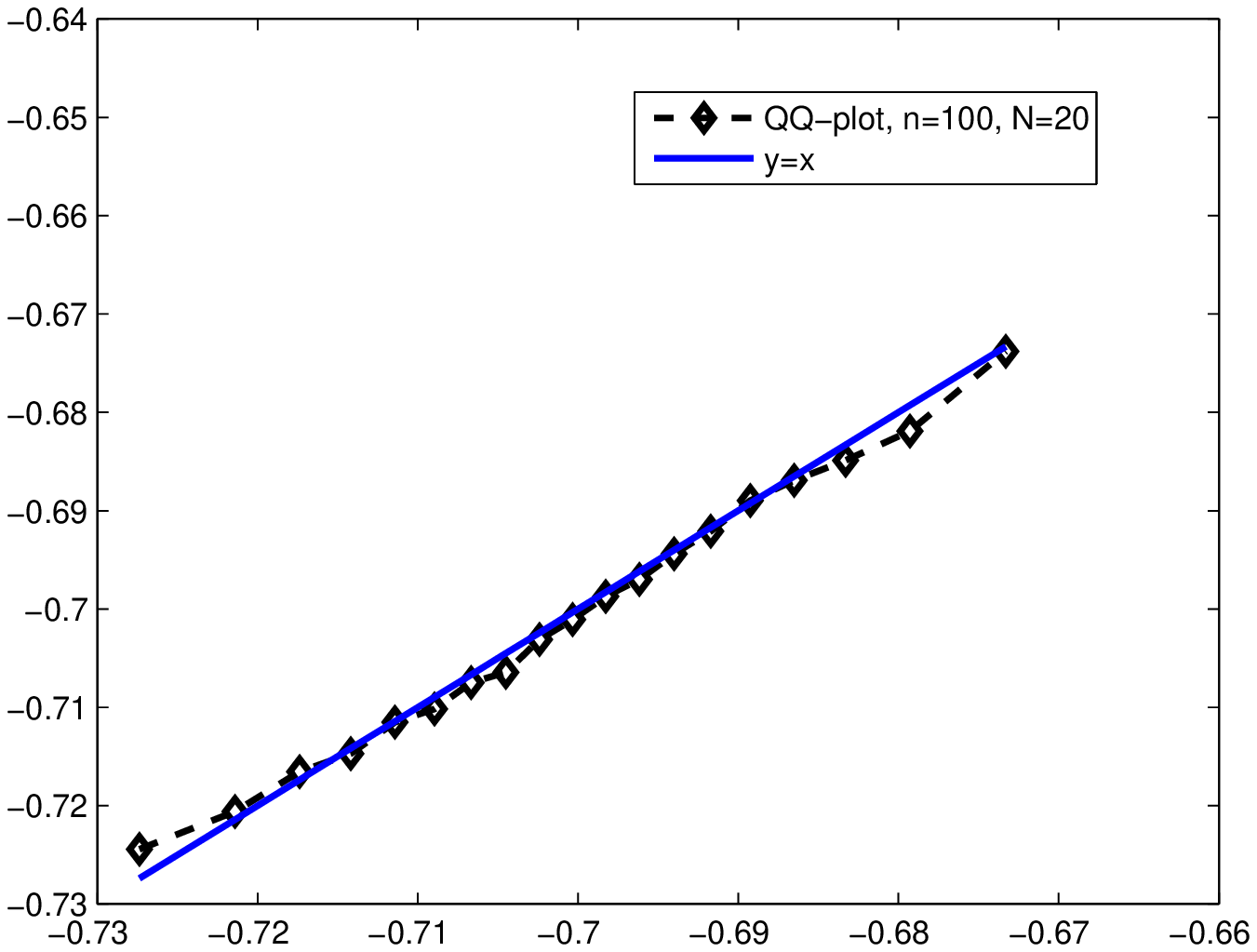}
&
\includegraphics[scale=0.5]{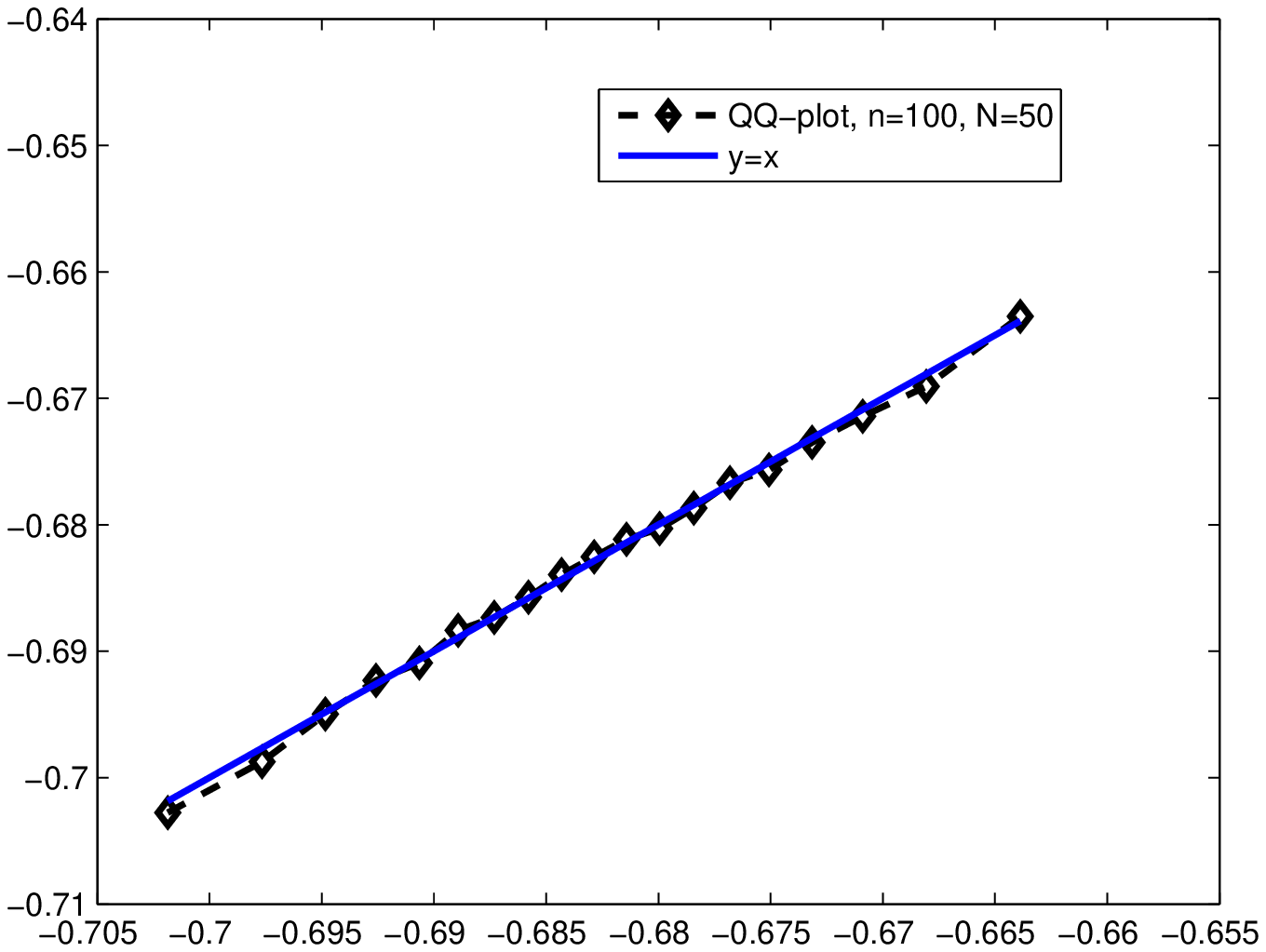}\\
\includegraphics[scale=0.5]{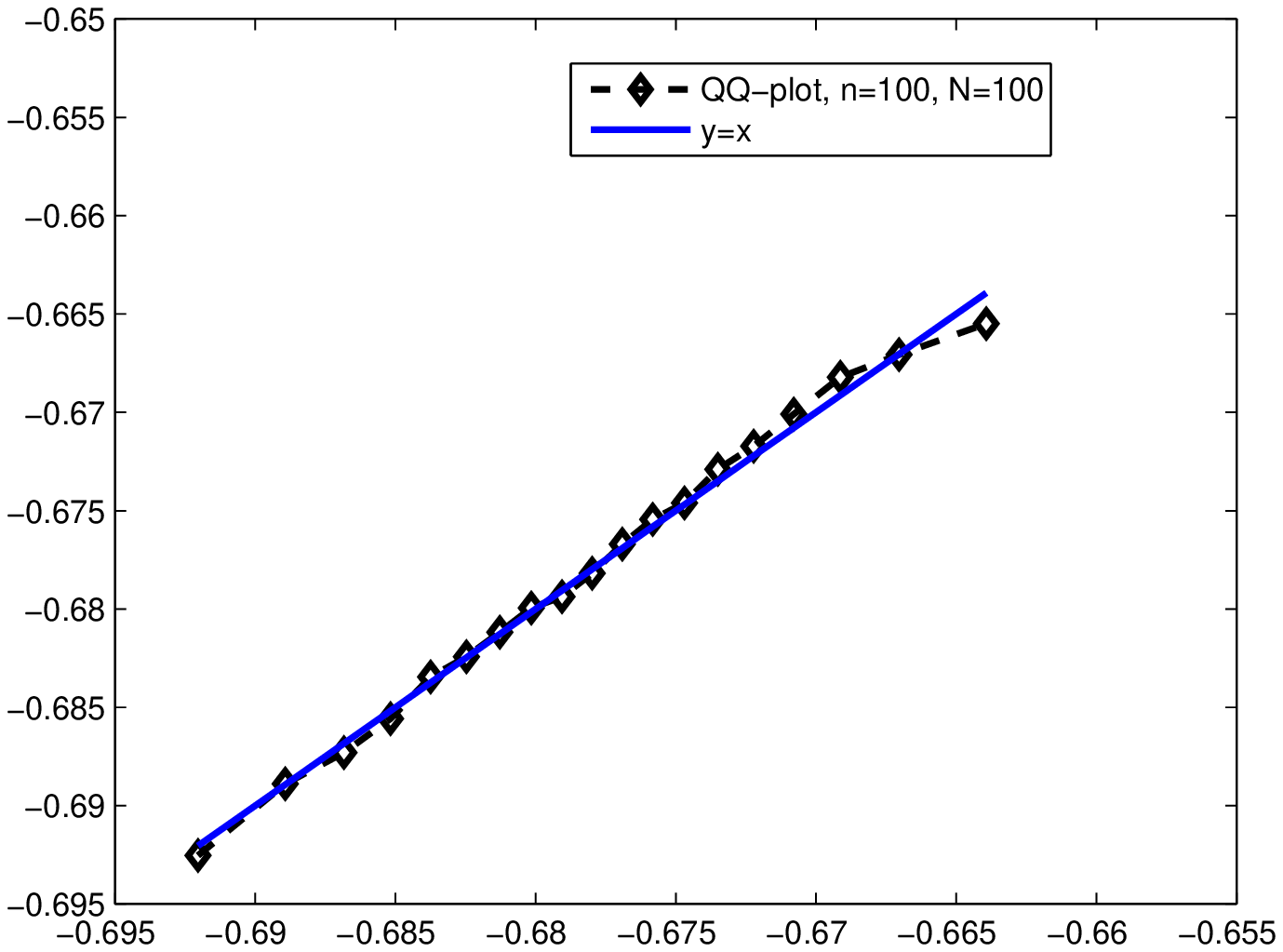}
&
\includegraphics[scale=0.5]{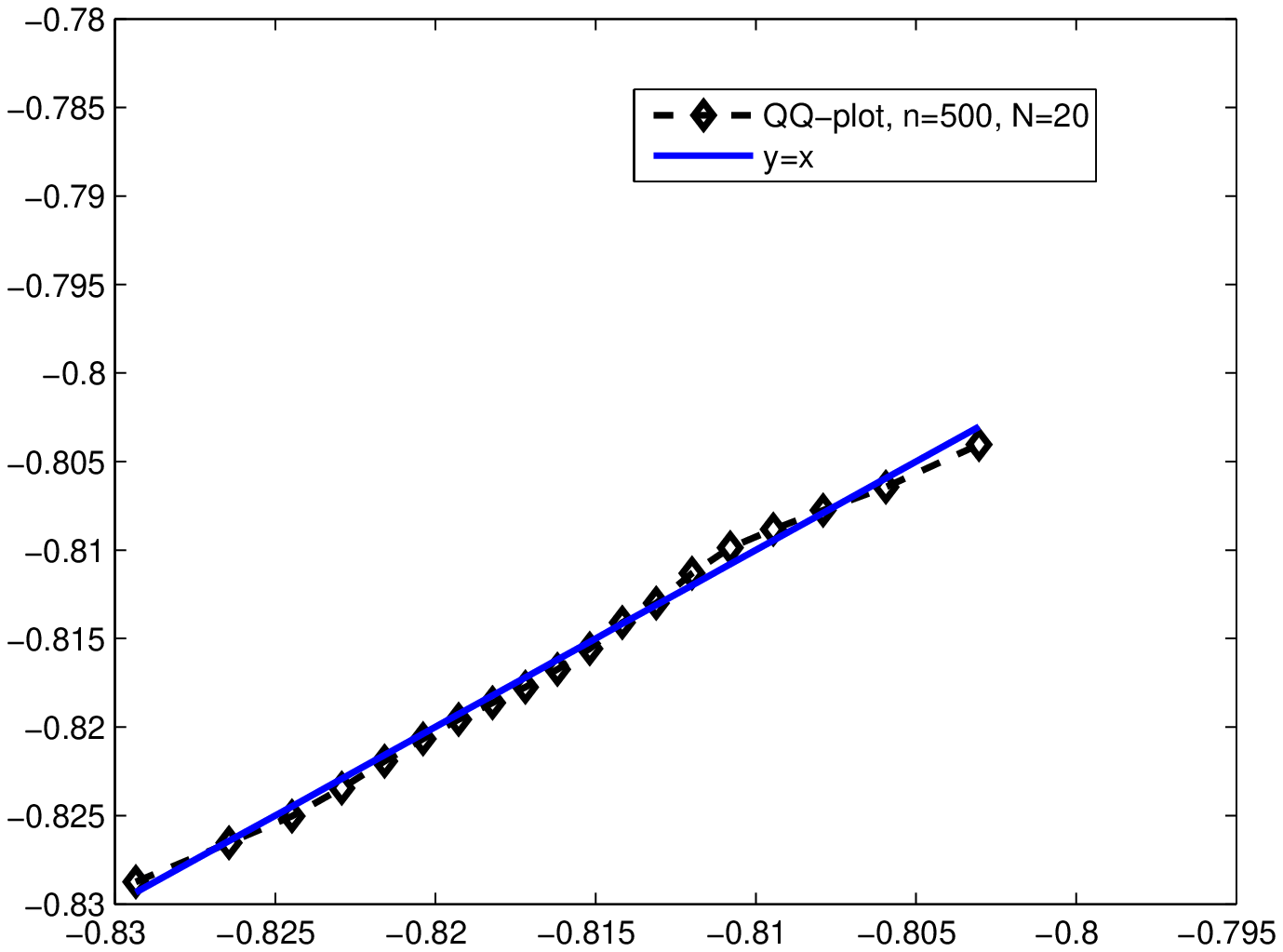}\\
\includegraphics[scale=0.5]{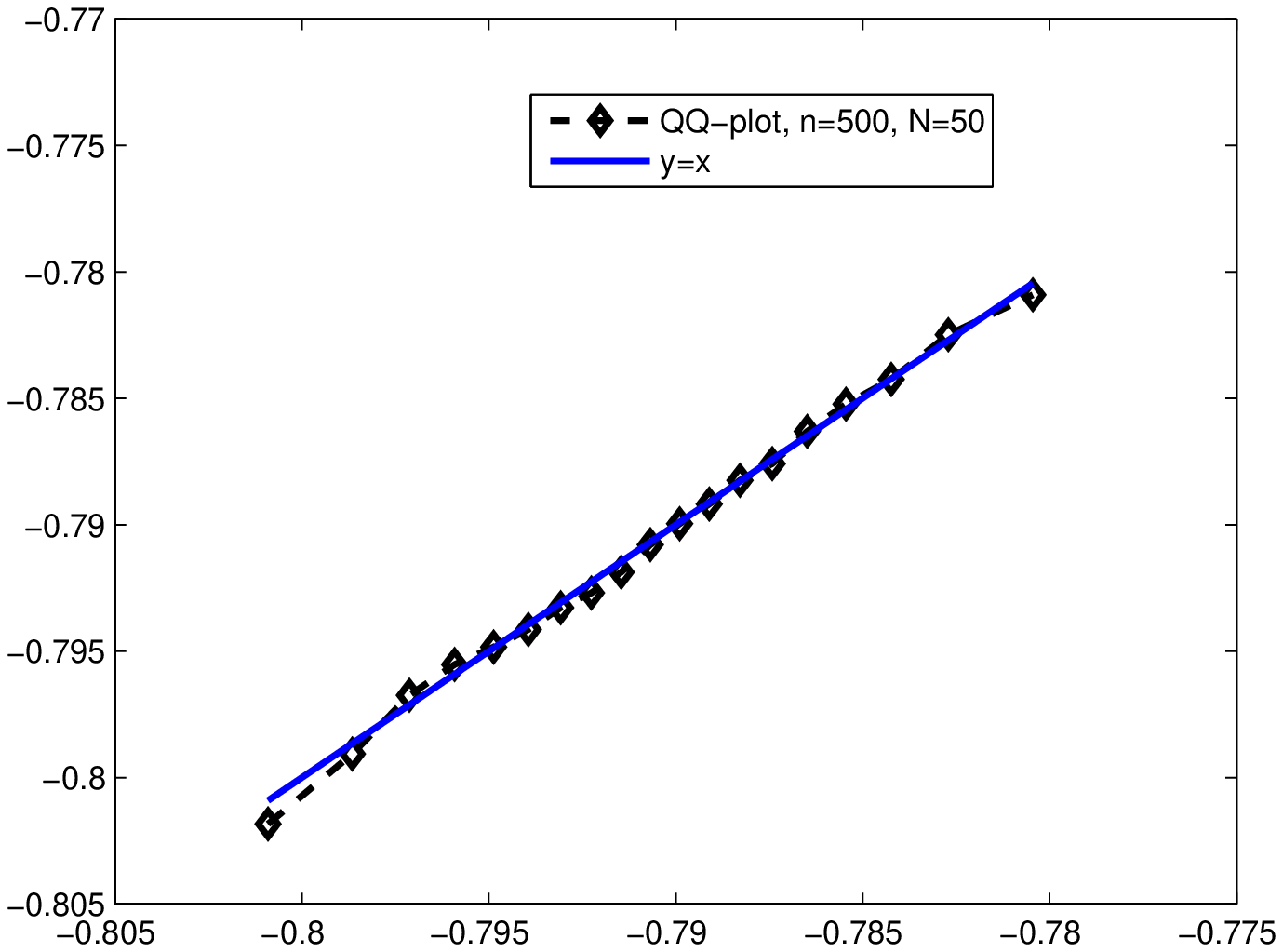}
&
\includegraphics[scale=0.5]{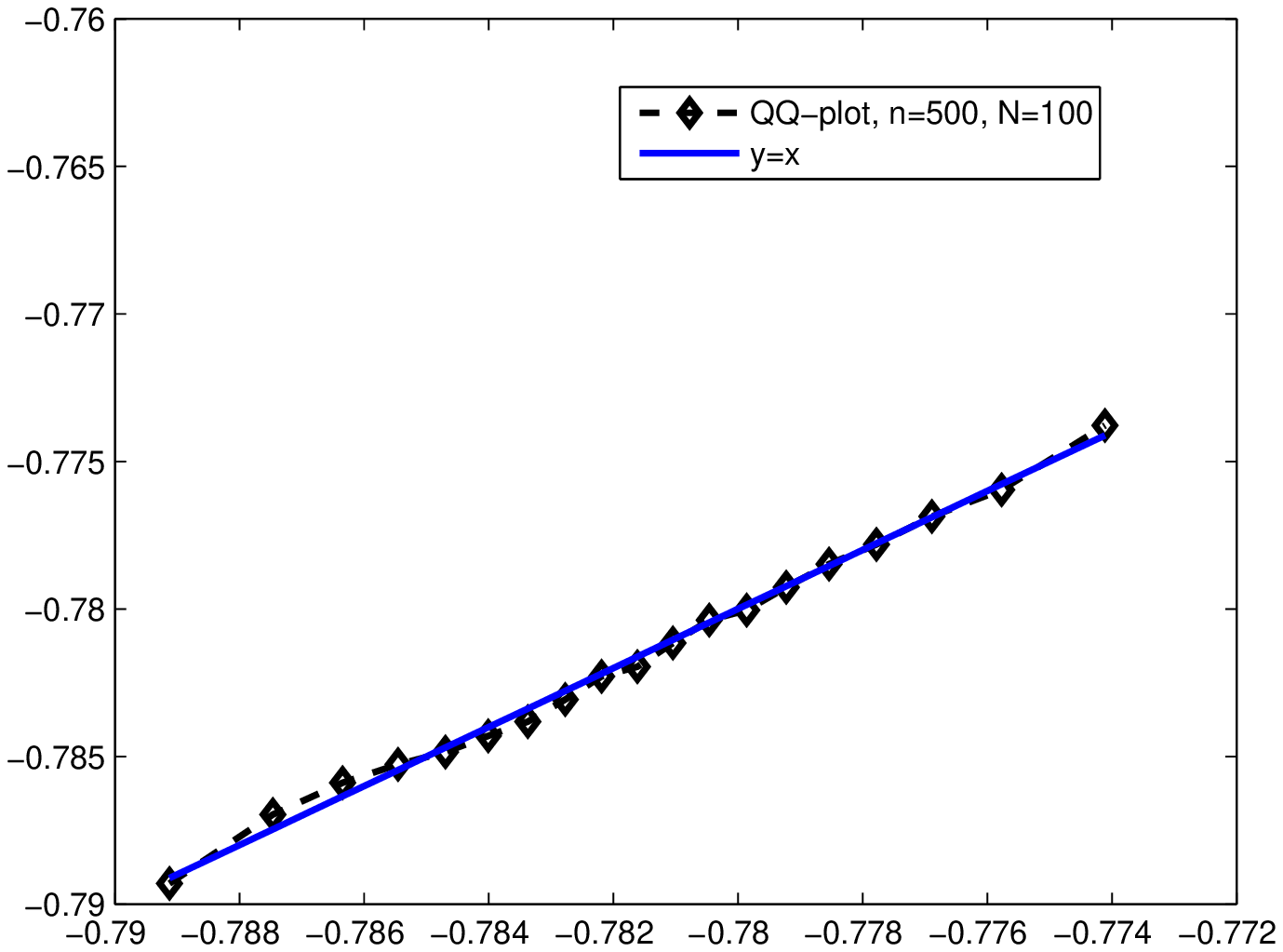}
\end{tabular}
\caption{QQ-plots of the empirical distribution of the SAA optimal value for problem \eqref{definstance3}
versus the normal distribution with parameters the empirical mean and standard deviation
of this empirical distribution for instances with $w_0=0.9, w_1=0.1, 1-\alpha = 0.1, \lambda_0=2$,
and various sample and problem sizes.}
\label{figureqq2}
\end{figure}

\begin{figure}
\begin{tabular}{ll}
\includegraphics[scale=0.5]{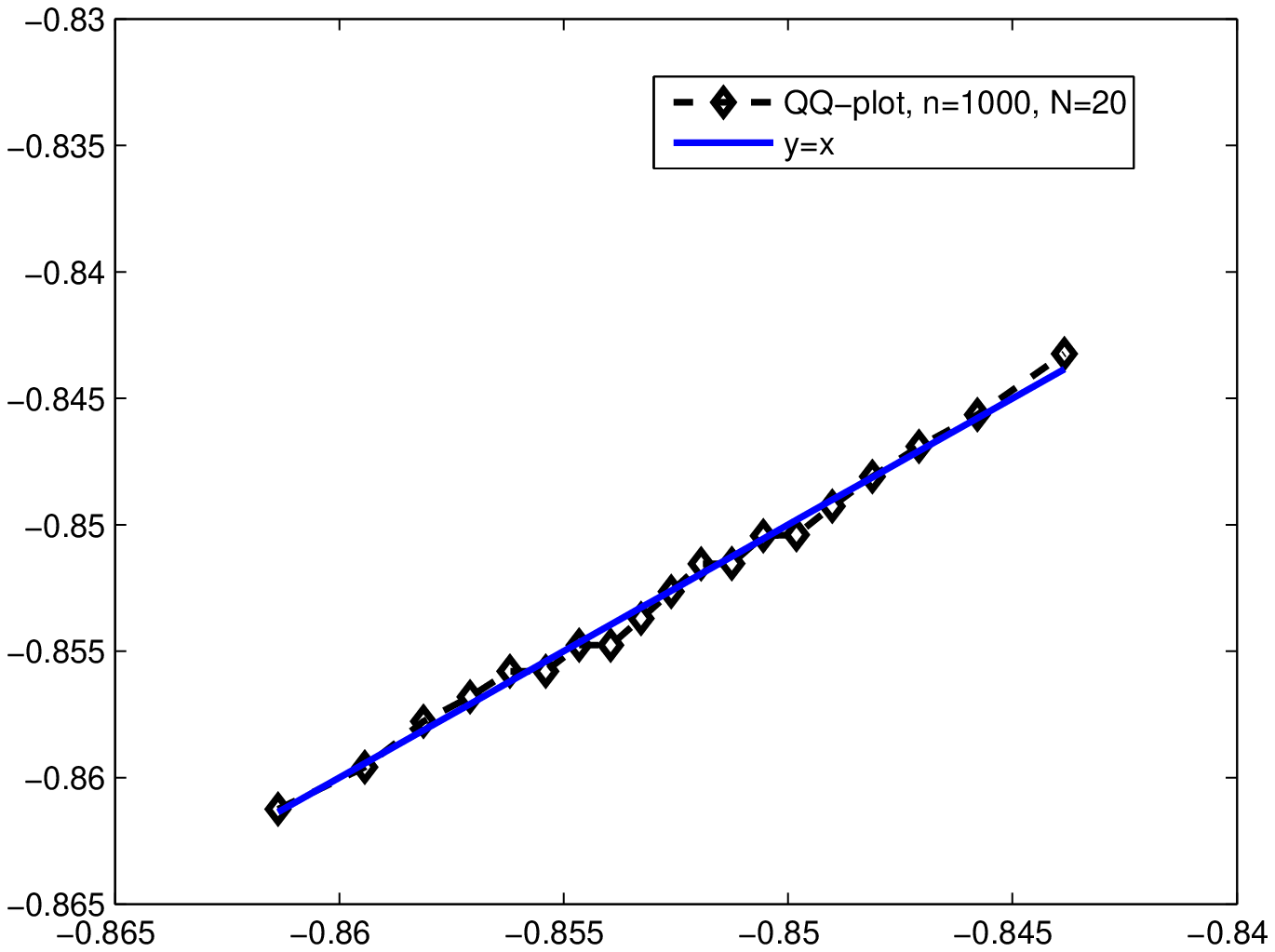}
&
\includegraphics[scale=0.5]{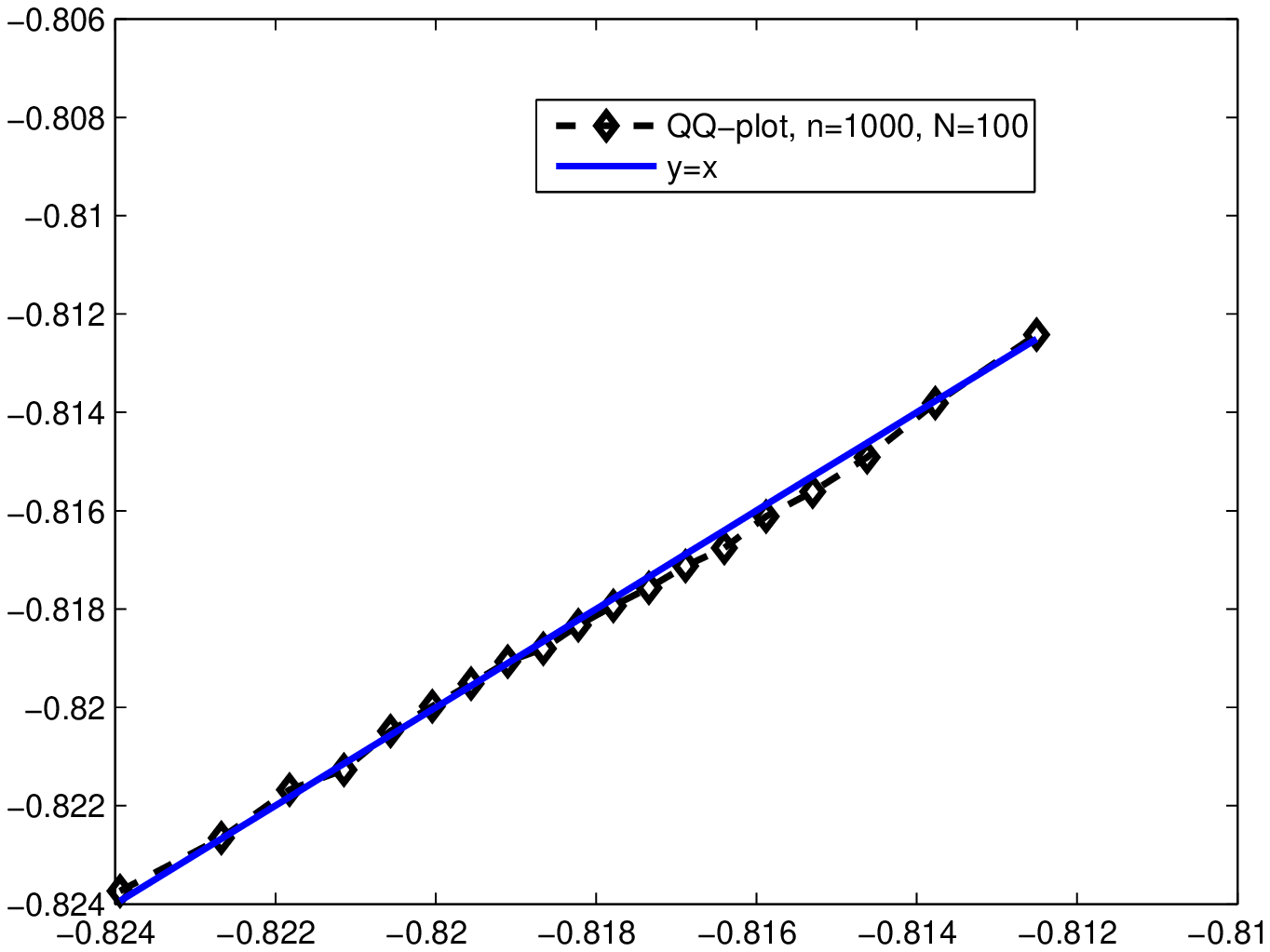}\\
\includegraphics[scale=0.5]{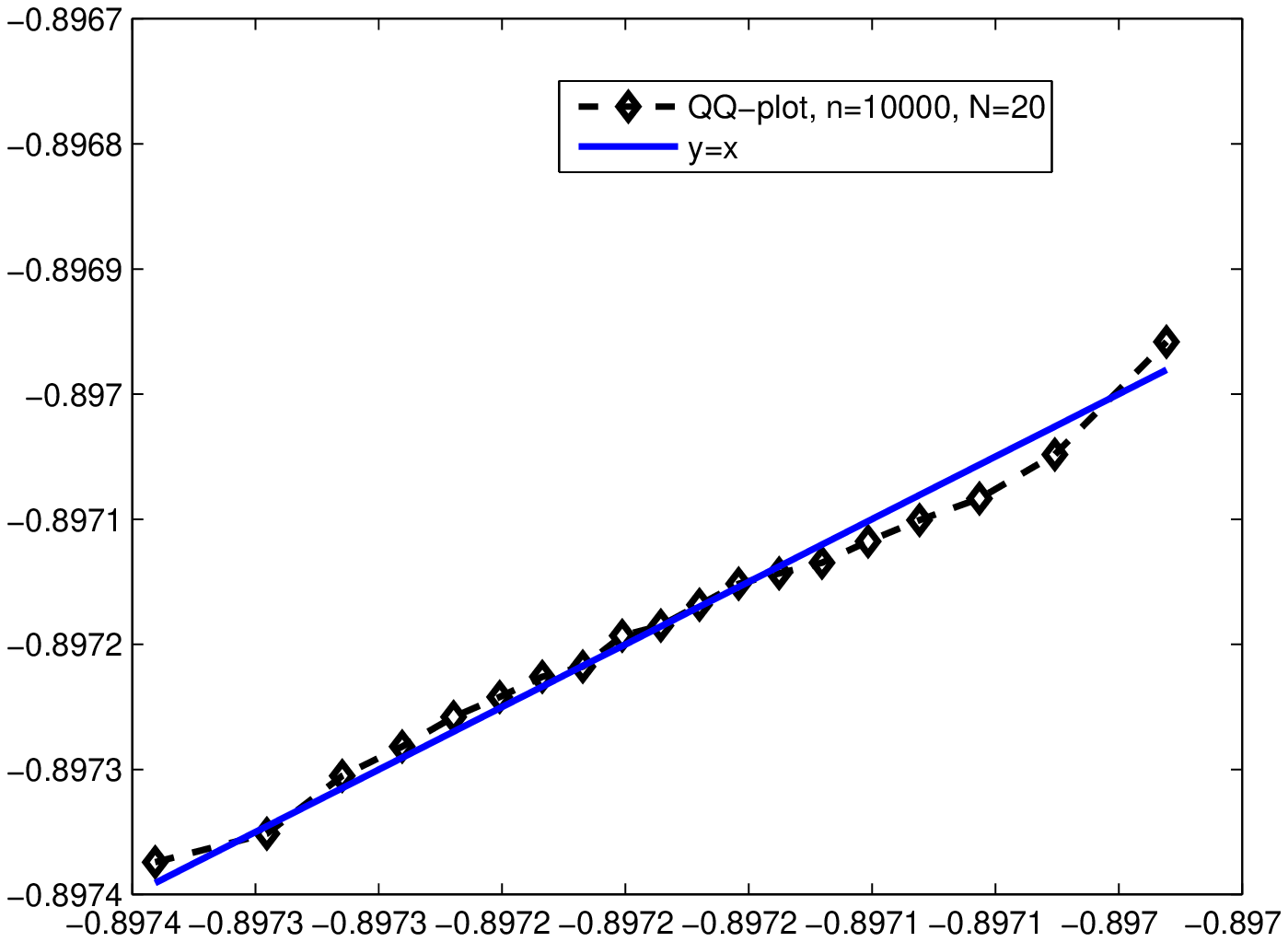}
&
\includegraphics[scale=0.5]{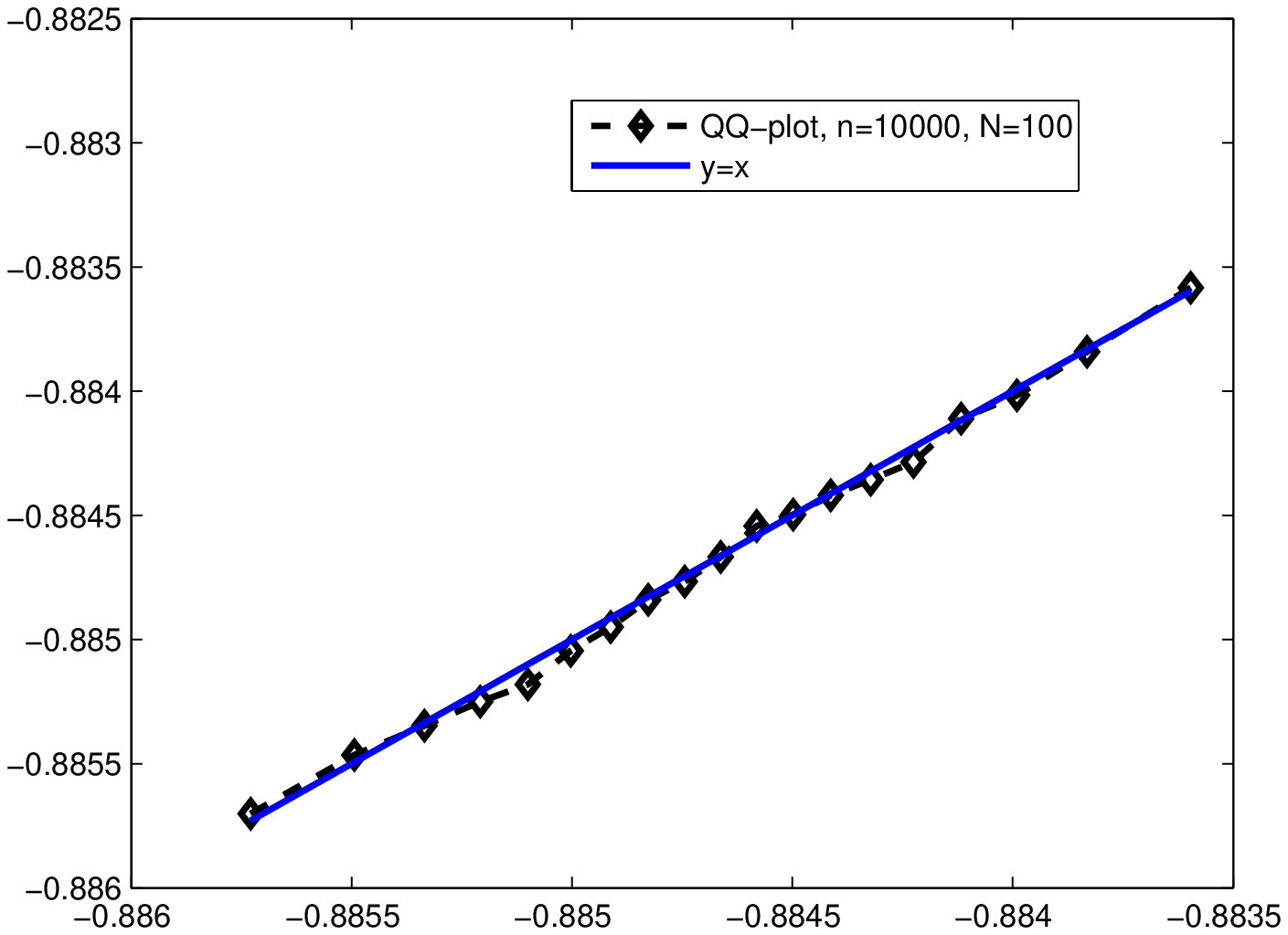}
\end{tabular}
\caption{QQ-plots of the empirical distribution of the SAA optimal value for problem \eqref{definstance3}
versus the normal distribution with parameters the empirical mean and standard deviation
of this empirical distribution for instances with $w_0=0.9, w_1=0.1, 1-\alpha = 0.1, \lambda_0=2$,
and various sample and problem sizes.}
\label{figureqq3}
\end{figure}

We now define in Table \ref{definstances} six instances $\mathcal{I}_1, \mathcal{I}_2, \mathcal{I}_3, \mathcal{I}_4, \mathcal{I}_5$,
and $\mathcal{I}_6$ of problem \eqref{definstance3}.

\begin{table}
\centering
\begin{tabular}{|c||c|c|c|c|}
\hline
{\tt{Instance}} &   $(w_0, w_1, 1-\alpha, \lambda_0)$  &  $c_0$ &  $n$  &  $(\mathbb{P}(\xi_i = 1))_i$  \\
\hline
$\mathcal{I}_1$ & $(0.9,0.1, 0.1, 2)$    &   $0$ &  $100$    &  $\Psi_1$   \\
\hline
$\mathcal{I}_2$ & $(0.9,0.1, 0.1, 2)$    &   $0$ &  $100$    &  $0.8 \Psi_1$   \\
\hline
$\mathcal{I}_3$ & $(0.9,0.1, 0.1, 2)$    &   $-3$ &  $100$    &  $0.8 \Psi_1$  \\
\hline
$\mathcal{I}_4$ & $(0.9,0.1, 0.1, 2)$    &   $0$ &  $500$    &  $\Psi_2$  \\
\hline
$\mathcal{I}_5$ & $(0.9,0.1, 0.1, 2)$    &   $0$ &  $500$    &  $0.8 \Psi_2$   \\
\hline
$\mathcal{I}_6$ & $(0.9,0.1, 0.1, 2)$    &   $-3$ &  $500$    &  $0.8 \Psi_2$\\
\hline
\end{tabular}
\caption{Definition of instances $\mathcal{I}_1$, $\mathcal{I}_2$, $\mathcal{I}_3$, $\mathcal{I}_4$, $\mathcal{I}_5$, and $\mathcal{I}_6$ of problem \eqref{definstance3}
($\Psi_1$ and $\Psi_2$ are vectors with entries drawn independently and randomly over $[0,1]$).}
\label{definstances}
\end{table}

We first compare the estimation of the optimal value of $\mathcal{I}_2$
using RSA and SAA.
For the RSA algorithm, we take $\|\cdot \| = \| \cdot \|_2 = \| \cdot  \|_*$ and (see \cite{guigues2016})
$L=\sqrt { \left( \frac{w_1 \alpha }{ 1 - \alpha } \right)^2   + n (w_0 + \frac{w_1}{1-\alpha})^2}+2\lambda_0$,
$M_1=2(w_0 + \frac{w_1}{1-\alpha})$,
$M_2=\sqrt{\left(\frac{w_1}{1-\alpha}\right)^2  + 4n \left( w_0 + \frac{w_1}{1-\alpha} \right)^2 }$.
The average approximate optimal value of instance $\mathcal{I}_2$ (averaging taking 100 samples of $\xi^N$)
using RSA and SAA is given in Table \ref{tablegn2} for various sample sizes $N$.
These values increase (resp. decrease) with the sample size for SAA (resp. RSA).
With SAA, the optimal value is already well approximated with small sample sizes while
large samples are needed to obtain a good approximation with RSA.
We also report in Table \ref{tabletest7}
the average values of the asymptotic and nonasymptotic confidence bounds (computed using 100 samples of $\xi^N$)
on the optimal values of instances
$\mathcal{I}_1$ and $\mathcal{I}_2$ and various sample sizes.\footnote{The nonasymptotic confidence interval is
$[{\tt{Low}}(\Theta_2, \Theta_3, N), {\tt{Up}}(\Theta_1, N)]$ with ${\tt{Low}}(\Theta_2, \Theta_3, N), {\tt{Up}}(\Theta_1, N)$ given by
\eqref{lowsmd1}, \eqref{upsmd1} and $\Theta_1=2\sqrt{\ln(2/\beta)}$, $\Theta_3=2\sqrt{\ln(4/\beta)}$,
and $\Theta_2$ satisfying $e^{1-\Theta_2^2} + e^{-\Theta_2^2/4}=\frac{\beta}{4}$.
The asymptotic confidence interval for \eqref{definstance3} is $[\hat \vv_N - \Phi^{-1}(1-\beta/2) \frac{\hat \nu_N }{\sqrt{N}} \hat \vv_N + \Phi^{-1}(1-\beta/2) \frac{\hat \nu_N }{\sqrt{N}}]$.}
Knowing that the optimal values of $\mathcal{I}_1$ and $\mathcal{I}_2$, estimated using SAA with a sample of size $10^6$, are
respectively $\vv_1=-0.6515$ and $\vv_2=-0.6791$, we observe that the asymptotic confidence interval is in mean much closer to the optimal value and of small width
while large samples are needed to obtain a nonasymptotic confidence interval of small width.
However,
the confidence bounds
on the optimal value obtained using RSA are almost independent on the problem size and as for
the one dimensional problem of the previous section the sample size $N=10^5$ provides confidence intervals
of small width and allows us to have small probabilities of type I and type II errors for nonasymptotic tests on the optimal value
of two instances of \eqref{definstance3} if their optimal values are sufficiently distant (see Lemmas \ref{tIItest1}, \ref{tIItest2}, and \ref{tIItest3}).
To check that and the superiority of the asymptotic tests for problems of moderate sizes ($n=100$ and $n=500$),
we compare the empirical probabilities of type II error of several tests of form \eqref{deftest1} with $K=2$ for which $H_1$ holds
and where $\vv_i$ is the optimal value of instance $\mathcal{I}_i$.

\begin{table}
\centering
\begin{tabular}{|c||c|c|c|c|c|c|}
\hline
 {\tt{Method}}    & $N=20$ & $N=50$ &$N=10^2$ & $N=10^3$ &  $N=10^4$ & $N=10^5$    \\
\hline
 {\tt{SAA}} & -0.7205 &-0.6965 & -0.6883& -0.6799& -0.6791 &-0.6791 \\
\hline
 {\tt{RSA}} &-0.4615 & -0.5274 & -0.5646  & -0.6389  & -0.6654  &  -0.6738  \\
\hline
\end{tabular}
\caption{Average approximate optimal value of instance $\mathcal{I}_2$ (computed using 100 samples of $\xi^N$)
using SAA and RSA for various sample sizes $N$.}
\label{tablegn2}
\end{table}

\begin{table}
\centering
\begin{tabular}{|c||c|c|c|c||c|c|c|c|}
\hline
 $N$    & {\tt{Low-As1}} & {\tt{Up-As1}} &  {\tt{Low-RSA1}} & {\tt{Up-RSA1}} & {\tt{Low-As2}} & {\tt{Up-As2}} &  {\tt{Low-RSA2}} & {\tt{Up-RSA2}}  \\
\hline
 $20$ & -0.7207  & -0.6666 & -95.7926  &  2.5227  & -0.7443 &  -0.6967 & -95.8354  &  2.4799 \\
 \hline
 $50$ & -0.6888 &  -0.6475 & -60.8057 &   1.3743 &  -0.7148 &  -0.6781 & -60.8472  &  1.3329\\
\hline
 $10^2$ & -0.6752  & -0.6444 & -43.1779   & 0.7900 &  -0.7019  & -0.6746 & -43.2171 &   0.7508\\
\hline
 $10^3$ &  -0.6573 &  -0.6474 & -14.0952 &  -0.1913  & -0.6843 &  -0.6755 & -14.1269 &  -0.2230\\
\hline
 $10^4$ &  -0.6532  & -0.6501 &  -4.9019 &  -0.5051  & -0.6805  & -0.6777 &  -4.9307 &  -0.5339\\
\hline
 $10^5$ & -0.6520   &-0.6510  & -1.9947 &  -0.6043 &  -0.6796 &  -0.6787  & -2.0226  & -0.6322 \\
\hline
\end{tabular}
\caption{Average values of the asymptotic and nonasymptotic confidence bounds (computed using 100 samples of $\xi^N$) for instances
$\mathcal{I}_1$ and $\mathcal{I}_2$ and various sample sizes.
For instance $\mathcal{I}_i$,
the average asymptotic confidence interval is [{\tt{Low-Asi}}, {\tt{Up-Asi}}] and
the average nonasymptotic confidence interval is [{\tt{Low-RSAi}}, {\tt{Up-RSAi}}].}
\label{tabletest7}
\end{table}

More precisely, the empirical  probabilities of type II error of asymptotic and nonasymptotic tests of form
\begin{equation}\label{test4}
H_0: \vv_i  = \vv_j  \;\;\mbox{against }H_1: \vv_i  \neq \vv_j,
\end{equation}
 are reported in Table \ref{tabletest8}
(for all these tests, we check that $H_1$ holds computing $\vv_i$ solving the SAA problem of instance $\mathcal{I}_i$
with a sample of $\xi$ of size $10^6$: $\vv_1=-0.6515, \vv_2=-0.6791, \vv_3=-3.6791, \vv_4=-0.7725, \vv_5=,-0.7868$,
and $\vv_6=-3.7868$).

\begin{table}
\centering
\begin{tabular}{|c|c|c|c|c|c|c|c|c|}
\cline{4-9}
 \multicolumn{1}{c}{}  &  \multicolumn{1}{c}{}   & \multicolumn{1}{c}{}  &\multicolumn{6}{|c|}{Sample size $N$}\\
 \hline
 $H_0$ & $H_1$    & Test type & $20$ &  $50$ &  $10^2$ & $10^3$ & $10^4$ & $10^5$ \\
\hline
$\vv_1 = \vv_2$   &   $\vv_1 \neq \vv_2$ & {\tt{Asymptotic}}  &   0.72   &   0.45   &  0.29  &  0   & 0  & 0\\
\hline
$\vv_1 = \vv_2$   &   $\vv_1 \neq \vv_2$  &{\tt{Nonasymptotic}}  &   1   &  1    &   1 &   1  & 1  &1\\
\hline
$\vv_1 = \vv_3$ & $\vv_1 \neq \vv_3$&  {\tt{Asymptotic}} &   0   &   0   &  0  & 0    & 0&0\\
\hline
$\vv_1 = \vv_3$   &  $\vv_1 \neq \vv_3$ & {\tt{Nonasymptotic}} &    1  &  1    & 1   &  1   & 1 & 0 \\
\hline
$\vv_4 = \vv_5$   &   $\vv_4 \neq \vv_5$ & {\tt{Asymptotic}}  &  0.33    &  0.36    &  0.21  &  0   &0 &0\\
\hline
$\vv_4 = \vv_5$   &   $\vv_4 \neq \vv_5$  &{\tt{Nonasymptotic}}  &   1   &  1    &   1 &   1  &1 &1\\
\hline
$\vv_4 = \vv_6$ & $\vv_4 \neq \vv_6$&  {\tt{Asymptotic}} &   0   & 0     &  0  &  0   & 0&0\\
\hline
$\vv_4 = \vv_6$   &  $\vv_4 \neq \vv_6$ & {\tt{Nonasymptotic}} &    1  &  1    & 1   &  1   & 1 &0\\
\hline
\end{tabular}
\caption{Empirical probabilities of  type II error for  tests of form  \eqref{test4}.}
\label{tabletest8}
\end{table}

Though it was observed in \cite{guigues2016}, \cite{ioudnemgui15} that for sample sizes that are not
much larger than the problem size
the coverage probability of
the asymptotic confidence interval is much lower than the coverage probability of the nonasymptotic confidence interval
and than the target coverage probability, the asymptotic confidence bounds are much closer to each other and much closer to the optimal
value than the nonasymptotic confidence bounds. This explains why the probability of type II error of the asymptotic test
is much less than the probability of type II error of the nonasymptotic test, even for small sample sizes and a smaller sample
is needed to always take the correct decision $H_1$ with the asymptotic test, i.e., to obtain a null probability of type II error.
Of course, in both cases,  for fixed $N$, the empirical probability of type II error  depends on the distance between $\vv_i$  and $\vv_j$.

Similar conclusions can be drawn from Table \ref{tabletest9} which
reports the empirical probability of type II error for various tests of form
\begin{equation}\label{test5}
H_0: \vv_i  \leq \vv_j   \;\;\mbox{against }H_1: \vv_j  < \vv_i.
\end{equation}
In particular, from these results, we see that we always take the correct decision $H_1$ with the asymptotic test
for sample sizes above $N=100$.

\begin{table}
\centering
\begin{tabular}{|c|c|c|c|c|c|c|c|c|}
\cline{4-9}
 \multicolumn{1}{c}{}  &  \multicolumn{1}{c}{}   & \multicolumn{1}{c}{}  &\multicolumn{6}{|c|}{Sample size $N$}\\
 \hline
 $H_0$ & $H_1$    & Test type & $20$ &  $50$  &  $10^2$ & $10^3$ & $10^4$ & $10^5$ \\
\hline
$\vv_1 \leq \vv_2$   &   $\vv_1 > \vv_2$ & {\tt{Asymptotic}}  &   0.54   &  0.38    & 0.16   &  0   & 0   &  0\\
\hline
$\vv_1 \leq \vv_2$   &   $\vv_1 > \vv_2$  &{\tt{Nonasymptotic}}  &   1   &  1    &   1 &   1  &1&1 \\
\hline
$\vv_1 \leq \vv_3$ & $\vv_1 > \vv_3$&  {\tt{Asymptotic}} & 0     & 0     &  0  &  0   & 0&0\\
\hline
$\vv_1 \leq \vv_3$   &  $\vv_1 > \vv_3$ & {\tt{Nonasymptotic}} &   1   &   1   & 1   &  1   & 1 &0\\
\hline
$\vv_4 \leq \vv_5$   &   $\vv_4 > \vv_5$ & {\tt{Asymptotic}}  &   0.29   &  0.26    & 0.15   & 0    &0 &0\\
\hline
$\vv_4 \leq \vv_5$   &   $\vv_4 > \vv_5$  &{\tt{Nonasymptotic}}  &  1    & 1     & 1   &  1   &1 &1\\
\hline
$\vv_4 \leq \vv_6$ & $\vv_4 > \vv_6$&  {\tt{Asymptotic}} &    0  &  0    &  0  &  0   & 0&0\\
\hline
$\vv_4 \leq \vv_6$   &  $\vv_4 > \vv_6$ & {\tt{Nonasymptotic}} &    1  &  1    & 1   &  1   & 1 &0\\
\hline
\end{tabular}
\caption{Empirical probabilities of  type II error for  tests of form  \eqref{test5}.}
\label{tabletest9}
\end{table}

\addcontentsline{toc}{section}{References}
\bibliography{Risk_Asym}

\if{

}\fi

\end{document}